\documentclass[a4paper,twoside,10pt]{article}
\usepackage[utf8]{inputenc}
\usepackage[T1]{fontenc}
\usepackage{textcomp}
\usepackage{abstract}
\usepackage[french]{babel}
\usepackage[autolanguage]{numprint}
\usepackage{hyperref}
\usepackage{verbatim}
\usepackage{fancyhdr}
\usepackage{url}
\usepackage{amsmath,amssymb,amsthm}
\usepackage{mathrsfs}
\usepackage[all,cmtip]{xy}
\usepackage{graphicx}
\usepackage{titlesec}
\titleclass{\subsubsubsection}{straight}[\subsection]

\usepackage{geometry}
\newcommand{\CBXQ}{\mathcal{C}_{Q}^{\operatorname{b}}(X)}

\newcommand\RR{\mathbf{R}}
\newcommand\CC{\mathbf{C}}
\newcommand\ZZ{\mathbf{Z}}
\newcommand\QQ{\mathbf{Q}}
\newcommand\NN{\mathbf{N}}
\newcommand\PP{\mathbf{P}}

\newcommand\cab{C_{a,b}}
\newcommand\dist{\operatorname{d}}
\newcommand\ep{e_{\mathfrak{p}}}

\newcommand\mukp{\mathcal{U}_D^+}
\newcommand\mukt{\tilde{\mu}_K}
\newcommand\muktp{\mathcal{U}_D^*}
\newcommand\VD{\varepsilon_D}

\newcommand\evbpd{E(\varepsilon,\eta,B)}
\newcommand\VDT{\varepsilon_D^*}
\newcommand{\ZZD}{\mathbf{Z}+\mathbf{Z}\sqrt{D}}
\DeclareMathOperator{\pgcd}{\operatorname{pgcd}}
\DeclareMathOperator{\ppcm}{\operatorname{ppcm}}
\DeclareMathOperator{\h}{H}
\newcommand\IM{\mathcal{I}_{|m|}}
\newcommand\IMM{\mathcal{I}_m^*}
\newcommand\IMP{\mathcal{I}_m^+}
\newcommand\OK{\mathcal{O}_K}
\newcommand\SMP{S_{D,m}^*}
\newcommand\SM{S_{D,m}}
\newcommand\evbm{E(\varepsilon,B)_m}

\newcommand\gal{\operatorname{Gal}}
\newcommand\nkq{N_{k/\QQ}}
\newcommand\be{\mathbf{e}}
\DeclareMathOperator{\aess}{\alpha_{\operatorname{ess}}}
\newcommand{\SABC}{S^*_{a,b,c}}
\newtheorem{theorem}{Théorème}[section]
\newtheorem{lemma}[theorem]{Lemme}
\newtheorem{proposition}[theorem]{Proposition}
\newtheorem{corollary}[theorem]{Corollaire}
\newtheorem{prop-def}[theorem]{Proposition-Définition}

\theoremstyle{definition}
\newtheorem{definition}[theorem]{Définition}
\newtheorem{notation}[theorem]{Notation}
\newtheorem{remark}[theorem]{Remarque}
\newtheorem{conjecture}[theorem]{Conjecture}
\newtheorem*{remark*}{Remarque}

\newcounter{subsubsubsection}[subsubsection]

\renewcommand\thesubsubsubsection{\thesubsubsection.\arabic{subsubsubsection}}

\titleformat{\subsubsubsection}
{\normalfont\normalsize\bfseries}{\thesubsubsubsection}{1em}{}
\titlespacing*{\subsubsubsection}
{0pt}{3.25ex plus 1ex minus .2ex}{1.5ex plus .2ex}

\makeatletter
\renewcommand\paragraph{\@startsection{paragraph}{5}{\z@}%
	{3.25ex \@plus1ex \@minus.2ex}%
	{-1em}%
	{\normalfont\normalsize\bfseries}}
\renewcommand\subparagraph{\@startsection{subparagraph}{6}{\parindent}
	{3.25ex \@plus1ex \@minus .2ex}%
	{-1em}%
	{\normalfont\normalsize\bfseries}}
\def\toclevel@subsubsubsection{4}
\def\toclevel@paragraph{5}
\def\toclevel@paragraph{6}
\def\l@subsubsubsection{\@dottedtocline{4}{7em}{4em}}
\def\l@paragraph{\@dottedtocline{5}{10em}{5em}}
\def\l@subparagraph{\@dottedtocline{6}{14em}{6em}}
\@addtoreset{subsubsubsection}{section}
\@addtoreset{subsubsubsection}{subsection}
\@addtoreset{paragraph}{subsubsubsection}
\makeatother

\title{\bf Approximation diophantienne et distribution locale sur une surface torique}
\author{\scshape{Zhizhong Huang}\\ 
\multicolumn{1}{p{.7\textwidth}}
{\centering{\textit{Institut Fourier, UMR 5582}\\\textit{Université Grenoble Alpes}\\ \textit{38610 Gières, France}\\ \texttt{zhizhong.huang@yahoo.com}}}}
\date{}
\begin{document}
	\numberwithin{equation}{section}
	\maketitle
	\setcounter{tocdepth}{1} 
	\renewcommand{\abstractname}{Résumé}
	\begin{abstract}
		Nous étudions dans ce texte l'approximation diophantienne et la distribution locale en un point rationnel sur une surface torique obtenue comme un éclatement de $\PP^1\times\PP^1$. Il s'avère qu'en dehors d'un fermé de Zariski les meilleures approximations s'obtiennent à l'aide d'une famille de courbes nodales. L'étude se ramène donc à la question de la distribution locale en un point quadratique sur la droite projective.
	\end{abstract}
\renewcommand{\abstractname}{Abstract}
\begin{abstract}
In this article we study Diophantine approximation and local distribution of a rational point on a toric surface obtained as a blow-up of $\mathbf{P}^1\times\mathbf{P}^1$. It turns out that outside a Zariski closed subset the best approximations are achieved through a family of nodal curves. Hence the investigation is reduced to the question of local distribution of a quadratic point on the projective line.
\end{abstract}

	\section{Introduction}
	\subsection{Motivation et heuristique}
	L'étude de la distribution des points rationnels sur les variétés algébriques est un sujet classique de la géométrie diophantienne.
	Beaucoup de variétés dont le fibré anticanonique est gros vérifient le principe de Batyrev-Manin \cite{B-M} avec la constante de Peyre \cite{Peyre1}, y compris les variétés toriques lisses projectives \cite{B-T2}, voir \cite{Browningbook} pour des exemples.
	Plus précisément, soit $X$ une \og bonne\fg\ variété (dans le sens de \cite[\S 3]{Peyre2}) définie sur un corps de nombres $k$ munie d'une hauteur de Weil $H$ associée au fibré anticanonique $\omega_{X}^{-1}$. On dit que $X$ vérifie \emph{le principe de Batyrev-Manin} s'il existe un ouvert dense $U$ de $X$ tel que
	$$\sharp U_{H\leqslant B}=\sharp\{x\in U(k):H(x)\leqslant B\}\sim C(X) B (\log B)^{\operatorname{rg}(\operatorname{Pic}(X))-1},$$
	où la constante $C(X)$ a reçu une interprétation géométrique (cf. \cite{Peyre1},\cite{Peyre2}).
	La raison pour laquelle on ne considère pas tous les points de la variété est que souvent il existe des sous-variétés, que l'on appelle ici \textit{globalement accumulatrices}, dont la croissance du nombre de points rationnels domine celle d'un ouvert dense. On peut définir certaines constantes (par exemple, la constante $\beta$ dans \cite{B-M}) pour détecter de telles variétés. Soit $U$ l'ouvert privé de toutes les sous-variétés globalement accumulatrices (supposons qu'un tel ouvert existe). Une extension naturelle de ce problème est de considérer la convergence au sens vague de la suite de mesures de probabilité 
	\begin{equation}\label{eq:diracmeasure2}
	\mu_{U_{H\leqslant B}}=\frac{1}{\sharp U_{H\leqslant B}}\sum_{x\in U(k),H(x)\leqslant B}\delta_x,
	\end{equation}
	où $\delta_x$ est la mesure de Dirac en $x$. Si cette suite $(\mu_{U_{H\leqslant B}})_B$ converge, elle reflète un phénomène d'équidistribution. Peyre \cite{Peyre1} a conjecturé une mesure limite, appelée mesure de Tamagawa.  
	
	Ce texte est concerné par le problème de la distribution \textit{locale} des points rationnels autour d'un point rationnel fixé. Le terme \og local\fg\ est au sens de l'approximation diophantienne systématiquement formulé en premier par D. McKinnon et M. Roth (\cite{Mc}, \cite{M-R}). Une différence majeure par rapport aux problèmes globaux est que même s'il n'existe pas de sous-variétés globalement accumulatrices, il peut y avoir des sous-variétés qui sont \textit{localement accumulatrices}. Comme dans le problème global, on peut aussi définir certaines constantes pour les caractériser. L'une de ces constantes, notée $\alpha(Q,Y)$ (Définition \ref{pd:appconstant}), appelée \textit{la constante d'approximation} pour la sous-variété $Y$ de $X$ et due à McKinnon et Roth, caractérise \emph{de meilleures approximations} pour un point rationnel $Q$ sur la variété $Y$. On dira qu'une sous-variété $Y$ donne de meilleures approximations si $\alpha(Q,Y)=\alpha(Q,X)$. 
	La deuxième constante, appelée \textit{la constante essentielle} $\aess(Q)$ (Définition \ref{df:essential}) et due premièrement à Pagelot, décrit des \emph{approximations génériques}. Elle est définie comme le supremum des $\alpha(Q,U)$ pour $U$ parcourant tous les ouverts denses de $X$.
	Au niveau de la distribution plus fine, S. Pagelot \cite{Pagelot} a étudié en premier \textit{la distribution locale} des points rationnels sur certaines variétés algébriques via une opération appelée \og zoom\fg. Le but est de décrire le comportement asymptotique des points rationnels proches d'un point rationnel fixé par un dénombrement fin.
	Cette opération nous permet d'obtenir plus d'informations sur l'accumulation locale des points rationnels que les constantes d'approximation car il existe des sous-variétés ayant la même valeur de constante d'approximation, sur lesquelles le nombre d'approximants qui se trouvent étant d'ordre de grandeur différente. On propose la notion \emph{localement faiblement accumulateur} (Définition \ref{df:weaklocal}) basée sur ce fait pour les caractériser.
	
	Avant d'esquisser la formulation de la distribution locale afin pour énoncer notre théorème principal, nous considérons d'abord un exemple particulier $\PP^2_\QQ$. On fixe le point $Q=[0:0:1]$ et la hauteur de Weil associée au fibré $\mathcal{O}(1)$ comme $$H([x:y:z])=\max(|x|,|y|,|z|),\quad x,y,z\in\ZZ,\quad \pgcd(x,y,z)=1.$$
	On s'intéresse à l'ensemble des points rationnels de hauteur bornée $$(\PP^2_\QQ)_{H\leqslant B}=\{P=[x:y:z]\in\PP^2(\QQ):H(P)\leqslant B\}.$$
	Un calcul nous dit que les points le plus proches de $Q$ sont à une distance d'ordre $B^{-1}$.
	Nous invitons les lecteurs à consulter \cite[FIGURE 1]{Peyre2} pour un dessin de la répartition des points de hauteur bornée autour de $Q$, où nous observons que les points se répartissent apparemment de préférence sur les droites rationnelles passant par $Q$. Ce qui se passe sur la variété  de produit $\PP^1\times\PP^1$ (\cite[FIGURE 3]{Peyre2}), munie d'une hauteur de Weil associée à $\mathcal{O}(1,1)$, est différent, puisqu'il s'avère que, la diminution de distance au point $Q^\prime=[1:1]\times[1:1]$ est d'ordre $B^{-1}$ sur les droites horizontale et verticale mais elle est d'ordre $B^{-2}$ ailleurs. En effet ces deux droites ont la constante d'approximation égale à $\alpha(Q^\prime,\PP^1\times\PP^1)$ et sont localement accumulatrices.
	L'idée est qu'après une manipulation de zoom avec un facteur bien choisi, on peut définir une suite de mesures dénombrant le nombre et décrivant le lieu des points rationnels proches de $Q$.
	
	Soit $X$ une \og bonne\fg\ variété (\cite{Peyre2} op. cit.) munie d'une hauteur de Weil $H$. On fixe un point lisse $Q\in X(\bar{k})$ défini sur $\RR$ à approcher et un difféomorphisme local $\rho$ de la variété réelle $X(\RR)$ en $Q$ sur le plan tangent $(T_Q X)_\RR$ qui envoie $Q$ sur l'origine. Pour $U$ un ouvert de Zariski, $r>0$ et $B\gg 1$, on définit $(\delta_{U,Q,B,r})_B$ la suite de mesures de Dirac par la formule suivante:
	\begin{equation}\label{eq:deltameasure}
	\delta_{U,Q,B,r}=\sum_{x\in U(k);H(x)\leqslant B}\delta_{B^{\frac{1}{r}}\rho(x)}.
	\end{equation}
	L'exposant $r$ est appelé \emph{facteur de zoom}. En général on ne prend que les $r\geqslant \aess(Q)$ (Proposition \ref{po:weakzoom}). 
	On fixe une distance euclidienne $d$ sur $(T_Q X)_\RR$ (qui est équivalente localement à toute distance projective définie sur $X$).
	On note $\mathbb{B}(0,\varepsilon)$ la boule centrée en l'origine de rayon $\varepsilon>0$ et $\chi(\varepsilon)$ la fonction caractéristique associée. S'il existe $\gamma\geqslant0,\beta\geqslant0$ tels que
	\begin{equation}\label{eq:growthorder}
	B^\gamma (\log B)^\beta \ll_\varepsilon \int \chi(\varepsilon) \operatorname{d} \delta_{U,Q,B,r} \ll_\varepsilon B^\gamma (\log B)^\beta,
	\end{equation}
	pour une infinité de $B\to\infty$,
	ils décrivent l'ordre de grandeur d'accumulation des points rationnels autour de $Q$. 
	On renormalise en considérant la suite de mesures 
	\begin{equation}\label{eq:normalisedmeasures}
	\left(\frac{1}{B^\gamma (\log B)^\beta} \delta_{U,Q,B,r}\right)_B.
	\end{equation}
	Si elles convergent (au sens vague) vers une mesure non-nulle, on l'appellera \textit{mesure limite}. En regardant la densité de la mesure limite on peut obtenir une caractérisation de la \og concentration locale\fg\ des points rationnels. 
	 Le phénomène général serait que quand $r>\aess(Q)$, (le zoom est \textit{sous-critique}, Définition \ref{def:criticalzoom}), $\int\chi(\varepsilon) \operatorname{d} \delta_{U,Q,B,r}$ aurait l'ordre de grandeur attendu et la répartition des points rationnels serait plus uniforme. Le cas où $r=\aess$ (le zoom est \textit{critique}) est particulièrement intéressant. S. Pagelot a étudié dans \cite{Pagelot} la distribution locale d'un $\QQ$-point sur la droite projective et sur des surfaces de del Pezzo (toriques) de degré $\geqslant 7$. En particulier pour $\PP^2_\QQ$ il a affirmé (pour le zoom critique) l'existence de la mesure limite attendue de masse concentrée sur les droites (ceci est redémontré dans \cite{Huang1}). Alors que pour $\PP^1\times\PP^1$ la mesure limite du zoom critique  existe uniquement en ayant retiré les deux droites particulières. Il est naturel de se poser la même question pour les points algébriques et pour les variétés de dimension supérieure. L'auteur a étudié dans \cite{Huang1} ce problème pour une surface de del Pezzo torique de degré $6$ et il a obtenu pour le cas $r=\aess(Q)$ une mesure limite.
	 
	 \subsection{Énoncés des résultats}
	 Dans ce texte on se placera sur une surface torique définie sur $\QQ$ que l'on notera toujours $Y_4$. Elle est obtenue en éclatant les $4$ points invariants de $\PP^1\times\PP^1$, ce qui correspond aux éventails en Figure \ref{fig:y4}.
	 \begin{figure}[h]
	 	\centering
	 	\includegraphics[scale=0.8]{Y4P1.mps}
	 	\caption{L'éclatement $Y_4\to\PP^1\times \PP^1$}
	 	\label{fig:y4}
	 \end{figure}
	 
	 On choisit le relevé du point $Q=[1:1]\times[1:1]$ de l'orbite ouverte que l'on va approcher. Comme conjecturé par D. McKinnon (Conjecture \ref{conj:mckinnon} \textit{infra}), les courbes rationnelles donnent des meilleurs approximants. En effet, nous démontrons qu'il existe $4$ courbes rationnelles $Z_i $ $(1\leqslant i\leqslant 4)$ passant par $Q$ de degré anticanonique minimal qui donnent des meilleures approximations. Le phénomène nouveau est qu'en dehors des $Z_i$, les approximants optimaux sont peu nombreux et se situent sur une famille de courbes nodales en $Q$ de degré bas dont la réunion est dense dans $Y_4$ et les tangentes au point $Q$ ont des pentes réelles et irrationnelles. Autrement dit, les points au-dessus de la désingularisation de la courbe en $Q$ ne sont pas définis sur $\QQ$ (en fait ils sont définis sur une extension quadratique réelle de $\QQ$).
	 
    \begin{notation}\label{notationcqbx}
    On note $\CBXQ$ l'espace vectoriel des fonctions continues de support compact définie sur $(T_Q X)_\RR$ à valeurs réelles. Pour $f\in\CBXQ$, on note désormais $\delta_{U,Q,B,r}(f)$ pour $\int f \operatorname{d} \delta_{U,Q,B,r}$.
\end{notation} 
    On énonce le théorème principal du texte.
    \begin{theorem}[cf. Théorèmes \ref{th:approximationessential}, \ref{th:finitenessofcriticzoom}, \ref{th:lowerbound} \textit{infra}]\label{th:maintheorem}
    		On a $\alpha(Q,Y_4)=\aess(Q)=2$. Soit $U=Y_4\setminus \cup_{i=1}^4Z_i$. Soit $r\geqslant 2$. Alors
    	\begin{enumerate}
    		\item Si $r=2$, on a que pour toute $f\in \CBXQ$, $$\delta_{U,Q,B,r}(f)=O_{f}(1).$$ 
    		Il n'existe pas de mesure limite dans ce cas.
    		\item Si $2<r<\frac{144}{55}$, alors il existe une constante $\eta\in\mathopen]0,1\mathclose[$ telle que pour toute positive $f\in\CBXQ$,
\begin{equation}\label{eq:lowerboundbegin}
    		\delta_{U,Q,B,r}(f)\gg \left(\int f \operatorname{d}\delta_r\right) B^{(1+\eta)(\frac{1}{2}-\frac{1}{r})}(\log B)^3
\end{equation}
    		où $\delta_r$ est une mesure absolument continue par rapport à la mesure de Lebesgue planaire.
    	\end{enumerate}
    \end{theorem}
    En combinant avec le résultat de Pagelot (Théorème \ref{th:pagelot}), on trouve que $\delta_{U,Q,B,2}(f)=o(\delta_{Z_i,Q,B,2}(f))$ pour toute fonction positive $f$. Ceci implique le nombre d'approximants sur les $Z_i$ domine celui dans le complémentaire bien que leur valeurs de constante d'approximation coïncident. La variété $\cup_{i=1}^4 Z_i$ est donc localement faiblement accumulatrice. 
    Le deuxième énoncé du Théorème \ref{th:maintheorem} assure que, quand on élargit légèrement le facteur de zoom $r$, l'ordre de croissance de la distribution locale est minoré par une puissance positive de $B$ avec une mesure uniforme $\delta_r$. 
    
    Notre stratégie s'appui sur le paramétrage par des courbes nodales mentionnées précédemment, ce qui est différente de celle pour traiter l'exemple de la surface del Pezzo torique de degré $6$ étudiée dans \cite{Huang1}, où pour le zoom critique, la mesure limite existe, obtenue en utilisant le paramétrage par des droites. 
    Quand on restreint aux courbes nodales, l'approximation du point $Q$ revient à l'approximation d'un point \textit{algébrique quadratique}, ce qui nous amène à étudier plus soigneusement la distribution locale pour les points algébriques sur la droite projective. Notons qu'en dimension $1$, il n'y a pas de sous-variétés localement accumulatrices. Donc la constante d'approximation vaut toujours la constante essentielle.
    	\begin{theorem}[cf. Théorèmes \ref{th:criticzoomforquadratics}, \ref{th:localdistributionofrealnumbers} \textit{infra}]\label{th:secondtheorem}
    	On munit le fibré $\mathcal{O}(1)$ sur $\PP^1$ d'une hauteur de Weil (cf. \eqref{eq:absoluteWeil}). 
    	Soit $\theta$ un nombre algébrique irrationnel et réel. On identifie $\theta$ à un point rationnel $[\theta:1]$ de $\PP^1$.
    	Alors $\alpha(\theta,\PP^1)=\frac{1}{2}$.
    	\begin{itemize}
    		\item Si $r=\frac{1}{2}$ et $\theta$ est un nombre quadratique, alors pour toute fonction $f\in\mathcal{C}_\theta^{\operatorname{b}}(\PP^1)$,
    		$$\delta_{\PP^1,\theta,B,\frac{1}{2}}(f)=O_{\theta,f }(1)$$ et il existe certaines fonctions $g\in\mathcal{C}_\theta^{\operatorname{b}}(\PP^1)$ telles que $$\liminf_B \delta_{\PP^1,\theta,B,\frac{1}{2}}(g)<\limsup_B\delta_{\PP^1,\theta,B,\frac{1}{2}}(g).$$ 
    	    En d'autres termes, il n'existe pas de mesure limite pour le zoom critique.
    			\item Si $r>\frac{1}{2}$. Alors pour toute $f\in\mathcal{C}_\theta^{\operatorname{b}}(\PP^1)$, $$\delta_{\PP^1,\theta,B,r}(f)= B^{2-\frac{1}{r}}\left(\frac{6}{\pi^2 \sup(1,|\theta|)^2}\int f(x)\operatorname{d}x+o_{\theta,f}(1)\right).$$ 
    			En particulier, la mesure limite est proportionnelle à la mesure de Lebesgue.
    	\end{itemize}
    \end{theorem}
    La constante d'approximation de $\theta$ est déduite des théorèmes classiques de Dirichlet et Roth (ou de l'inégalité de Liouville). Le premier énoncé implique que des approximants contribuant au zoom critique sont peu nombreux, de sorte qu'ils ne donnent pas de comportement uniforme. 
    Alors que la formule asymptotique pour $r>\frac{1}{2}$ dit qu'il existe beaucoup d'approximants réalisant le facteur de zoom proche de $\frac{1}{2}$ et ils se répartissent de façon assez uniforme. 
    Nous comparons ce résultat avec ceux antérieurs (Remarques \ref{rmk:Lang1} et \ref{rmk:Lang2}) s'agissant, d'après S. Lang \cite{Lang2}, \og Asymptotic Approximations\fg{}.
    
    \subsection{Méthodes, résultats auxiliaires et structure du texte}
    Au cours de la démonstration du Théorème \ref{th:maintheorem}, un certain nombre de résultats portant d'intérêt indépendant sont également achevés et méritent d'explications que nous fournissons ici.
    
    Tout d'abord, dans la Section \ref{se:approximationconstants} on introduit la notion de constante d'approximation et l'on précise l'opération de zoom. 
    
     Pour les nombres quadratiques réels, le zoom critique est directement lié aux équations de Pell-Fermat de la forme
     \begin{equation}\label{eq:diophantineequation}
     ax^2-by^2=c, \quad a,b\in\NN^*,c\in\ZZ,\quad \pgcd(a,b)=1,
     \end{equation}
     dont la structure des solutions constituent le contenu de la Section \ref{se:theoryalgebraic}.
     On note $\SABC$ l'ensemble des solutions $(x,y)\in\ZZ^2$ de \eqref{eq:diophantineequation}. 
     Trouver des points dans le zoom équivaut plus ou moins à trouver une infinité de solutions \textit{primitives}, i.e. $(x,y)\in \SABC$ telle que $\pgcd(x,y)=1$.
     S'appuyant sur la théorie algébrique des nombres pour le corps quadratique réel $\QQ(\sqrt{ab})$,  l'on démontre que l'ensemble $\SABC$ se décompose en des orbites disjointes par l'action de $S^*_{a,b,a}$, qui est un sous-groupe du groupe d'unités engendré par une puissance de l'unité fondamentale. En particulier, cela nous permet de construire une suite de solutions primitives de l'équation \eqref{eq:diophantineequation} à partir de n'importe quelle solution primitive fixée. Les résultats principaux sont la Proposition \ref{po:pgcddonotchange} et son Corollaire \ref{rmk:inftysol}, démontrés à l'aide d'un résultat classique de Dirichlet \cite{Dirichlet}. Ils étant explicites et effectifs, la constante implicite dans le premier énoncé du Théorème \ref{th:secondtheorem} est alors explicitement calculable. 
     
     Fixons $\theta$ comme dans le théorème \ref{th:secondtheorem}. Alors le calcul de la limite de  $(\delta_{\PP^1,\theta,B,r}(\chi(\varepsilon)))_B$ revient au comptage des ensembles qui consistent en les $[u:v]\in\PP^1(\QQ)$ vérifiant les conditions suivantes
     \begin{equation}\label{eq:Eebr2}
     u,v\in\ZZ,v\neq 0,\quad \pgcd(u,v)=1,\quad\max(|u|,|v|)\leqslant B,
     \end{equation}
     \begin{equation}\label{eq:Eebr1}
     \left|\frac{u}{v}-\theta\right|\leqslant \varepsilon B^{-\frac{1}{r}}.
     \end{equation}
     Le cas où $\theta\in\PP^1(\QQ)$ et $r=1$ est dû à Pagelot \cite{Pagelot}. Nous donnons une démonstration au appendice \ref{se:app1} pour la complétude. Pour $\theta$ irrationnel,
     la difficulté majeure pour le comptage est que 
     les points que l'on veut dénombrer sont des points entiers primitifs dans une région très fine. La formule empirique (le terme principal espéré est le volume et le terme d'erreur est d'ordre de grandeur majoré par la longueur du bord) n'est plus applicable pour $r$ petit. Le zoom critique ($r=\frac{1}{2}$ Théorème \ref{th:criticzoomforquadratics}) est traité dans la Section \ref{se:criticalquadratic} en utilisant les faits établis dans la Section \ref{se:theoryalgebraic}. 
     Pour les zoom sous-critiques (i.e. $r>\frac{1}{2}$), nous allons démontrer le deuxième énoncé du Théorème \ref{th:secondtheorem} pour tous les nombres réels ayant la mesure d'irrationalité $2$ (Définition \ref{def:irrationality}) en adaptant la théorie d'équirépartition modulo $1$, avec un contrôle des termes d'erreur à l'aide de l'inégalité de Koksma-Denjoy (Théorème \ref{th:koksma-denjoy}), en dépit de l'ineffectivité (du terme d'erreur) à cause de celle du théorème de Roth. Notre méthode s'applique avec succès aux cas où l'on imposent dans \eqref{eq:Eebr2} que $(u,v)$ soit sur un réseau fixé. Le terme principal fait apparaître le déterminant du réseau. Voir les détails dans la Section \ref{se:gerelattice}.

    L'étude de la distribution locale de la surface torique $Y_4$ (à tout niveau) est faite dans la dernière Section \ref{se:toricsurface}. La démonstration du Théorème \ref{th:maintheorem} utilise le paramétrage par des courbes nodales, puisqu'elles réalisent la constante essentielle $\aess(Q)$. La première partie se déduit de celle du Théorème \ref{th:secondtheorem} en remarquant que dans un voisinage fixé du point $Q$, il n'y a qu'un nombre fini de courbes nodales qui interviennent pour le zoom critique ($r=2$, Théorème \ref{th:finitenessofcriticzoom}). Pour le zoom sous-critique ($r>2$), chacune de ces courbes donne des points d'ordre de grandeur égale à une puissance de $B$, sur lesquelles on trouve une mesure limite de dimension $1$ uniforme au sens du deuxième énoncé du Théorème \ref{th:secondtheorem} (cf. Proposition \ref{po:zoompointsoncab} \textit{infra}). Afin d'intégrer toutes ces contributions, la prochaine étape est de les sommer. Plus concrètement nous avons besoin d'estimer la somme (voir les fonctions arithmétiques $\Psi$ \eqref{eq:arithmeticfunctionkey} et $\Phi$ \eqref{eq:arithmeticfunctionPhi})
\begin{equation}\label{eq:sumquestion}
    \sum_{\substack{(x_1,x_2)\in X\mathcal{R}\\\pgcd(x_1,x_2)=1}}\frac{\Phi(x_1)\Phi(x_2)\Psi(x_2-x_1)}{x_2x_1^\frac{1}{2}},
\end{equation}
    où $\mathcal{R}\subset ]0,1[^2\subset \RR^2$ est une région. La fonction $\Psi$ ressemblant à $\tau$ \eqref{eq:thefunctiontau}, nous bornons cette somme par 
    \begin{equation}\label{eq:sumbound}
    \sum_{\substack{(x_1,x_2)\in X\mathcal{R}\\\pgcd(x_1,x_2)=1}}\frac{\Psi(x_1)\Psi(x_2)\Psi(x_2-x_1)}{x_2x_1^\frac{1}{2}},
    \end{equation}
    ce qui nous permet d'utiliser la technique développée par R. de la Bretèche et T. D. Browning dans \cite{B-B1}, \cite{B-B2}, \cite{Browning} pour traiter l'ordre moyen des diviseurs pour des forme binaires, dont une présentation se trouve au appendice \ref{se:app2}, qui nous entraîne une formule asymptotique pour le dénombrement sur les paramètres qui sont petits pour des courbes nodales, ainsi que le deuxième énoncé du Théorème \ref{th:maintheorem}.
    Avec les techniques utilisées dans ce texte, on n'est pas capable de trouver une formule asymptotique pour le zoom sous-critique de la surface $Y_4$ (c'est-à-dire ``$\gg$'' remplacé par ``$\sim$''). Une raison est le passage de \eqref{eq:sumquestion} à \eqref{eq:sumbound}. L'existence d'une formule asymptotique pour \eqref{eq:sumquestion} pourrait signaler une \og meilleure\fg\ forme de la mesure $\delta_r$. Remarquons aussi que d'après l'heuristique naïve du principe de Batyrev-Manin (cf. \eqref{eq:prediction}), on souhaiterait que dans \eqref{eq:growthorder} la puissance de $B$ soit $2-\frac{1}{r}$. C'est-à-dire que l'on voudrait pourvoir prendre $\eta=1$. Nous démontrons la minoration \eqref{eq:lowerboundbegin} uniquement pour tous les $\eta<\frac{1}{35}$ (cf. Théorème \ref{th:lowerbound}).
    
    Dans la Section \ref{se:thinset}, nous proposons une façon d'interpréter les résultats pour $Y_4$ en voyant $Y_4(\QQ)$ comme un sous-ensemble de $(\PP^1\times\PP^1)(\QQ)$ qui est \emph{mince}, dont l'apparence dans le principe de Batyrev-Manin a reçu beaucoup d'attentions. Le but de la Section \ref{se:othervarieties} est de souligner qu'en construisant des variétés comme produit de $Y_4$ avec les espaces projectifs, le phénomène que des courbes nodales donnent les meilleures approximations génériques est très général.
	\subsection{Notations}
 	La lettre $p$ désigne un nombre premier. 
    On note $v_p$ la fonction d'ordre $p$-adique, $\varphi$ la fonction indicatrice d'Euler,
	$\tau$ le fonction arithmétique donnant le nombre total de diviseurs donnée par:
	\begin{equation}\label{eq:thefunctiontau}
	\tau(n)=\sum_{d|n} 1,\quad \text{pour } n\geqslant 1,
	\end{equation}
	et $\mu$ la fonction de Möbius.
	On définit quelques fonctions arithmétiques, pour $n\geqslant 1$, 
	\begin{equation}\label{eq:thefunctionphi}
	\phi(n)=\prod_{p|n}\left(1-\frac{1}{p}\right),
	\end{equation}
	\begin{equation}\label{eq:arithmeticfunction 1}
	\Psi_1(n)=\prod_{p|n}\left(1+\frac{1}{p}\right)^{-1},
	\end{equation}
	\begin{equation}\label{eq:functiong}
	g(n)=\prod_{p}p^{\lceil \frac{v_p(n)}{2}\rceil},\quad n\in\NN,
	\end{equation}
	\begin{equation}\label{eq:arithmeticfunctionkey}
	\Psi(n)=\sum_{d|n}\Psi_1(d)\sum_{e|d}\frac{\mu(e)}{e}=\sum_{d|n}\Psi_1(d)\phi(d),
	\end{equation}
	\begin{equation}\label{eq:arithmeticfunctionPhi}
	\Phi(n)=\sum_{d|n}\frac{d\Psi_1(g(d))}{g(d)}\sum_{e|d}\frac{\mu(e)}{e}=\sum_{d|n}\frac{d\Psi_1(g(d))}{g(d)}\phi(d).
	\end{equation}
	Alors l'indicatrice d'Euler est donnée par $\varphi(n)=n\phi(n)$.
	
	Pour un nombre réel $x$, on note $\lfloor x\rfloor$ \textit{la partie entière} de $x$ qui est le plus grand entier $\leqslant x$, $\lceil x \rceil$ le plus petit entier $\geqslant x$, et $\{x\}=x-\lfloor x\rfloor\in\mathopen [0,1\mathclose[ $ \textit{la partie fractionnaire} de $x$.
	
	Fixons $E$ un ensemble. Pour $A,B\subseteq E$, on note $A\Delta B$ le sous-ensemble de $E$ défini par
\begin{equation}\label{eq:delta}
	A \Delta B=\left(A\setminus B \right)\bigcup \left(B\setminus A\right)=(A\bigcup B)\setminus(A\bigcap B).
\end{equation}
Pour $F$ un sous-ensemble de $\ZZ^2$, on note
$$F_{\operatorname{prem}}=\{\mathbf{x}=(x_1,x_2)\in\ZZ^2:\pgcd(x_1,x_2)=1\}$$
l'ensemble des points primitifs de $F$.
	\tableofcontents
	\newpage
 \section{Constantes d'approximation}\label{se:approximationconstants}
 \subsection{Constantes d'approximation et constante essentielle}
 On définit plusieurs \textit{constantes d'approximation} d'un point rationnel sur une variété algébrique généralisant la notion de \textit{mesure d'irrationalité} venant de l'approximation diophantienne classique.
 Dans cet article, nous considérons uniquement les distances archimédiennes réelles. Une définition analogue pourrait s'appliquer aux places non-archimédiennes.
 Soit $X$ une variété projective définie sur un corps de nombres $k$.  On fixe un point rationnel $Q\in X(\bar{k})$, $\nu$ une place réelle et $d=d_\nu$ une distance projective (cf. \cite[(2.1)]{M-R}). Par exemple, si $X\hookrightarrow \PP^n$, $d$ peut être la restriction à $X$ de la distance projective usuelle sur $\PP^n$. Soit $L$ un fibré en droites sur $X$ muni d'une hauteur de Weil absolue $H_L$.
 Pour $V$ une partie constructible de $X$, on considère les ensembles 
 $$A(Q,V)=\{\gamma>0|\exists (y_i)\in (V(k)\setminus \{Q\})^\NN,\exists C>0,d(Q,y_i)\to 0 \text{ et } d(Q,y_i)^\gamma H_L(y_i)<C,\forall i\},$$
 $$ B(Q,V)=\{\gamma>0| \exists C>0 , d(Q,y)^\gamma H_{L}(y)\geqslant C, \forall y\in  V(k)\setminus\{Q\}\}.$$
 Chacun est un intervalle: si $\gamma_0 \in B(Q,V)$, tout $0<\gamma<\gamma_0$ appartient à $B(Q,V)$; 
 de même, si $\gamma_0 \in A(Q,V)$, tout $\gamma>\gamma_0$ appartient à $A(Q,V)$. On suppose que $k(Q)$ le corps résiduel de $Q$ est contenu dans $k_\nu$, la complété de $k$.
 \begin{definition}\label{def:Northcott}
 	Soit $U$ un ouvert de $X$. On dit que $L$ vérifie \emph{la propriété de Northcott} pour $U$ si pour tout $B>0$,
 	$$\sharp\{x\in U(k):H_L(x)\leqslant B\}<\infty.$$
 \end{definition}
 \begin{prop-def}[\cite{M-R}, Proposition 2.11]\label{pd:appconstant}
 	Soient $Q\in X(\bar{k})$ et $V$ une partie constructible de $X$. Supposons que $L$ vérifie la propriété de Northcott pour un ouvert de Zariski $U$ contenant $Q$. Alors on a
 	$$\inf A(Q,U\cap V)=\sup B(Q,U\cap V).$$
 	Cette quantité est appelée \emph{constante d'approximation} en $Q$ dans $V$, notée $\alpha_L(Q,V)$, ou simplement $\alpha(Q,V)$ si le fibré $L$ est fixé dans la considération.
 	S'il existe une sous-variété fermée $Z$ contenue dans $V$ tel que $\alpha(Q,Z)=\alpha(Q,V)$, on dit que $\alpha(Q,V)$ peut être calculée sur $Z$.
 \end{prop-def}
 Cette définition ne dépend pas du choix de $U$. Dans la suite, lorsque l'on parle de la constante d'approximation, 
 on admet que $L$ vérifie la propriété de Northcott pour un ouvert dense contenant $Q$.
 \begin{definition}[\cite{Pagelot}]\label{df:essential}
 	On définit \emph{la constante essentielle} en $Q$ comme la quantité
 	\begin{equation}
 	\aess_L(Q)=\aess(Q)=\sup_V \alpha(Q,V)
 	\end{equation}
 	où $V$ parcourt tous les parties constructibles denses de $X$.
 	S'il existe une sous-variété $Z$ de $X$ de sorte que pour tout ouvert dense $W\subset Z$, $\alpha(Q,W)<\aess(Q)$, on dit que $Z$ est \textit{localement accumulatrice}.
 \end{definition}
 Ces notions ne dépendent ni du choix de la hauteur associée à fibré $L$, ni du choix de la distance projective.
 
   Puisque le problème que l'on étudie est local et que le point que l'on va approcher est dans l'orbite ouverte, on peut se restreindre à des ouverts de $X$. Dans la suite on suppose que le fibré en droites $L$ est gros, c'est-à-dire sa classe est dans l'intérieur du cône pseudoeffectif de $X$. Alors il existe un entier $m$ tel que l'application rationnelle $\Upsilon_m:X\dashrightarrow \PP(H^0(X,L^{\otimes m}))$ est birationnelle vers son image. En particulier on peut choisir un ouvert $U$ tel que $U\simeq\Upsilon_m(U)\subset \PP(H^0(X,L^{\otimes m}))$ et donc 
   $$\sharp\{x\in U(k):H_L(x)\leqslant B\}=\sharp\{y\in\Upsilon_m(U)(k):H_{\mathcal{O}(1)}(y)\leqslant B\}<\infty$$
   grâce au théorème de Northcott (cf. \cite[\S 2.4]{Serre}).
   Dans ce qui va suivre, on suppose toujours que $L$ vérifie cette propriété, ce qui est le cas pour le fibré anticanonique des variétés toriques complètes lisses (parce que le polytope associé à $-K_X$ est le reflété de l'éventail de $X$ contenant l'origine, cf. \cite{fulton} \S3.4 p. 66). En outre le fibré $-K_X$ est sans lieu de base sur l'orbite ouverte et donc la propriété de Northcott est vérifiée en tout point de l'ouvert.
 
 Toutes ces constantes sont connues en dimension $1$.
 \begin{theorem}[Théorème de Roth, Principe de Dirichlet, \cite{M-R} Lemma 2.15]\label{th:approximationtheorem1}
 	Soient $x\in\PP^1(\bar{\QQ})\cap \PP^1(\RR)$ et $d\in\NN$. On fixe une hauteur de Weil absolue associée à $\mathcal{O}(d)$.
 	Alors
 	\begin{equation}
 	\alpha(x,\PP^1)=
 	\begin{cases}
 	d \mbox{ si } x\in \PP^1(\QQ);\\
 	\frac{d}{2} \mbox{ si } x \mbox{ est irrationnel}.
 	\end{cases}
 	\end{equation}
 \end{theorem}
 \begin{remark}[\cite{M-R} Theorem 2.16]\label{rmk:thmofmckinnon}
 	La constante d'approximation est sensible aux singularités. Prenons une courbe rationnelle $C$ définie sur $k$ et $L$ un faisceau inversible ample sur $C$. On fixe $Q\in C(\bar{k})$. Soit $f:\PP^1\to C$ le morphisme de normalisation. Alors 
 	$$\alpha(Q,C)=\min_{P\in f^{-1}(Q)}\frac{d}{m_P r_P},$$
 	où
 	$d=\deg_C(L)$, $m_P$ est la multiplicité de la branche de $C$ passant par $Q$ correspondant à $P$ et 
 	\begin{equation*}
 		r_P=\begin{cases}
 		0 \text{ si } k(P)\not\subset k_\nu;\\
 		1 \text{ si } k(P)=k;\\
 		2 \text{ sinon}.
 		\end{cases}
 	\end{equation*}
 	Ici $r_P=0$ veut dire que $\frac{d}{m_Pr_P}=\infty$, qui arrive par exemple pour $k=\QQ$ et $Q$ un point imaginaire $k(Q)\not\subset \RR=\QQ_\infty$ qu'il est donc impossible d'approcher par des nombres rationnels. 
 	
 	Dans ce texte, on s'intéresse particulièrement aux cas où $C$ est une section irréductible de $\mathcal{O}(2,2)$ définie sur $\QQ$ dans $\PP^1\times \PP^1$ qui est nodale en $Q\in C(\QQ)$. Soient $\tau_i(Q)\in \bar{\QQ},i=1,2$ les pentes des deux branches de $C$ en $Q$. Conservant les notations ci-dessus, le pré-image $f^{-1}(Q)$ contient deux points $Q_1,Q_2$, dont les corps de définition $k(Q_1)=k(Q_2)$ qui valent $\QQ(\tau_i(Q))$. Notons que $[\QQ(\tau_i(Q)):\QQ]\leqslant 2$.
 	Alors la formule ci-dessus s'écrit
 	$$\alpha(Q,C)=\frac{d}{\max(r_{Q_1},r_{Q_2})}=\begin{cases}
 	d \text{ si } k(Q_i)=\QQ;\\
 	\frac{d}{2}\text{ si } \QQ\subsetneq k(Q_i)\subset \RR;\\
 	\infty \text{ sinon}.
 	\end{cases}$$
 \end{remark}
C'est un fait empirique que les approximants proviennent essentiellement des courbes rationnelles. En effet D. McKinnon a proposé la conjecture suivante  \cite[Conjecture 2.7]{Mc}:
 \begin{conjecture}[McKinnon]\label{conj:mckinnon}
 	Soient $L$ un fibré ample sur $X$ avec une hauteur de Weil choisie et $Q\in X(k)$. Supposons qu'il y a une courbe rationnelle passant par $Q$ (ceci implique en particulier que $\alpha(Q,X)<\infty$). Alors il existe une courbe rationnelle $C$ telle que 
 	$$\alpha(Q,C)=\alpha(Q,X).$$
 \end{conjecture}
Cette conjecture prédit que $\alpha(Q,X)$ peut être calculée sur des courbes rationnelles passant par $Q$, qui est accord avec notre cas.
 
 On utilise souvent la procédure suivante pour déterminer la constante essentielle. Elle consiste essentiellement en deux étapes: démontrer une borne inférieure uniforme valide dans un ouvert pour le produit d'une certaine puissance de la distance avec la hauteur et trouver une famille de courbes rationnelles dont la réunion dans la variété est dense pour la topologie de Zariski et chacune atteint la constante d'approximation souhaitée. Il convient de remarquer que cette procédure donne seulement une condition \og suffisante \fg{}\ pour trouver $\aess$. En général, il n'est pas clair que $\aess$ puisse toujours être atteinte sur une famille de courbes rationnelles.
 \subsection{Opération de zoom}\label{se:zoomoperation}
 Maintenant nous allons décrire l'opération de zoom en détail. On suppose toujours que $k_\nu=\RR$. On identifie localement (pour la topologie réelle) $X(\RR)$ avec l'espace tangent $T_Q X$ en envoyant $Q$ sur l'origine, en utilisant un système de coordonnées analytiques. Dans $T_Q X$, on définit une \og contraction\fg\ par une homothétie de rapport une puissance de $B$ (la lettre $B$ désigne ici la borne de la hauteur). Plus précisément, si l'on note le difféomorphisme local entre $X(\RR)$ et $T_Q X$ par $\rho$, alors \textit{l'opérateur de zoom} $\Psi_{r,B}$ avec \textit{le facteur} $r>0$ est défini par
 $$\Psi_{r,B}(P)=B^{\frac{1}{r}}\rho(P),$$
 pour tout $P\in X(\RR)$ dans un voisinage de $Q$ où $\rho$ est défini. 
 On note qu'avec cette convention, \emph{plus $r$ est grand, plus le zoom est faible}.
 Fixons $U$ une sous-variété de $X$.
 Rappelons la notation $\CBXQ$ (Notation \ref{notationcqbx}). Soit $f\in\CBXQ$, on introduit la mesure $\delta_{U,Q,B,r}$ \eqref{eq:deltameasure} définie par
 $$\delta_{U,Q,B,r}(f)=\sum_{x\in U(k):H_L(x)\leqslant B} f(B^{\frac{1}{r}}\rho(x)).$$
 Cette définition dépend de la différentielle en l'origine du difféomorphisme choisi.
  
Supposons que $L=\omega_{X}^{-1}$ et qu'il existe $\beta\geqslant0,\gamma\geqslant 0$ tels que \eqref{eq:growthorder} soit vérifiée pour une infinité de $B$.
 	On peut donner, à l'aide de la conjecture de Batyrev-Manin et celle de Peyre sur l'équidistribution des points rationnels, une prédiction naïve de cet ordre de grandeur. Soit $n=\dim X$.  Quand on calcul $(\delta_{U,Q,B,r}(\chi(\varepsilon)))_B$, on est en train de compter des points rationnels dans la boule $\mathbb{B}(0,B^{-\frac{1}{r}}\varepsilon)$ dont le volume est de grandeur $B^{-\frac{n}{r}}$. Supposons la validité de la conjecture de Batyrev-Manin pour l'ouvert $U$, c'est-à-dire (on note $\kappa=\operatorname{rg}(\operatorname{Pic}(X))$)
 $$\sharp U_{H\leqslant B}\sim C(X)\ B (\log B)^{\kappa-1},$$
 et la validité d'une forme forte de l'équidistribution, c'est-à-dire pour des voisinages réels dont la taille dépendent de $B$, à qui donnerait dans \eqref{eq:growthorder},
 $$\delta_{U,Q,B,r} (\chi(\varepsilon))\gg\ll_\varepsilon B^{-\frac{n}{r}}\times B (\log B)^{\kappa-1}=B^{1-\frac{n}{r}}(\log B)^{\kappa-1}.$$
 Autrement dit, on devrait avoir 
 \begin{equation}\label{eq:naivepredictionofpowers}
 \beta=1-\frac{n}{r},\quad \gamma=\kappa-1.
 \end{equation}
 En fait, dans les (rares) exemples connus, la valeur de $\beta$ est en accord avec cette prédiction, mais celle de $\gamma$ ne l'est pas. Il arrive parfois que $\beta=\gamma=0$, comme pour la variété considérée dans ce texte.
 Le comportement de la suite \eqref{eq:normalisedmeasures} décrit la \textit{distribution locale} autour du point $Q$ sur $X$.
 On appelle \textit{mesure limite} comme la limite de la suite \eqref{eq:normalisedmeasures} (si elle existe). Prenant du recul, une minoration du type
 $$\frac{1}{B^\beta (\log B)^\gamma}\delta_{U,Q,B,r}(\chi(\varepsilon)) \gg \int \chi(\varepsilon) \operatorname{d} \widetilde{\delta _r}, \quad \forall \varepsilon >0$$
 pour une mesure $\widetilde{\delta_r}$ définie sur $T_Q X$ veut dire que la distribution locale est uniformément minorée par une mesure, ce qui assure l'existence de \og beaucoup\fg\ de points dans le zoom. C'est ce que nous allons démontrer pour certains zooms de la surface $Y_4$.
 
 Le zoom d'un certain facteur va mettre en évidence les sous-variétés dont la constante d'approximation atteignent ce facteur et ignorer celles dont la constante d'approximation est plus grande. Donc le zoom d'un facteur plus petit que $\aess$ va nous ramener à des sous-variétés fermées localement accumulatrices. Cela plus la proposition suivante explique la raison pour laquelle on s'intéresse au zoom avec le facteur $\geqslant\aess$.
 \begin{proposition}\label{po:weakzoom}
 	Soit $U$ une partie constructible dense de $X$ pour laquelle $L$ vérifie la propriété de Northcott telle que $\alpha(Q,U)=\aess(Q)$. Alors pour tout $0<r<\aess$ et $f\in\CBXQ$, on a 
 	\begin{enumerate}
 		\item $\delta_{U,Q,B,r}(f)=f(Q)$ si $Q\in U(\bar{k})$;
 		\item  $\delta_{U,Q,B,r}(f)=0$ sinon,
 	\end{enumerate}
 	pour $B\gg _{f,r} 1$.
 \end{proposition}
 \begin{proof}
 	On peut supposer que $f$ est une fonction caractéristique $\chi(\varepsilon)$ pour $\varepsilon>0$. On choisit une distance $d$ sur $T_Q X$. D'après la Proposition-Définition \ref{pd:appconstant}, 
 	pour $\beta>0$ tel que $r+\beta<\aess$, il existe une $C>0$ telle que pour tout $P\in U(k)$ différent de $Q$, $d(\varrho(P),\varrho(Q))^{r+\beta}H_L(P)\geqslant C$.
 	Supposons que $H(P)\leqslant B$. Cela implique que
 	$$d(\varrho(P),\varrho(Q))\geqslant C^{\frac{1}{r+\beta}}H_L(P)^{-\frac{1}{r+\beta}}\geqslant C^\prime B^{-\frac{1}{r+\beta}}.$$
 	Les points après zoom doivent situer dans le support de $f$, donc
 	$d(\varrho(P),\varrho(Q))B^{\frac{1}{r}}\leqslant \varepsilon,$ d'où $d(\varrho(P),\varrho(Q))\leqslant\varepsilon B^{-\frac{1}{r}}.$
 	Donc un tel $P$ n'existe pas quand $B>(\varepsilon (C^\prime)^{-1})^{\frac{r(r+\beta)}{\beta}}$.
 \end{proof}
 \begin{definition}\label{def:criticalzoom}
 	On dit que le zoom est \textit{critique} (resp. \textit{sous-critique}) si son facteur $r=\aess$ (resp. $>\aess$).
 \end{definition}

 Inspiré par le fait que certaines courbes rationnelles contiennent beaucoup plus de points que les autres dans une même variété dans l'opération de zoom, bien que leur constantes d'approximation coïncident, nous proposons la définition suivante pour les distinguer. 
 \begin{definition}\label{df:weaklocal}
 	Supposons que $\aess(Q)<\infty$.
 	On dit qu'une sous-variété fermée stricte $W$ de $X$ est \textit{localement faiblement accumulatrice} si $\aess(Q)=\alpha(Q,W)$ et s'il existe $U_1$ ouvert dense de $X$ tel que pour tout $\varepsilon>0$ suffisamment grand et pour tout ouvert $U_2$ dense de $X$ satisfaisant à $\alpha(Q,U_2)=\aess(Q)$ et $U_1\cap U_2\cap W$ dense dans $W$, on ait, en notant $U=U_1\cap U_2$,
 	$$\delta_{U\setminus W,Q,B,\aess(Q)}(\chi(\varepsilon))=o\left(\delta_{U\cap W,Q,B,\aess(Q)}(\chi(\varepsilon))\right).$$
 \end{definition}
Autrement dit, les sous-variétés localement faiblement accumulatrices ne sont pas localement accumulatrices mais dominent leur complémentaire dans le zoom critique.

 Dans cet article on s'intéresse aux variétés toriques lisses et projectives dont le fibré en droites anticanonique est engendré pas ses sections globales et
 on utilise une hauteur de Weil associée. Les travaux \cite{B-T1} \cite{B-T2} \cite{Salberger} confirment le principe de Batyrev-Manin pour les variétés toriques munies d'une hauteur de Weil associée au fibré anticanonique. Ceci est en faveur de la prédiction précédente sur l'ordre de grandeur \eqref{eq:naivepredictionofpowers} (au mois pour $\beta$).
 
\section{Solutions des équations de Pell-Fermat}\label{se:theoryalgebraic}
Le but de cette section est de discuter la structure des solutions $(x,y)\in\ZZ^2$ des équations de Pell-Fermat généralisées \eqref{eq:diophantineequation}.
Le contenu de la Section \ref{se:pell-fermatsimple} est consacré aux solutions des équations de type $x^2-Dy^2=m$, dont la structure est classiquement connue. En effet l'ensemble des solutions est un espace homogène par l'action d'un sous-groupe $\muktp$ d'unités (cf. \eqref{eq:thegroupud*}). On donne aussi une majoration effective du nombre des orbites.
Le résultat principal est le Corollaire \ref{co:structuretheorem1}, qui sera utilisé dans la démonstration de la Proposition \ref{po:controlofepsilon} plus loin.
Deuxièmement dans la Section \ref{se:eqpellfermat}, on généralise des résultats classiques à \eqref{eq:diophantineequation}. Ceci sert comme préparation à la démonstration de la Proposition \ref{po:trackpoints}.
\subsection{Rappels sur la théorie algébrique des corps de nombres quadratiques}\label{se:pell-fermatsimple}
On rappelle des faits classiques sur les corps quadratiques réels. Pour les détails, voir par exemple (\cite{Samuel} 2.5, 4.6, 5.4). On fixe un entier  $D>0$ sans facteur carré.
Soit $K=\QQ(\sqrt{D})$. 
 On note $\varepsilon_D$ l'unité fondamentale de $\OK$ telle que $\varepsilon_D>1$. Alors le groupe d'unités $\OK^*=\{\pm 1\}\times\{\VD^n,n\in\ZZ\}$.
 On introduit les sous-groupes de $\OK^*$: 
\begin{equation}\label{eq:thegroupud+}
 \mukp=\{z\in\OK:N(z)=1\},
\end{equation}
\begin{equation}\label{eq:thegroupud*}
 \muktp=\{z\in\ZZ+\ZZ\sqrt{D}:N(z)=1\}.
\end{equation}
On note $\VD^+>1$ (resp. $\VDT>1$)  l'élément de $\mukp$ (resp. $\muktp$) dont la classe engendre le groupe $\mukp/\{\pm 1\}$ (resp. $\muktp/\{\pm 1\}$).
 \begin{lemma}[\cite{Samuel}, p. 78]\label{le:samuel}
 	\begin{equation}
 	\VDT=\begin{cases}
 	\VD^+ \text{ si } \VD\in\ZZD;\\
 	(\VD^+)^3 \text{ si } \VD\notin\ZZD.
 	\end{cases}
 	\end{equation}
 \end{lemma}
Pour $m\in\ZZ$, on désigne par $\SMP$ l'ensemble des solutions entières de l'équation Pell-Fermat \begin{equation}\label{eq:pell-fermat}
x^2-D y^2=m. 
\end{equation}
\begin{align*}
\SMP=\{z=x+y\sqrt{D}\in\ZZD:N(z)=m\}.
\end{align*}
On note aussi
$$\SM=\{z\in\OK:N(z)=m\}.$$
 La notation $I\triangleleft A$ signifie ici que $I$ est un idéal de l'anneau $A$. On utilise $N_{K/\QQ}(\cdot)$ pour désigner la norme des idéaux et $N(\cdot)$ pour désigner la norme des éléments.
On définit $$\mathcal{I}_{|m|}=\{I\triangleleft\mathcal{O}_K:N_{K/\QQ}(I)=|m|\}.$$
On note In (resp. De, Ra) l'ensemble des nombres premiers qui sont inertes (resp. sont décomposés, se ramifient) dans l'extension $\QQ(\sqrt{D})/\QQ$.  
Pour tout $z=x+y\sqrt{D}\in K$, sa conjugaison est définie par
$$\bar{z}=x-y\sqrt{D}.$$
On a besoin d'une estimation explicite du cardinal de l'ensemble $\IM$ ainsi que celui de $\IMM$.
\begin{proposition}
On fixe $m\in\ZZ$. Soient les entiers $m_1,m_2>0$ définis par
\begin{equation}\label{factorisationofm}
m_1=\prod_{\substack{p\in\operatorname{De}\\p|m}}p^{v_p(m)},\quad m_2=\prod_{\substack{p\notin\operatorname{De}\\p|m}} p^{v_p(m)}.
\end{equation}
Alors on a
\begin{equation*}
	\sharp \IM=\begin{cases}
	0 \text{ s'il existe } p\in\operatorname{In},p|m \text{ et } v_p(m) \text{ impair};\\
		\prod _{\substack{ p\in\operatorname{De}\\p|m}}(v_p(m)+1)=\tau(m_1) \text{ sinon};	
	\end{cases}
\end{equation*}
où $\tau$ désigne la fonction arithmétique de nombre de diviseurs (cf. \eqref{eq:thefunctiontau}).
\end{proposition}
On observe que l'ensemble $\IM$ est parfois vide. 
\begin{corollary}\label{co:upperboundforimm}
\begin{equation*}
\sharp \IM\leqslant \tau(m_1)\leqslant\tau(|m|).
\end{equation*}
\end{corollary}
\begin{proof}[Démonstration de la proposition]
	Quelque soit $I\in\IM$, on considère sa décomposition
	$$I=\prod_{\mathfrak{p}\text{ premier}\triangleleft\OK}\mathfrak{p}^{\ep}.$$
	Alors comme l'application $N_{K/\QQ}(\cdot)$ est multiplicative,
	$$|m|=N_{K/\QQ}(I)=\prod_{\mathfrak{p}\text{ premier}\triangleleft\OK}N_{K/\QQ}(\mathfrak{p})^{\ep}=\prod _{\substack{p\in\operatorname{De}\\p\OK=\mathfrak{p}\bar{\mathfrak{p}}}}p^{\ep+e_{\bar{\mathfrak{p}}}}\prod_{\substack{p\in\operatorname{In}\\p\OK=\mathfrak{p}}}p^{2\ep}\prod_{\substack{p\in\operatorname{Ra}\\p\OK=\mathfrak{p}^2}}p^{\ep}.$$
	La quantité à droite est en fait la factorisation de $m$ (\ref{factorisationofm}). On en conclut que
	\begin{equation*}v_p(m)=
	\begin{cases}
	e_{\bar{\mathfrak{p}}}+\ep \text{ si } p\in\operatorname{De};\\
	\ep \text{ si }p\in\operatorname{Ra};\\
	2\ep \text{ si }p\in\operatorname{In}.\\
	\end{cases}
	\end{equation*}
	En particulier $v_p(m)$ est pair pour $p\in\operatorname{In}$.
	Réciproquement, pour tout $p\in\operatorname{In}$, on choisit $m_p$ pair et pour tout $p\in\operatorname{De}$, donnons-nous un couple $(\ep^1,\ep^2)\in\NN^2$ tel que $\ep^1+\ep^2=m_p$. Le produit
	$$\prod _{\substack{p\in\operatorname{De}\\p\OK=\mathfrak{p}\bar{\mathfrak{p}}}}\mathfrak{p}^{\ep^1}\bar{\mathfrak{p}}^{\ep^2}\prod_{\substack{p\in\operatorname{In}\\p\OK=\mathfrak{p}}}\mathfrak{p}^{\frac{m_p}{2}}\prod_{\substack{p\in\operatorname{Ra}\\p\OK=\mathfrak{p}^2}}\mathfrak{p}^{m_p}.$$ est un idéal de norme $|m|$. 
	Pour chaque $p\in\operatorname{De}$, il y a $v_p(m)+1$ tels couples d'entiers. D'où l'énoncé.
\end{proof}
\begin{corollary}\label{co:structuretheorem1}
	Le groupe $\mukp$ \eqref{eq:thegroupud*} agit sur $\SM$ librement et l'ensemble des orbites est en bijection avec 
	$$\IMP=\{I\triangleleft\OK:\exists a_I\in\OK,N(a_I)=m,I=(a_I)\}.$$
	Le groupe $\muktp$ \eqref{eq:thegroupud+} agit sur l'ensemble $\SMP$ librement et on a la majoration suivante pour le cardinal de l'ensemble des orbites:
	$$\sharp(\SMP/\muktp)\leqslant 3\sharp \IMP\leqslant 3\sharp \IM\leqslant 3\tau(m).$$
\end{corollary}
\begin{proof}
	D'après le Lemme \ref{le:samuel} on a $[\mukp:\muktp]\leqslant 3$. Le cardinal des orbites de l'action de $\muktp$ sur $\SMP$, qui est aussi libre,  est majoré par
	$$\sharp(\SMP/\muktp)\leqslant \sharp(\SM/\muktp)\leqslant 3\sharp (\SM/\mukp)=3\sharp \IMP\leqslant 3\sharp \IM.$$
	La conclusion découle du Corollaire \ref{co:upperboundforimm}.
\end{proof}
\subsection{Théorème de Dirichlet et équations de Pell-Fermat généralisées}\label{se:eqpellfermat}
On rappelle \eqref{eq:diophantineequation} et la notation $\SABC$.
On factorise 
\begin{equation}\label{eq:factorise}
a=A^\prime (a^\prime)^2,\quad b=B^\prime (b^\prime)^2,
\end{equation}
avec $A^\prime,B^\prime$ sans facteur carré.
On rappelle qu'une solution $(x,y)\in\SABC$ est dite \textit{primitive} si $\pgcd(x,y)=1$.
Comme toute solution $(x,y)\in\ZZ^2$ de l'équation (on rappelle que $\pgcd(a,b)=1$)
\begin{equation}\label{eq:pellgeneralised}
ax^2-by^2=a
\end{equation}
vérifie $A^\prime a^\prime|y$, en écrivant $y=az,w=b^\prime z$ et en utilisant les notations  \eqref{eq:factorise}, on obtient l'équivalence suivante
\begin{equation}\label{eq:pell-fermatequiv}
\exists (x,y)\in\ZZ^2 ,ax^2-by^2=a\Leftrightarrow \exists (x,z)\in\ZZ^2 ,x^2-A^\prime b z^2=1\Leftrightarrow \exists (x,w)\in\ZZ\times  b^\prime \ZZ,x^2-A^\prime B^\prime w^2=1.
\end{equation}
Notons qu'ici $A^\prime B^\prime$ est un entier sans facteur carré.
Ainsi chercher des solutions de \eqref{eq:pellgeneralised} revient à résoudre une équation de Pell sous certaine condition de divisibilité, dont la résolubilité est assurée grâce à un théorème classique de Dirichlet. On en donne une preuve courte pour faciliter la lecture.
\begin{theorem}[Dirichlet \cite{Dirichlet}]\label{th:Dirichlet}
Soit $D$ en entier positif qui n'est pas un carré. On note $$D=ds^2,\quad d,s \in\NN^*, d \text{ sans facteur carré}.$$
Alors l'ensemble $\mathcal{A}_D^*$ des solutions $(x,y)\in\ZZ^2$ de l'équation
\begin{equation*}
x^2-Dy^2=1
\end{equation*}
est un sous-groupe de $\mathcal{O}_{\QQ(\sqrt{d})}^*$ d'indice fini.
\end{theorem}
\begin{proof}
	Notons comme précédemment $K=\QQ(\sqrt{D})$.
	En vue du Lemme \ref{le:samuel}, on peut supposer que $$\varepsilon_d=\varepsilon_d^*=u_1+v_1\sqrt{d}\in\ZZ+\ZZ\sqrt{d}.$$
	On identifie $\mathcal{A}_D^*$ avec 
	\begin{equation}\label{eq:pellgeneralised2}
	\{z=u+v\sqrt{d}\in\ZZ+\ZZ\sqrt{d}:s|v,N_{K/\QQ}(z)=1\},
	\end{equation}
	qui est un sous-groupe de $\mathcal{O}_{K}^*$. Il suffit de démontrer l'existence d'un élément de $\mathcal{A}_D^*$ tel que sa classe dans $\mathcal{A}_D^*/\{\pm1\}$ soit non-nulle. Le raisonnement suivant repose sur celui d'origine de Dirichlet.\\
	\textbf{Étape 1}. On suppose que $s=p$ est un nombre premier. On utilise dans cette partie la notation $u_n+v_n\sqrt{d}=\varepsilon_d^n$ pour $n\in\NN^*$. Si $p|v_1$, il n'y a rien à démontrer. Si $p=2$, on voit que l'élément 
	$(\varepsilon_d)^2=u_1^2+v_1^2 d+2u_1v_1\sqrt{d}$ suffit. Dans la suite, on suppose que $p\nmid v_1$ et $p$ impair. \\
	\textbf{Cas I}. $p\mid d$. Autrement dit, l'idéal $p\ZZ$ se ramifie dans $K$. Alors l'élément $u_p+v_p\sqrt{d}$ vérifie 
	$$v_p=\sum_{l \text{ impair},1\leqslant l\leqslant p} \binom{p}{l}u_1^{p-l} v_1^{l}d^{\frac{l-1}{2}}=\sum_{l \text{ impair},1\leqslant l < p} \binom{p}{l}u_1^{p-l} v_1^{l}d^{\frac{l-1}{2}}+v_1^pd^{\frac{p-1}{2}},$$
	et donc $p\mid v_p$.\\
	\textbf{Cas II}. $p\nmid d$. Si $p\mathcal{O}_{K}=\mathfrak{p}_1\mathfrak{p}_2$, alors on a $\mathcal{O}_{K}/\mathfrak{p}_i\simeq \mathbb{F}_{p}$ et l'automorphisme de Frobenius nous donne 
	$$\varepsilon_d^p\equiv \varepsilon_d\mod\mathfrak{p}_i\Leftrightarrow \mathfrak{p}_i\mid (\varepsilon_d^{p-1}-1)\mathcal{O}_{K}.$$
	Donc on obtient que
	$$p\mathcal{O}_{K}=\mathfrak{p}_1\mathfrak{p}_2\mid(\varepsilon_d^{p-1}-1)\mathcal{O}_{K}$$
	et donc $p\mid v_{p-1}$. Si $p\mathcal{O}_{K}$ est un idéal premier, on a $\mathcal{O}_{K}/p\mathcal{O}_{K}\simeq \mathbb{F}_{p^2}$. Dans ce cas on a
	$$p\mathcal{O}_{K}\mid(\varepsilon_d^{p^2-1}-1)\mathcal{O}_{K}$$
	et donc $p|v_{p^2-1}$.\\
	\textbf{Étape 2}. On suppose que $s=p^e$ où $e\geqslant 2$. Supposons que l'on a trouvé $z+w\sqrt{d}\in\mathcal{O}_{K}^*$ satisfaisant à $p^{e-1}|w$. Alors l'élément $z_p+w_p\sqrt{d}=(z+w\sqrt{d})^p$ vérifie que $p^e|w_p$.\\
	\textbf{Étape 3}. On suppose que $s=\prod_{i=1}^{m}p_i^{r_i}$ où $p_i$ sont des nombres premiers distincts. Supposons que l'on a trouvé un élément $z+w\sqrt{d}$ satisfaisant à $\prod_{i=1}^{m-1}p_i^{r_i}\mid w$. Puisque toute sa puissance $z_n+w_n\sqrt{d}=(z+w\sqrt{d})^n$ satisfait à $\prod_{i=1}^{m-1}p_i^{r_i}\mid w_n$, on reprend l'argument précédent en remplaçant $\varepsilon_d$ par $z+w\sqrt{d}$.
\end{proof}
\begin{proposition}\label{po:pgcddonotchange}
	Notons $\theta=\sqrt{\frac{b}{a}}$. Le groupe $\mathcal{A}_{ab}^*$ \eqref{eq:pellgeneralised2} agit librement sur l'ensemble des solutions de \eqref{eq:diophantineequation} de la manière suivante. Pour $(u,v)\in S^*_{a,b,a}$, et $(x,y)\in\ZZ^2$ une solution de \eqref{eq:diophantineequation}, $(x^\prime,y^\prime)\in\ZZ^2$ défini par 
\begin{equation}\label{eq:stepbystep}
		x^\prime+\theta y^\prime=(x+\theta y)(u+\theta v).
\end{equation}
	est une solution de \eqref{eq:diophantineequation}. De plus, $\pgcd(x^\prime,y^\prime)=\pgcd(x,y)$. Par conséquent, si $\SABC\neq \varnothing$, alors $\pgcd(x,y)$ ne dépend pas des éléments $(x,y)\in \SABC$ qui sont dans une même orbite. 
\end{proposition}
\begin{proof}
	La vérification que $(x^\prime,y^\prime)\in \SABC$ est immédiate. 
	On a évidemment $$\pgcd(x,y)|\pgcd(x^\prime,y^\prime).$$
	Puisque 
	$$x+\theta y=(x^\prime+\theta y^\prime)(u-\theta v),$$
	on en déduit que
	$$\pgcd(x^\prime,y^\prime)|\pgcd(x,y),$$
	d'où l'énoncé.
\end{proof}

\begin{corollary}\label{rmk:inftysol}
	Pour $a,b\in\NN^*,\gcd(a,b)=1$ fixés, alors
	$$\sharp\{c\in\ZZ: \exists (x,y)\in \SABC,\pgcd(x,y)=1\}=\infty.$$
	De plus, une fois que $\SABC\neq \varnothing$, il existe $(x_n,y_n)\in\SABC$ telle que $x_n,y_n\to\infty$.
\end{corollary}
\begin{proof}
	Les couples $(x_0,1)$ avec $x_0\in\ZZ$ quelconque donnent des valeurs arbitrairement grandes du polynôme $|ax^2-b|$, le premier énoncé en résulte compte tenu de la  Proposition \ref{po:pgcddonotchange}. 
	Pour le deuxième, soit $(x,y)\in\SABC$. On peut supposer que $x,y>0$. Le Théorème \ref{th:Dirichlet} assure l'existence d'un élément $v+w\sqrt{A^\prime B^\prime}\in\mathcal{A}^*_{ab}$ de sorte que $v,w\in\NN^*$. Alors d'après l'équivalence \eqref{eq:pell-fermatequiv}, l'élément $(v,\frac{a w}{a^\prime b^\prime})\in S^*_{a,b,a}$ et le processus \eqref{eq:stepbystep} dans la Proposition \ref{po:pgcddonotchange} donne une suite $(x_n,y_n)\in\SABC$ dont $x_n,y_n\to \infty$.
\end{proof}

\section{Approximation asymptotique locale des nombres algébriques}\label{se:zoomalgebraic}
Dans cette section on étudie la distribution locale d'un point $\theta\in\PP^1(\RR)\setminus\PP^1(\QQ)$. 
On identifie localement $\PP^1$ avec $T_{\theta}\PP^1$ via l'application définie pour $v\neq 0$,
$$[u:v]\longmapsto \frac{u}{v}-\theta\quad ([\theta:1]\longmapsto 0).$$
On utilise la hauteur de Weil canonique associée à $\mathcal{O}(1)$:
\begin{equation}\label{eq:absoluteWeil}
H([u:v])=\max(|u|,|v|), \quad(u,v)\in\ZZ^2,\quad \pgcd(u,v)=1.
\end{equation}
On utilise la valeur absolue comme distance.
Il convient de choisir, pour tout $\varepsilon>\eta>0$, les fonction caractéristiques $\chi(\varepsilon)$ et $\chi(\varepsilon,\eta)$ définies sur $T_\theta\PP^1$ par $$\chi(\varepsilon)=\chi(\{y\in\RR:|y|\leqslant \varepsilon\}),\quad\chi(\varepsilon,\eta)=\chi(\{y\in\RR:\eta<|y|\leqslant \varepsilon\})$$ pour tester la répartition asymptotique des rationnels autour du point $\theta$. On remarque que $\chi(\varepsilon,\eta)$ est en fait la fonction caractéristique de la réunion des intervalles $\mathopen]\eta,\varepsilon\mathclose], \mathopen[-\varepsilon,-\eta\mathclose[$.
On peut faire le même pour $\chi(\varepsilon)$.

\subsection{Cas critique pour des nombres quadratiques}\label{se:criticalquadratic}
Soient $\varepsilon>0,r>0$. On rappelle les conditions \eqref{eq:Eebr1} et \eqref{eq:Eebr2}.
 Si $[\QQ(\theta):\QQ]=2$, i.e. $\theta$ est quadratique, un tel nombre est représenté de façon unique sous la forme $(a,b,P,Q\in\ZZ,aQ\neq 0)$ 
 $$\theta=\frac{P}{Q}+\sqrt{\frac{b}{a}} \text{ ou } \frac{P}{Q}-\sqrt{\frac{b}{a}},\quad \pgcd(a,b)=\pgcd(P,Q)=1.$$ 
 En vertu de \eqref{eq:Eebr1}, il suffit de considérer les nombres de la forme $\theta=\sqrt{\frac{b}{a}}$ avec $a,b\in\ZZ,a\neq 0\pgcd(a,b)=1$. 
 D'après les théorèmes de Roth et Dirichlet (Théorème \ref{th:approximationtheorem1}), on a $\aess(\theta)=\alpha(\theta,\PP^1)=\frac{1}{2}$.
 Donc on prend $r=\frac{1}{2}$ correspondant au zoom critique.
 On réécrit l'inéquation \eqref{eq:Eebr1} comme
 \begin{equation}\label{eq:Eebr11}
 \left|\frac{u}{v}-\sqrt{\frac{b}{a}}\right|\leqslant \frac{\varepsilon}{B^2}.
 \end{equation}
 On a aussi l'inégalité de Liouville
 	\begin{equation}\label{eq:liouvillelemma}
 \left|\frac{u}{v}-\theta\right|\geqslant\frac{\Xi(\theta)}{v^2},
 \end{equation} 
 où la constante $\Xi(\theta)>0$ est calculable (cf. la démonstration de la Proposition \ref{po:effectiveliouville} \emph{infra}). En combinant \eqref{eq:Eebr11} et \eqref{eq:liouvillelemma},
 \begin{equation}\label{eq:lowerboundforepsilon}
 \frac{\varepsilon}{B^2}\geqslant \frac{\Xi(\theta)}{v^2}\geqslant\frac{\Xi(\theta)}{B^2},
 \end{equation}
 d'où $\varepsilon\geqslant \Xi(\theta)>0$. C'est-à-dire quand le diamètre du voisinage auquel la fonction caractéristique correspond est suffisamment petit, il n'y a pas de points rationnels pour tout $B$. 
 Le résultat principal de cette partie est le suivant, qui démontre la partie $r=\frac{1}{2}$ du Théorème \ref{th:secondtheorem}. 
 \begin{theorem}\label{th:criticzoomforquadratics}
	Pour $C_1,C_2>0$, il existe $\varepsilon>\eta>C_1$, $\varepsilon-\eta<C_2$ tels que 
 	$$0=\liminf_{B}\delta_{\PP^1,\theta,B,\frac{1}{2}}(\chi(\varepsilon,\eta))<\limsup_{B}\delta_{\PP^1,\theta,B,\frac{1}{2}}(\chi(\varepsilon,\eta))=O_{\varepsilon,\eta}(1).$$
 	Par conséquent, il n'existe pas de mesure limite pour le zoom critique.
 \end{theorem}
Ce théorème est une conséquence des Propositions \ref{po:controlofepsilon}, \ref{po:avoidance} et \ref{po:trackpoints}.
\begin{remark}\label{rmk:Lang1}
	Dans \cite{Lang1} S. Lang a considéré aussi le dénombrement d'approximants de nombres quadratiques similaire mais différent du nôtre. 
	Plus précisément, il a démontré que
	\begin{equation*}
	\sharp\left\{
	\begin{aligned}
	(u,v)\in\ZZ^*\times \NN^*
	\end{aligned}
	\left|
	\begin{aligned}
	&\left|\frac{u}{v}-\theta\right|\leqslant  \frac{\varepsilon}{v^2}\\
	&\max(|u|,|v|)\leqslant B
	\end{aligned}
	\right\}\right. =c(\varepsilon)\log B +O(1).
	\end{equation*}
	 Ceci peut être recouverte par notre raisonnement dans la Section \ref{se:avoid} comme étant une conséquence directe des conclusions sur la structure des solutions des équations de Pell-Fermat dans la Section \ref{se:eqpellfermat}.
\end{remark}
\subsubsection{Lien avec des équations de Pell-Fermat}\label{se:thelinkbetweenfermat}
Nous trouvons d'abord le lien avec des équations de Pell-Fermat. Les équations \eqref{eq:Eebr11} et \eqref{eq:Eebr2} nous donnent 
\begin{equation}\label{eq:therangeofm}
\begin{split}
	|au^2-bv^2|&=av^2\left|\left(\frac{u}{v}+\sqrt{\frac{b}{a}}\right)\left(\frac{u}{v}-\sqrt{\frac{b}{a}}\right)\right|\\ &\leqslant av^2\left(2\sqrt{\frac{b}{a}}+\frac{\varepsilon}{B^2}\right)\frac{\varepsilon}{B^2}\\ &\leqslant 2\varepsilon \sqrt{ab}+\frac{a\varepsilon^2}{B^2}=2\varepsilon\sqrt{ab}+o(1).
\end{split}
\end{equation}
On en conclut que $|au^2-bv^2|$ ne prend qu'un nombre fini de valeurs entières et ce nombre est indépendant de $B$. 
\subsubsection{Suites  de bornes évitant les solutions}\label{se:avoid}
On suppose comme précédemment que $\theta=\sqrt{\frac{b}{a}}$. Le but de cette section est double. S'appuyant sur les résultats de la Section \ref{se:pell-fermatsimple}, on donne une borne effective pour $\delta_{\PP^1,\theta,B,r}(\chi(\varepsilon))$ et on construit des suites de bornes telles que $\delta_{\PP^1,\theta,B,r}(\chi(\varepsilon,\eta))=0$ pour toute fonction caractéristique $\chi(\varepsilon,\eta)$ suffisamment \og petite\fg. On conserve les notations dans \eqref{eq:factorise}.
\begin{proposition}\label{po:controlofepsilon}
	Pour tout $\varepsilon>\Xi(\theta)$ (cf. \eqref{eq:liouvillelemma}), on a, pour tout $B\gg_{\varepsilon}1$,
	$$\delta_{\PP^1,\theta,B,\frac{1}{2}}(\chi(\varepsilon))\leqslant 6\sum_{|m|\leqslant 2\varepsilon\sqrt{ab}+1}\tau(A^\prime m)\left(\left\lfloor\frac{\log(\varepsilon)-\log(\Xi(\theta))}{2\log(\varepsilon^*_{A^\prime B^\prime})}\right\rfloor+1\right),$$
	où $A^\prime,B^\prime$ sont définis par \eqref{eq:factorise} et $\varepsilon^*_{A^\prime B^\prime}$ par \eqref{le:samuel}.
\end{proposition}
\begin{proposition}\label{po:avoidance}
	Il existe une fonction d'escalier $G(\cdot):\RR_{>0}\to\RR_{>1}$ qui est décroissante et semi-continue inférieurement telle que pour tout $\varepsilon>0$ et tout $\eta\in\mathopen]\frac{\varepsilon}{G(\varepsilon)},\varepsilon\mathclose[$, on puisse choisir une suite de bornes $(B_n)$ qui tend vers $\infty$ de manière que
	$$\delta_{\PP^1,\theta,B_n,\frac{1}{2}}(\chi(\varepsilon,\eta))=0.$$
\end{proposition}
\begin{proof}[Démonstration des Propositions \ref{po:controlofepsilon} \& \ref{po:avoidance}]
		On rappelle d'abord la factorisation \eqref{eq:factorise}.
		Une première observation c'est qu'on a une bijection de $\SABC$ sur le sous-ensemble $\{(x,y)\in S^*_{1,A^\prime B^\prime,A^\prime c}:A^\prime a^\prime|x,b^\prime|y\}$ de $S^*_{1,A^\prime B^\prime,A^\prime c}$, ce qui nous ramène aux cas où $a=1$ pour borner le nombre de solutions de \eqref{eq:diophantineequation} ainsi que $\delta_{\PP^1,\theta,B,\frac{1}{2}}(\chi(\varepsilon))$. 
		Pour $\varepsilon>\eta>0$ et pour chaque entier $m$, on note $$\varepsilon^\prime=\frac{A^\prime a^\prime}{b^\prime}\varepsilon,\quad \eta^\prime=\frac{A^\prime a^\prime}{b^\prime}\eta,$$ et $E^\prime(\varepsilon,\eta,B)_m$ l'ensemble des points $(u,v)\in \ZZ^2$ vérifiant les conditions 
\begin{equation}\label{eq:Eebc3}
\eta^\prime B^{-2}<\left|\frac{u}{v}-\sqrt{D}\right|\leqslant \varepsilon^\prime B^{-2};
\end{equation}
\begin{equation}\label{eq:Eebc4}
u^2-v^2D=A^\prime m.
\end{equation}
L'encadrement \eqref{eq:boundforv} suivant montre que $E^\prime(\varepsilon,\eta,B)_m$ est fini.
Pour démontrer la Proposition \ref{po:avoidance}, il suffit de démontrer qu'il existe une suite $(B_n)$ telle que $\forall n,B_n^2\geqslant a\varepsilon^2$, et $\sharp E^\prime(\varepsilon,\eta,B_n)_m=0$ puisque
\begin{equation*}
\delta_{\PP^1,\theta,B,\frac{1}{2}}(\chi(\varepsilon,\eta))\leqslant\sum_{|m|\leqslant 2\varepsilon\sqrt{ab}+1}\sharp E^\prime(\varepsilon,\eta,B)_m,
\end{equation*}
pour tout $B^2\geqslant a\varepsilon^2$ en vertu de \eqref{eq:therangeofm}.
D'après \eqref{eq:lowerboundforepsilon}, on a que pour $\eta<\Xi(\theta)$, 
$$\delta_{\PP^1,\theta,B,\frac{1}{2}}(\chi(\varepsilon,\eta))=\delta_{\PP^1,\theta,B,\frac{1}{2}}(\chi(\varepsilon)).$$
Alors pour démontrer la Proposition \ref{po:controlofepsilon} il suffit de borner $\sharp E^\prime(\varepsilon,\eta,B)_m$ pour tout $m\leqslant 2\sqrt{ab}+1$.

On commence par démontrer la Proposition \ref{po:controlofepsilon}.
L'équation \eqref{eq:Eebc4} nous amène à un problème concernant des solutions des équations de Pell-Fermat.
On extrait un encadrement pour $v$. 
D'une part,
$$\frac{A^\prime |m|}{v^2}=\frac{|u^2-v^2D|}{v^2}=\left|\frac{u}{v}-\sqrt{D}\right|\left|\frac{u}{v}+\sqrt{D}\right|\leqslant \frac{\varepsilon^\prime}{B^2}\left(2\sqrt{D}+\frac{\varepsilon^\prime}{B^2}\right),$$
d'où
$$ v^2\geqslant \frac{A^\prime |m|}{\varepsilon^\prime(2\sqrt{D}+\frac{\varepsilon^\prime}{B^2})}B^2.$$
D'autre part, 
$$\frac{A^\prime |m|}{v^2}=\frac{|u^2-v^2D|}{v^2}\geqslant\frac{\eta^\prime}{B^2}\left(2\sqrt{D}-\frac{\varepsilon^\prime}{B^2}\right),$$
d'où
$$v^2\leqslant \frac{A^\prime |m|}{\eta^\prime(2\sqrt{D}-\frac{\varepsilon^\prime}{B^2})}B^2.$$
En résumé, on a 
\begin{equation}\label{eq:boundforv}
\sqrt{\frac{A^\prime |m|}{\varepsilon^\prime(2\sqrt{D}+\frac{\varepsilon^\prime}{B^2})}}B\leqslant |v|\leqslant  \sqrt{\frac{A^\prime |m|}{\eta^\prime(2\sqrt{D}-\frac{\varepsilon^\prime}{B^2})}}B.
\end{equation}
En reportant dans \eqref{eq:Eebc3}, on a
$$|u-v\sqrt{D}|\leqslant \frac{\varepsilon^\prime |v|}{B^2}=O_{\varepsilon,\eta}\left(\frac{1}{B}\right).$$
Donc
\begin{align*}
	|u+v\sqrt{D}|\leqslant |u-v\sqrt{D}|+2|v|\sqrt{D}
	=2|v|\sqrt{D}+O_{\varepsilon,\eta}\left(\frac{1}{B}\right)
	\leqslant  2\sqrt{\frac{A^\prime |m|D}{\eta^\prime(2\sqrt{D}-\frac{\varepsilon^\prime}{B^2})}}B+O_{\varepsilon,\eta}\left(\frac{1}{B}\right).
\end{align*}
De la même manière
\begin{align*}
|u+v\sqrt{D}|\geqslant 2|v|\sqrt{D}-|u-v\sqrt{D}|=2|v|\sqrt{D}+O_{\varepsilon,\eta}\left(\frac{1}{B}\right)
\geqslant  2\sqrt{\frac{A^\prime |m|D}{\varepsilon^\prime(2\sqrt{D}+\frac{\varepsilon^\prime}{B^2})}}B+O_{\varepsilon,\eta}\left(\frac{1}{B}\right).
\end{align*}
En prenant le logarithme, il en résulte que
\begin{equation}\label{boundforloguv1}
\begin{split}
\frac{\log|u+v\sqrt{D}|}{\log(\VDT)}\leqslant \frac{1}{\log(\VDT)}\left(\frac{1}{2}\log(4A^\prime |m|D)-\frac{1}{2}\log(2\eta^\prime\sqrt{D})+\log B\right)+O_{\varepsilon,\eta}\left(\frac{1}{B^2}\right)\\
\frac{\log|u+v\sqrt{D}|}{\log(\VDT)}\geqslant \frac{1}{\log(\VDT)}\left(\frac{1}{2}\log(4A^\prime |m|D)-\frac{1}{2}\log(2\varepsilon^\prime\sqrt{D})+\log B\right)+O_{\varepsilon,\eta}\left(\frac{1}{B^2}\right)
\end{split}
\end{equation}
Donc \eqref{boundforloguv1} entraîne une borne pour le nombre de solutions $(u,v)\in\ZZ^2$ vérifiant \eqref{eq:Eebc3} dans toute famille de solutions de l'équation 
\begin{equation}\label{eq:familyofpellequation}
u^2-Dv^2=A^\prime m, \quad |m|\leqslant 2\varepsilon\sqrt{ab}+1.
\end{equation}
Rappelons que le groupe $\{(\varepsilon_D^*)^n\}$ agit librement sur l'ensemble de solutions de \eqref{eq:familyofpellequation} pour $m$ fixé. Prenons une orbite $J=\{\pm U_J(\VDT)^n,n\in\ZZ\}$, où $U_J\in\ZZD$ est un représentant.
Considérons l'application $\Phi_D$ définie par $$\Phi_D(u,v)=\frac{\log(|u+v\sqrt{D}|)}{\log(\VDT)}.$$
L'image de $J$ 
$$\frac{\log(J)}{\log(\VDT)}=\left\{\frac{\log|U_J|}{\log(\VDT)}+n,n\in\ZZ\right\}$$
est un translaté de $\ZZ$ dans $\RR$.  On a donc (compte-tenu du signe) pour $B\gg_{\varepsilon,\eta}1$,
$$\sharp (J\cap E^\prime(\varepsilon,\Xi(\theta),B)_m)\leqslant 2\left(\left\lfloor\frac{\log(\varepsilon)-\log(\Xi(\theta))}{2\log(\varepsilon^*_{D})}\right\rfloor+1\right).$$
La Proposition \ref{co:structuretheorem1} donne une majoration du nombre d'orbites, on arrive donc, en sommant sur $m$, à la borne dans la Proposition \ref{po:controlofepsilon}.
	
Maintenant nous démontrons la Proposition \ref{po:avoidance}. Pour tout $\varepsilon$ fixé, l'ensemble des images de toutes les solutions de (\ref{eq:familyofpellequation}) par l'application $\Phi_D$, noté $\mathcal{P}(\varepsilon)$, est une réunion finie de réseaux de même période dans $\RR$. Donc la fonction $F:\RR_{>0}\to\RR$ définie par
	\begin{equation}\label{eq:functionF}
	F(\varepsilon)=\begin{cases}
	1 & \text{ si } \mathcal{P}(\varepsilon)=\varnothing;\\
\min_{\lambda,\beta\in\Phi_D( \mathcal{P}(\varepsilon))} |\lambda-\beta|  &\text{ sinon.}
	\end{cases}
	\end{equation}
	est évidemment décroissante et semi-continue inférieurement. Elle est constante sur tout intervalle
	$$\left[\frac{n-1}{2\sqrt{ab}},\frac{n}{2\sqrt{ab}}\right[,\quad (n\in\NN).$$
	On note $I(\varepsilon,\eta)$ l'intervalle
	$$\left[\frac{1}{\log(\VDT)}\left(\frac{1}{2}\log(4A^\prime |m|D)-\frac{1}{2}\log(2\varepsilon^\prime\sqrt{D})\right),\frac{1}{\log(\VDT)}\left(\frac{1}{2}\log(4A^\prime |m|D)-\frac{1}{2}\log(2\eta^\prime\sqrt{D})\right)\right].$$
Si l'on prend $\eta,\varepsilon$ assez proches tels que
$$0<\frac{\log(\varepsilon)-\log(\eta)}{2\log(\VDT)}=\frac{\log(2\varepsilon^\prime\sqrt{D})-\log(2\eta^\prime\sqrt{D})}{2\log(\VDT)}<F(\varepsilon),$$
à savoir \begin{equation}
\label{eq:functionG}
\frac{\varepsilon}{G(\varepsilon)}<\eta<\varepsilon, \quad G(\varepsilon)=(\VDT)^{2F(\varepsilon)},
\end{equation}
on peut choisir $(B_n)\to\infty$ de sorte qu'il existe une constant $H(\varepsilon)>0$ tel que l'intervalle
$I_n=I(\varepsilon,\eta)+\frac{\log (B_n)}{\log(\VDT)}$
vérifie $$\min_{\lambda\in\Phi_D(\mathcal{P}(\varepsilon))}\operatorname{dist}(\lambda,I_n)>H(\varepsilon)>0.$$
D'après (\ref{boundforloguv1}), ceci démontre que $\sharp E^\prime(\varepsilon,\eta,B_n)_m=0$ pour tout $n\geqslant 1,|m|\leqslant2\varepsilon\sqrt{ab}+1$. La fonction $G$ vérifie les même propriétés que celles de $F$, d'où la Proposition \ref{po:avoidance}.
\end{proof}
\subsubsection{Construction de suites de bornes avec des solutions}
En imitant cette méthode, quand la fonction caractéristique choisie permet d'avoir une solution d'une équation Pell-Fermat généralisée, on peut choisir la suite $(B_n)$ de la façon qu'elle donne au moins une famille de solutions. 

\begin{proposition}\label{po:trackpoints}
	Soit $\theta=\sqrt{\frac{b}{a}}>1$ avec $\pgcd(a,b)=1$. Soient $\varepsilon>\eta>0$, tel que
  l'intervalle $\mathopen ]2\sqrt{ab}\theta^{-2}\eta,2\sqrt{ab}\theta^{-2}\varepsilon\mathclose[$ contienne la valeur absolue d'un entier $m$ de sorte que $S^*_{a,b,m}$ contienne une solution primitive (cf. les notations au début de la Section \ref{se:theoryalgebraic}).
  Alors il existe une suite de bornes $(B_n)\to\infty$ tel que 
  $$\delta_{\PP^1,\theta,B_n,\frac{1}{2}}(\chi(\varepsilon,\eta))\geqslant 1, \quad \forall n\geqslant 1.$$
\end{proposition}
\begin{proof}
 On choisit $\lambda\in\mathopen]0,1\mathclose[$ tel que $|m|\in\mathopen] 2\sqrt{ab}\theta^{-2}\eta,2\sqrt{ab}\theta^{-2}\varepsilon \lambda^2\mathclose[$. On suppose dans la suite que $m>0$, la démonstration étant analogue lorsque $m<0$.
 Soit $B>0$ tel que
 \begin{equation}\label{eq:lowerboundform}
 (a\lambda^2-(2\varepsilon)^{-1}\theta m)B^2>m> \frac{2\sqrt{ab}\eta}{\theta^2}+\frac{a\varepsilon\eta}{\theta^2B^2}.
 \end{equation}
 On définit l'ensemble $F(\lambda,B,m)$ des $(u,v)\in\NN^2$ primitifs vérifiant 
\begin{equation}\label{eq:rangeofu}
 au^2-bv^2= m \text{ et }
 \lambda B\leqslant u\leqslant B.
\end{equation}
 Pour $(u,v)\in F(\lambda,B,m)$, on a, d'après \eqref{eq:rangeofu},
 $$v^2=\frac{au^2-m}{b}\leqslant \frac{a}{b}B^2-\frac{m}{b}<\frac{a}{b}B^2.$$
 En revanche, grâce à \eqref{eq:rangeofu} et à la deuxième inégalité de \eqref{eq:lowerboundform}, 
 $$v^2\geqslant \frac{a}{b}\lambda^2 B^2-\frac{m}{b}>\frac{m}{2\sqrt{ab}\varepsilon}B^2.$$
 On note $X=\frac{u}{v}-\theta$. Alors la positivité de $m$ implique celle de $X$. Notons que $\frac{u}{v}+\theta=X+2\theta$, on obtient, d'après l'hypothèse sur $m$,
 $$X=\frac{m}{av^2(X+2\theta)}<\frac{m}{2\theta av^2}<\frac{\varepsilon}{B^2}.$$
 Cela entraîne aussi la majoration
 $$X+2\theta \leqslant 2\theta +\frac{\varepsilon}{B^2},$$
 ainsi que la minoration pour $X$:
 $$X= \frac{m}{av^2(X+2\theta)}\geqslant \frac{m}{a\theta^{-2}B^2\left(2\theta+\frac{\varepsilon}{B^2}\right)}>\frac{\eta}{B^2}$$
 par la première inégalité de \eqref{eq:lowerboundform}. Cela montre que $$\delta_{\PP^1,\theta,B,\frac{1}{2}}(\chi(\varepsilon,\eta))\geqslant 
\sharp F(\lambda,B,m).$$
 On prend $(u_0,v_0)\in \NN^2$ une solution primitive de l'équation
 $$ax^2-by^2=m.$$ On définit $((u_n,v_n))_{n\in\NN}$ comme dans le Corollaire \ref{rmk:inftysol}. Comme $u_n\to \infty$, on peut choisir $(B_n)\to\infty$ de la manière que
$$\lambda B_n< u_n\leqslant B_n,$$
Donc pour tout $n$ suffisamment grand, on a $(u_n,v_n)\in F(\lambda,B_n,m)$,
d'où 
\begin{equation*}
\delta_{\PP^1,\theta,B_n,\frac{1}{2}}(\chi(\varepsilon,\eta))\geqslant \sharp F(\lambda,B_n,m)\geqslant 1.\qedhere
\end{equation*}
\end{proof}
\begin{proof}[Démonstration du Théorème \ref{th:criticzoomforquadratics}]
	Pour $C_1,C_2>0$ fixés, on peut choisir $m\in\NN$ tel que 
	$$\frac{m\theta^2}{2\sqrt{ab}}>2C_1,$$ et que $S^*_{a,b,m}$ contienne une solution primitive (voir le Corollaire \ref{rmk:inftysol}). Rappelons les propriétés des fonctions $F$ \eqref{eq:functionF} et $G$ \eqref{eq:functionG} dans la démonstration de la Proposition \ref{po:avoidance}. Il existe alors $\delta_0>0$ tel que 
	$$G(x)=G\left(\frac{m\theta^2}{2\sqrt{ab}}\right)\leqslant|\varepsilon_D^*|^2, \quad \forall x\in\left[\frac{m\theta^2}{2\sqrt{ab}},\frac{m\theta^2}{2\sqrt{ab}}+\delta_0\right[.$$
	On définit 
	$$\delta_1=\min\left(\frac{m\theta^2}{2\sqrt{ab}}\left(G\left(\frac{m\theta^2}{2\sqrt{ab}}\right)-1\right), \frac{C_2}{2},\delta_0,C_1\right)>0.$$
	Maintenant on prend 
	$$\varepsilon=\frac{m\theta^2}{2\sqrt{ab}}+\frac{\delta_1}{2},\quad \eta\in\left]\frac{m\theta^2}{2\sqrt{ab}}-\frac{\delta_1}{2|\varepsilon_D^*|^2},\frac{m\theta^2}{2\sqrt{ab}}\right[.$$
	Puisque
	\begin{align*}
		\frac{\varepsilon}{G(\varepsilon)}=\frac{\frac{m\theta^2}{2\sqrt{ab}}+\delta_1}{G(\frac{m\theta^2}{2\sqrt{ab}})}-\frac{\delta_1}{2G(\frac{m\theta^2}{2\sqrt{ab}})}
		\leqslant \frac{m\theta^2}{2\sqrt{ab}}-\frac{\delta_1}{2G(\frac{m\theta^2}{2\sqrt{ab}})}
		\leqslant\frac{m\theta^2}{2\sqrt{ab}}-\frac{\delta_1}{2|\varepsilon_D^*|^2}<\eta,
	\end{align*}
	Le couple $(\varepsilon,\eta)$ vérifie toutes les hypothèses des Propositions \ref{po:avoidance}, \ref{po:trackpoints}.
	Compte tenu de la majoration fournie par la Proposition \ref{po:controlofepsilon}, il en découle donc que 
\begin{equation*}
	\liminf_{B} \delta_{\PP^1,Q,B,\frac{1}{2}}(\chi(\varepsilon,\eta))=0<1\leqslant \limsup_{B} \delta_{\PP^1,Q,B,\frac{1}{2}}(\chi(\varepsilon,\eta))\leqslant \limsup_{B}\delta_{\PP^1,Q,B,\frac{1}{2}}(\chi(\varepsilon))=O_{\varepsilon}(1).\qedhere
\end{equation*}
\end{proof}
\subsection{Zoom sous-critique}\label{se:weakzoomofalgebraicnumbers}
La principale occupation de cette section est la partie $r>\frac{1}{2}$ du Théorème \ref{th:secondtheorem}. On va prouver un résultat sur la distribution locale en un point réel vérifiant des inégalités de type analogue à celle de Liouville, qui \textit{a fortiori} s'applique aux nombres algébriques. Pour cela on rappelle des notions standard d'approximation diophantienne des nombres réels. 
\begin{definition}\label{def:irrationality}
	Soit $t\in\RR_{>0}$. On dit qu'un nombre réel $\theta$ a \textit{un ordre d'irrationalité} $t$ s'il existe une constante $C=C(\theta,t)>0$ telle que 
\begin{equation}\label{eq:Ct}
	\left|\theta-\frac{m}{n}\right| \geqslant \frac{C}{n^t} \quad \forall (m,n)\in\ZZ\times \NN^*.
\end{equation}
On désigne par $\dist(\theta,\ZZ)=\min(\{\theta\},1-\{\theta\})$. 
	\textit{La mesure d'irrationalité} $m(\theta)$ est la borne inférieure de l'ensemble des ordres d'irrationalité de $\theta$. De manière équivalente, elle est égale au supremum des nombres réels $s$ tels que 
	$$\liminf_{q\in\NN}q^{s-1}\dist(q\theta,\ZZ)=0.$$
	Si $\theta$ est algébrique, $m(\theta)$ est rien d'autre que la constante d'approximation $\alpha(\theta,\PP^1)$ associée au fibré $\mathcal{O}(1)$.
\end{definition}
Résumons les théorèmes classiques suivants (voir par exemple \cite[Chapter 1]{Bugeaud} pour des détails). 
\begin{theorem}[Dirichlet, Liouville, Roth]
	Le nombre $2$ est un ordre d'irrationalité et la mesure d'irrationalité d'un nombre quadratique. Pour tout $\varepsilon>0$, le nombre $2+\varepsilon$ est un ordre d'irrationalité d'un nombre algébrique de degré $\geqslant 3$ dont la mesure d'irrationalité est $2$.
\end{theorem}
Le résultat principal de cette section est le suivant, qui est valide pour tout nombre algébrique de degré $\geqslant 2$. Pour ceux dont le degré est $2$, on a un meilleur contrôle du terme d'erreur.
\begin{theorem}\label{th:localdistributionofrealnumbers}
	Soient $\alpha$ un nombre réel et $r>\frac{1}{2}$. Alors pour tout $\varepsilon_1>\varepsilon_2\geqslant 0$ et $\tau>0$,
	\begin{itemize}
		\item si $ \frac{1}{2}<r\leqslant 1$, supposons que $m(\alpha)=2$, on a
	\begin{equation}\label{eq:weakalpha}
		\delta_{\PP^1,\alpha,B,r}(\chi(\varepsilon_1,\varepsilon_2))=
		B^{2-\frac{1}{r}}\frac{3}{\pi^2 \sup(1,\alpha^2)}\int\chi(\varepsilon_1,\varepsilon_2) \operatorname{d}x+O_{\alpha,\varepsilon_i,\tau}(B^{1-\frac{1}{2r}+\tau}).
	\end{equation}
	Si de plus $2$ est un ordre d'irrationalité de $\alpha$, à savoir $\exists \Xi(\alpha)\in\mathopen]0,1\mathclose[$ tel que pour tout nombre rationnel $\frac{p}{q}$ on ait
	\begin{equation}\label{eq:LiouvillealphaXi}
	\left|\alpha-\frac{p}{q}\right|\geqslant \frac{\Xi(\alpha)}{q^2},
	\end{equation}
	on a, en notant
	\begin{equation}\label{eq:deltalpha}
	\varDelta(\alpha)=\frac{\Xi(\alpha)^{-1}}{\log(\Xi(\alpha)^{-1})},
	\end{equation}
\begin{equation}\label{eq:strongalpha}
\begin{split}
	\delta_{\PP^1,\alpha,B,r}(\chi(\varepsilon_1,\varepsilon_2))&= B^{2-\frac{1}{r}}\frac{3}{\pi^2 \sup(1,\alpha^2)}\int\chi(\varepsilon_1,\varepsilon_2) \operatorname{d}x\\
	&\quad+O_{\varepsilon_i}\left(\varDelta(\alpha)\Xi(\alpha)^{-\frac{1}{2}}B^{1-\frac{1}{2r}}\log (B)\log (\Xi(\alpha)^{-\frac{1}{2}}B^{1-\frac{1}{2r}})\right).
	\end{split}
\end{equation}
\item si $r>1$, on a
$$\delta_{\PP^1,\alpha,B,r}(\chi(\varepsilon_1,\varepsilon_2))=B^{2-\frac{1}{r}}\frac{3}{\pi^2 \sup(1,\alpha^2)}\int\chi(\varepsilon_1,\varepsilon_2) \operatorname{d}x+O_{|\alpha|,\varepsilon_i}(B\log B).$$
	\end{itemize}
\end{theorem}
\begin{remark*}
	La dépendance du terme d'erreur dans \eqref{eq:weakalpha} sur $\alpha$ et $\tau$ est intrinsèque en un sens similaire au théorème de Roth. En effet, elle dépende des constantes $C(\alpha,t)$ dans \eqref{eq:Ct} avec $t$ proche de $2$ (cf. aussi \eqref{eq:boundford} \textit{infra}). Alors que le terme d'erreur dans \eqref{eq:strongalpha} ne dépende que de la taille de $\alpha$.
\end{remark*}
\begin{remark}\label{rmk:Lang2}
	Il convient de noter qu'un dénombrement similaire dans la direction de l'approximation diophantienne métrique fut considéré notamment par Erd\H{o}s, Lang, Leveque, Schmidt. Un cas particulier est le suivant (voir par exemple \cite[II \S3 Theorem 8]{Lang2}). Soit $r>\frac{1}{2}$, alors si $m(\alpha)=2$ (ceci implique que $\alpha$ est du type $g(t)=t^\delta$ pour tout $\delta>0$, cf. \cite[p. 20]{Lang2})
	$$\sharp\left\{(p,q)\in\NN^{*2}:\frac{\varepsilon_2}{q^{\frac{1}{r}}}<\alpha-\frac{p}{q}\leqslant \frac{\varepsilon_1}{q^{\frac{1}{r}}},1\leqslant q\leqslant B\right\}\sim B^{2-\frac{1}{r}} \left(2-\frac{1}{r}\right)^{-1}\int\chi(\varepsilon_1,\varepsilon_2) \operatorname{d}x.$$
	Observons que la constante $\left(2-\frac{1}{r}\right)^{-1}$ n'apparaît pas dans notre résultat \eqref{eq:strongalpha}.
\end{remark}

Tout d'abord on reformule ce problème de zoom en un problème de dénombrement.
Pour $\varepsilon_1>\varepsilon_2\geqslant 0,B>0$, on définit $S(\varepsilon_1,\varepsilon_2,B)$ l'ensemble des point $(u,v)\in\ZZ\times\NN^*$ vérifiant 
\begin{equation}\label{eq:thesetS}
\pgcd(u,v)=1,\quad \varepsilon_2B^{-\frac{1}{r}}<\left|\frac{u}{v}-\alpha\right|\leqslant \varepsilon_1 B^{-\frac{1}{r}},\quad \max(|u|,|v|)\leqslant B 
\end{equation}
\begin{figure}[h]
	\centering
	\includegraphics[scale=0.7]{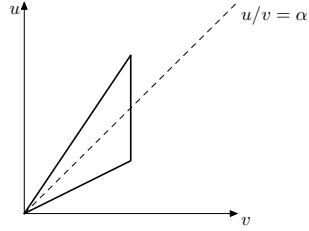}
	\caption{La région triangulaire}
	\label{fig:area}
\end{figure}
Graphiquement, si l'on prend $\varepsilon_2=0$, c'est-à-dire on prend la fonction de test $\chi(\varepsilon)$ et on calcule $\delta_{\PP^1,\alpha,B,r}(\chi(\varepsilon))$, on compte des points entiers primitifs à l'intérieur du triangle dont l'aire est d'ordre de grandeur $B^{2-\frac{1}{r}}$ et celle de la longueur du bord est $B$ (cf. Figure \ref{fig:area}). Donc la comparaison classique avec l'aire du triangle n'est utilisable que pour les cas où $r>1$. Ceux où $r\leqslant 1$ nécessitent un travail supplémentaire. Notre approche s'appuie sur la théorie d'équirépartition modulo $1$ initiée dans \cite{Pagelot}. 

	\begin{definition}\label{de:discrépance}
		Soit $(x_n)$ une suite de nombres réels dans $[0,1[$. 
		Soit $N>0$. La \textit{discrépance} (à  $N$) $D_{(x_n)}(N)$ de cette suite est définie par 
		$$  D_{(x_n)}(N)=\sup _{\lambda\in\mathclose[0,1\mathclose]}\left|\frac{\sharp\{n\leqslant N:x_n\in\mathopen[0,\lambda\mathclose]\}}{N}-\lambda\right|.$$
		Si la suite $(x_n)=(\{n\alpha\})$ pour un nombre $\alpha$ réel fixé, on notera $D_\alpha(N)=D_{(n\alpha)}(N)$.
	\end{definition}
	On rappelle l'inégalité de Koksma-Denjoy (voir, par exemple, \cite{K-N}  p. 143).
	\begin{theorem}[Koksma-Denjoy]\label{th:koksma-denjoy}
		Soient $(x_n)$ une suite de nombres réels dans $[0,1[$ et $N\geqslant 1$. Soit $\phi$ une fonction mesurable à variation bornée définie sur $\mathopen[0,1\mathclose]$ (on note $V(\phi)$ la variation totale de $\phi$). Alors
		$$\left|\frac{1}{N}\sum_{n=1}^{N}\phi(x_n)-\int_{0}^{1}\phi(t)\operatorname{d}t\right|\leqslant V(\phi)D_{(x_n)}(N).$$
	\end{theorem}
\begin{remark*}
	Une suite $(x_n)$ de nombres réels dans $[0,1[$ est dite \textit{équirépartie modulo $1$} si pour tout $0\leqslant a<b\leqslant 1$, on a 
	$$\lim\limits_{N\to\infty}\frac{1}{N}\sharp \{1\leqslant n\leqslant N: x_n\in\mathopen[a,b\mathclose [ \}=b-a.$$
	En effet ceci revient à dire que $D_{(x_n)}(N)=o(N)$ pour $N\to\infty$ (cf. \cite[Corollary 1.1]{K-N}). De manière équivalente, pour toute fonction $\phi$ à variation bornée sur $\mathopen[0,1\mathclose]$, on a 
	$$\frac{1}{N}\sum_{n=1}^{N}\phi(x_n)\longrightarrow\int_{0}^{1}\phi(t)\operatorname{d}t,\quad N\to\infty.$$
	Pour tout nombre irrationnel $\theta$, la suite $(\{n\theta\})$ est équirépartie modulo $1$ (cf. \cite[Example 2.1]{K-N}).
\end{remark*}
 L'inégalité de Erd\H{o}s-Tur\'an donne un contrôle de la discrépance (cf.  \cite{K-N}  p. 122-123) pour les nombres irrationnels dont la mesure d'irrationalité est finie.
	\begin{theorem}[Erd\H{o}s-Tur\'an]\label{co:controlforthediscrepancy}
				Soit $\vartheta$ un nombre irrationnel ayant un ordre d'irrationalité $t>1$. Alors pour tout $\sigma>0$, on a $$D_{\vartheta}(N)=O_{\sigma,\vartheta,t}(N^{-\frac{1}{t-1}+\sigma}).$$
	\end{theorem}

\begin{definition}
	Pour un nombre réel $\vartheta$, écrivons $\vartheta=[a_0;a_1,\cdots],a_i\in\NN ,\forall i\geqslant 0$ son \emph{expansion en fraction continue} (cf. par exemple \cite[\S1.2]{Bugeaud}). Les nombres $a_i,i\geqslant 0$ sont appelés \emph{quotients partiels}. 
\end{definition}
Si les quotients partiels dans l'expansion en fraction continue d'un nombre irrationnel sont bornés, de manière équivalente, l'inégalité de Liouville \eqref{eq:LiouvillealphaXi} étant vérifiée, alors une meilleure majoration de la discrépance existe.
\begin{theorem}[\cite{K-N}, Theorem 3.4, p. 125]
	Soient $\vartheta=[a_0;a_1,\cdots]$ un nombre irrationnel. Supposons que les quotients partiels de $\vartheta$ sont bornés, i.e. il existe $M>0$ tel que $a_i\leqslant M,\forall i\geqslant 0$. Alors
	$$ND_\vartheta(N)\leqslant 3+\left(\frac{1}{\log\xi}+\frac{M}{\log(M+1)}\right) \log N,$$
	où
	$\xi=(1+\sqrt{5})/2$.
\end{theorem}
\begin{lemma}\label{le:boundforboundquotients}
	Soit $\vartheta$ un nombre irrationnel. Supposons qu'il existe $C(\vartheta)\in\mathopen]0,1\mathclose[$ tel que pour tout nombre rationnel $\frac{p}{q}$, on ait
	\begin{equation}\label{eq:liouvillegeneral}
	\left|\vartheta-\frac{p}{q}\right|\geqslant \frac{C(\vartheta)}{q^2}.
	\end{equation} 
	Soit $\vartheta=[a_0;a_1,\cdots]$ son expansion en fraction continue. Alors pour tout $i\geqslant 1$, $$a_i\leqslant (C(\vartheta))^{-1}.$$
	Par conséquent,$$ND_\vartheta(N)=O\left(\frac{C(\vartheta)^{-1}}{\log(C(\vartheta)^{-1})}\log N\right).$$
\end{lemma}
\begin{proof}
	Soient $(\frac{p_n}{q_n})_{n=-1}^\infty$ les convergents (principaux) de $\vartheta$ (cf. \cite[Definition 1.2]{Bugeaud}, \cite[p. 4]{Khinchin}). Alors ils vérifient, pour tout $k\geqslant 1$, la règle de récurrence (voir \cite[Theorem 1]{Khinchin})
	$$p_k=a_kp_{k-1}+p_{k-2};$$
	$$q_k=a_kq_{k-1}+q_{k-2},$$
	et (voir \cite[Theorem 9]{Khinchin}) pour tout $k\geqslant 0$,
	$$\left|\vartheta-\frac{p_k}{q_k}\right|<\frac{1}{q_kq_{k+1}}.$$
	Il en découle que
	$q_{k+1}>a_{k+1}q_k$ et donc
	$$\left|\vartheta-\frac{p_k}{q_k}\right|<\frac{1}{a_{k+1} q_k^2}.$$
	En reportant dans \eqref{eq:liouvillegeneral}, on obtient que pour $k\geqslant 0$,
\begin{equation*}
	a_{k+1}\leqslant (C(\vartheta))^{-1}.\qedhere
\end{equation*}
\end{proof}
		Avant de poursuivre le raisonnement pour le Théorème \ref{th:localdistributionofrealnumbers}, on établit d'abord les formules asymptotiques suivantes un peu plus générales faisant intervenir une constante fixée $K$, qui permet de l'appliquer à une hauteur équivalente à celle donnée. Pour une utilisation ultérieure, on précise la dépendance en $K$ du terme d'erreur.
\begin{proposition}\label{po:smallepsilon}
	Soit $\alpha\in\RR_{>1}$ tel que $m(\alpha)=2$. Soient $B\geqslant 1$ et $\frac{1}{2}< r\leqslant 1$. Fixons $d\in \NN^*,K>0$ et $\varepsilon>\eta\geqslant 0$. 
	Supposons que
	\begin{equation}\label{eq:epsiloncondition12}
	(\varepsilon-\eta)KB^{1-\frac{1}{r}}<\frac{\alpha}{4}.
	\end{equation}
	On définit l'ensemble $	T_K(\varepsilon,\eta,d,B)$ des $(u,v)\in\ZZ\times\NN^*$ satisfaisant à
	$v\leqslant \frac{K}{\alpha d}B $
	ainsi que
	\begin{equation}\label{eq:epsilonetapositive}
	\eta B^{-\frac{1}{r}}<\frac{u}{v}-\alpha\leqslant \varepsilon B^{-\frac{1}{r}}.
	\end{equation}
	Alors pour tout $\sigma>0$ et pour tout $N>\max(1,\frac{4\eta K}{\alpha d})$,
	\begin{equation}\label{eq:Tsmallepsilon}
	\sharp T_K(\varepsilon,\eta, d,B)=\frac{(\varepsilon-\eta)K^2}{2\alpha^2 d^2}B^{2-\frac{1}{r}}+O_{\varepsilon,\eta}\left(\frac{K^2B^{2-\frac{1}{r}}}{N d^2}\right)+O_{\sigma}\left(\frac{K^\sigma B^\sigma N}{d^\sigma}\right).
	\end{equation}
	Si de plus \eqref{eq:LiouvillealphaXi} est vérifiée, 
	on a (rappelons $\varDelta(\alpha)$ \eqref{eq:deltalpha})
	\begin{equation}\label{eq:Tstrongalpha}
	\sharp T_K(\varepsilon,\eta, d,B)=\frac{(\varepsilon-\eta)K^2}{2\alpha^2 d^2}B^{2-\frac{1}{r}}+O_{\varepsilon,\eta}\left(\frac{K^2 B^{2-\frac{1}{r}}}{N d^2}\right)+O\left(\varDelta(\alpha)N\log\left(KB\right)\right).
	\end{equation}
	Les formules ci-dessus restent valides si l'on remplace \eqref{eq:epsilonetapositive} par 
	\begin{equation}\label{eq:epsilonetanegative}
	-\varepsilon B^{-\frac{1}{r}}\leqslant \frac{u}{v}-\alpha< -\eta B^{-\frac{1}{r}}.
	\end{equation}
\end{proposition}

\begin{proof}
	Soit $N>\max(1,\frac{4\eta K}{\alpha d})$ fixé dans la suite. On peut supposer que $N$ est un entier (quitte à rajouter des constantes absolues dans les termes d'erreur).
	On considère le découpage de l'intervalle $\mathopen]0,\frac{K}{\alpha d}B \mathclose ]$ en les intervalles
	\begin{equation}\label{eq:thedivisionofv}
	\left]\frac{(k-1)KB}{\alpha d N},\frac{kKB}{\alpha d N}\right],\quad (1\leqslant k\leqslant N).
	\end{equation}
	Soit $(u_0,v_0)\in \ZZ\times \NN^*$ tel que
	$$\frac{(k-1)KB}{\alpha d N}<v_0\leqslant\frac{kKB}{\alpha d N} \text{ et }\eta B^{-\frac{1}{r}}<\frac{u_0}{v_0}-\alpha\leqslant \varepsilon B^{-\frac{1}{r}},$$
	on a
	$$\eta v_0 B^{-\frac{1}{r}}>\frac{\eta (k-1)K}{\alpha d N} B^{1-\frac{1}{r}},\quad \varepsilon v_0 B^{-\frac{1}{r}}\leqslant \frac{\varepsilon k K}{\alpha d N}B^{1-\frac{1}{r}}.$$
	et donc 
	$$\frac{\eta K(k-1) }{\alpha d N}B^{1-\frac{1}{r}}< u_0-\alpha v_0 \leqslant \frac{\varepsilon kK}{\alpha d N}B^{1-\frac{1}{r}}.$$
	D'après la condition \eqref{eq:epsiloncondition12} et le choix de $N$,
	$$\frac{\varepsilon kK}{\alpha d N}B^{1-\frac{1}{r}}-\frac{\eta (k-1)K}{\alpha d N}B^{1-\frac{1}{r}}=\frac{(\varepsilon-\eta)kK}{\alpha  d N}B^{1-\frac{1}{r}}+\frac{\eta K B^{1-\frac{1}{r}}}{\alpha d N}<\frac{1}{2}.$$
	On en conclut que pour tout $k\leqslant N$ et pour tout entier positif $v$ fixé, il existe au plus un entier $u$ tel que
	$(u,v)\in T_K(\varepsilon,\eta,d,B)$. Fixons dans la suite $k$ et considérons l'intervalle
	$$J_k=\left[-\frac{\varepsilon kK}{\alpha d N}B^{1-\frac{1}{r}},-\frac{\eta (k-1)K}{\alpha d N}B^{1-\frac{1}{r}}\right[.$$
	On cherche maintenant une équivalence à l'hypothèse suivante:
	\begin{center}
		\textbf{(H)} Pour $v\in\NN$ fixé, il existe un (unique) $u\in \ZZ$ tel que $u-\alpha v\in\left]\frac{\eta (k-1)K}{\alpha d N}B^{1-\frac{1}{r}},\frac{\varepsilon k K}{\alpha d N}B^{1-\frac{1}{r}}\right[ $.
	\end{center}
	La difficulté de la démonstration qui suit vient du cas $r=1$, pour lequel l'intervalle $J_k$ peut contenir un entier, y compris pour des grandes valeurs de $B$.
	On a les deux possibilités suivantes: \\
	\textbf{Cas I}: l'intervalle $J_k$ contient un entier $u_k$. Alors on a
	$$J_k-(u_k-1)=\left[-\frac{\varepsilon k K}{\alpha d N}B^{1-\frac{1}{r}}-(u_k-1),1\right[ \bigcup \left[1,-\frac{\eta (k-1)K}{\alpha d N}B^{1-\frac{1}{r}}-(u_k-1)\right[.$$
	Alors \textbf{(H)} revient à dire que 
	$$ \{\alpha v\} =\alpha v -u-(u_k-1) \in I_{k,1}=\left[-\frac{\varepsilon k K}{\alpha d N}B^{1-\frac{1}{r}}-(u_k-1),1\right[ $$
	$$\text{ ou } \{\alpha v\}=\alpha v-u-u_k\in I_{k,2}=\left[0,-\frac{\eta (k-1)K}{\alpha d N}B^{1-\frac{1}{r}}-u_k\right[.$$
	Réciproquement, si $v$ vérifie l'une des conditions ci-dessus, l'entier $u$ tel que \textbf{(H)} soit vérifiée existe et vaut $\alpha v-\{\alpha v\}-u_k+1$ ou $\alpha v-\{\alpha v\}-u_k$ selon l'appartenance de $\{\alpha v\}$. Notons que $I_{k,1}\cap  I_{k,2}=\varnothing$. En désignant $I_k=I_{k,1}\cup I_{k,2}\subset\mathopen]0,1\mathclose[$, qui est soit un intervalle (si l'un des $I_{k,i}$ est vide), soit la réunion disjointe de deux intervalles, on conclut que la condition \textbf{(H)} est équivalente à $\{\alpha v\} \in I_k$.\\
	\textbf{Cas II}: l'intervalle $J_k$ ne contient aucun entier. Suppose que $J_k\subset \mathopen ]u_{k}-1,u_k\mathclose[$ où $u_k\in\ZZ$. Alors dans ce cas \textbf{(H)} est équivalente à 
	$$\{\alpha v\}=\alpha v-u-(u_k-1) \in  I_k=\left[-\frac{\varepsilon k K}{\alpha d N}B^{1-\frac{1}{r}}-(u_k-1),-\frac{\eta(k-1)K}{\alpha d N}B^{1-\frac{1}{r}}-(u_k-1)\right[\subset \mathopen ]0,1\mathclose[.$$
	On vérifie que l'on a une équivalence analogue si l'on suppose \eqref{eq:epsilonetanegative}.
	Ayant établi cette équivalence, on peut réduire le comptage en deux variables $(u,v)$ en une seule $v$.	Par abuse de notations, pour une propriété $\mathcal{P}(x)$, on note $\mathtt{1}_{\mathcal{P}(x)}=\mathtt{1}_{\mathcal{P}(x)}(x)$ la fonction qui vaut $1$ si $\mathcal{P}(x)$ est vérifiée et vaut $0$ sinon. En notant
	$$T_{\varepsilon,\eta,d,k,B}=\sum_{\frac{(k-1)KB}{\alpha d N}<v\leqslant\frac{kKB}{\alpha d N}}\mathtt{1}_{\{\alpha v\}\in I_k},$$
	nous avons la décomposition suivante
	\begin{equation}\label{eq:decomposionofTgeneral}
	\begin{split}
	\sharp T_K(\varepsilon,\eta,d,B)&=\sum_{1\leqslant k\leqslant N}\sum_{\frac{(k-1)KB}{\alpha d N}<v\leqslant\frac{kKB}{\alpha d N}}\sum_{u\in\ZZ}\mathtt{1}_{\eta v B^{-\frac{1}{r}}<u-\alpha v\leqslant \varepsilon v B^{-\frac{1}{r}}}\\
	&=\sum_{1\leqslant k\leqslant N}\left(\sum_{\frac{(k-1)KB}{\alpha d N}<v\leqslant\frac{kKB}{\alpha d N}}\mathtt{1}_{\{\alpha v\}\in I_k}\right)+\operatorname{Er},\\
	&=\left(\sum_{1\leqslant k\leqslant N} T_{\varepsilon,\eta,d,k,B}\right)+\operatorname{Er},
	\end{split}
	\end{equation}
	où le terme d'erreur $\operatorname{Er}$ admet la majoration suivante.
	\begin{equation}\label{eq:Erterm}
	\begin{split}
	|\operatorname{Er}|&\leqslant 2\sum_{1\leqslant k\leqslant N}\sum_{\frac{(k-1)KB}{\alpha d N}<v\leqslant\frac{kKB}{\alpha d N}} \mathtt{1}_{\substack{\varepsilon v B^{-\frac{1}{r}}< u-\alpha v<\frac{\varepsilon kK}{\alpha d N}B^{1-\frac{1}{r}}\\ \text{ ou } \frac{\eta (k-1)K}{\alpha d N}B^{1-\frac{1}{r}}\leqslant u-\alpha v\leqslant \eta vB^{-\frac{1}{r}}}}\\
	&\leqslant 2\sum_{1\leqslant k\leqslant N}\sum_{\frac{(k-1)KB}{\alpha d N}<v\leqslant\frac{kKB}{\alpha d N}}\left(\mathtt{1}_{\frac{\varepsilon (k-1)K}{\alpha d N}B^{1-\frac{1}{r}}<u-\alpha v<\frac{\varepsilon kK}{\alpha dN}B^{1-\frac{1}{r}}} +\mathtt{1}_{\frac{\eta (k-1)K}{\alpha dN}B^{1-\frac{1}{r}}\leqslant u-\alpha v\leqslant\frac{\eta kK}{\alpha dN}B^{1-\frac{1}{r}}}\right)\\
	&=2 \sum_{1\leqslant k\leqslant N} \left(T_{\varepsilon,\varepsilon,d,k,B}+ T_{\eta,\eta,d,k,B}\right).
	\end{split}
	\end{equation}
	On traite d'abord le terme principal en appliquant les Théorèmes \ref{co:controlforthediscrepancy} et \ref{th:koksma-denjoy}. On définit la fonction $\phi_k:\mathopen[0,1\mathclose]\to\mathbb{R}$ comme
	\begin{equation}\label{eq:choiceofphi}
	\phi_k(u)=\mathtt{1}_{u\in I_k}(u).
	\end{equation}
	On décompose $T_{\varepsilon,\eta,d,k,B}$ comme
	$$T_{\varepsilon,\eta,d,k,B}=\sum_{v\leqslant\frac{kKB}{\alpha d N}}\mathtt{1}_{\{\alpha v\}\in I_k}-\sum_{v\leqslant\frac{(k-1)KB}{\alpha d N}}\mathtt{1}_{\{\alpha v\}\in I_k}.$$
	Puisque $m(\alpha)=2$, pour tout $\sigma^\prime>0$, $2+\sigma^\prime$ est un ordre d'irrationalité de $\alpha$.
	D'après les Théorèmes \ref{th:koksma-denjoy} et \ref{co:controlforthediscrepancy} (avec $t=2+\sigma^\prime$), pour tout $\sigma>0$, en choisissant $\sigma^\prime>0$ tel que $\frac{\sigma^\prime}{1+\sigma^\prime}+\sigma^\prime\leqslant  \sigma$, on a
	\begin{align*}
	T_{1,k}&=\sum_{v\leqslant\frac{kKB}{\alpha d N}}\mathtt{1}_{\{\alpha v\}\in I_k}\\
	&=\frac{kK}{\alpha d}\frac{B}{N}\left(\int_{0}^{1}\phi_k(t)\operatorname{d}t+O\left(D_\alpha\left(\frac{kK}{\alpha d}\frac{B}{N}\right)\right)\right)\\
	&=\frac{kK}{\alpha d}\frac{B}{N}|I_k|+O\left(\frac{kK}{\alpha d}\frac{B}{N}D_\alpha \left(\frac{kK}{\alpha d}\frac{B}{N}\right)\right)\\
	&=\frac{(\varepsilon-\eta) k^2 K^2}{\alpha^2 d^2}\frac{B^{2-\frac{1}{r}}}{N^2}+\frac{\eta k K^2}{\alpha^2 d^2}\frac{B^{2-\frac{1}{r}}}{N^2}+O\left(\frac{k^\sigma K^\sigma B^\sigma}{d^\sigma N^\sigma}\right)
	\end{align*}
	où la constante implicite dépend de $\sigma$ et $\alpha$. 
	De façon similaire on obtient
	\begin{align*}
	T_{2,k}=\sum_{v\leqslant\frac{(k-1)K B}{\alpha d N}}\mathtt{1}_{\{\alpha v\}\in I_k}
	=\frac{(\varepsilon-\eta) k(k-1)K^2}{\alpha^2 d^2}\frac{B^{2-\frac{1}{r}}}{N^2}+\frac{\eta  (k-1)K^2}{\alpha^2 d^2}\frac{B^{2-\frac{1}{r}}}{N^2}+O\left(\frac{k^\sigma K^\sigma B^\sigma}{d^\sigma N^\sigma}\right).
	\end{align*}
	Donc
	\begin{equation}\label{eq:Tepsilonkb}
	T_{\varepsilon,\eta,d,k,B}=\frac{(\varepsilon-\eta) kK^2}{\alpha^2 d^2}\frac{B^{2-\frac{1}{r}}}{N^2}+\frac{\eta K^2 }{\alpha^2 d^2}\frac{B^{2-\frac{1}{r}}}{N^2}+O\left(\frac{k^\sigma K^\sigma B^\sigma}{d^\sigma N^\sigma}\right).
	\end{equation}
	On somme sur tous les $k$. Faisons-le d'abord pour le terme principal de $T_{\varepsilon,\eta,d,k,B}$ \eqref{eq:Tepsilonkb}:
	\begin{align*}
	\sum_{k=1}^{N}\left( \frac{(\varepsilon-\eta) kK^2}{\alpha^2 d^2}\frac{B^{2-\frac{1}{r}}}{N^2}+\frac{\eta K^2 }{\alpha^2 d^2}\frac{B^{2-\frac{1}{r}}}{N^2}\right)&=\frac{(\varepsilon-\eta)K^2}{\alpha^2 d^2}\frac{N(N+1)}{2N^2}B^{2-\frac{1}{r}}+\frac{\eta K^2}{\alpha^2 d^2}\frac{B^{2-\frac{1}{r}}}{N}\\
	&=\frac{(\varepsilon-\eta)K^2}{2\alpha^2 d^2}B^{2-\frac{1}{r}}+O_{\varepsilon,\eta}\left(\frac{K^2B^{2-\frac{1}{r}}}{N d^2}\right).
	\end{align*} 
	Ensuite pour le terme d'erreur de $T_{\varepsilon,d,k,B}$ \eqref{eq:Tepsilonkb}:
	\begin{align}
	\sum_{k=1}^{N}\frac{k^\sigma K^\sigma B^\sigma}{d^\sigma N^\sigma}= O\left(\frac{K^\sigma B^\sigma N}{d^\sigma}\right).
	\end{align}
	On obtient \begin{equation}\label{eq:sumofTersilonkb}
	\sum_{k=1}^{N}T_{\varepsilon,\eta,d,k,B}=\frac{(\varepsilon-\eta) K^2}{2\alpha^2 d^2}B^{2-\frac{1}{r}}+O_{\varepsilon,\eta}\left(\frac{K^2 B^{2-\frac{1}{r}}}{N d^2}\right)+O_{\sigma}\left(\frac{K^\sigma B^\sigma N}{d^\sigma}\right)
	\end{equation}
	où la constante implicite dépend de $\sigma$ et $\alpha$. 
	Nous obtenons aussi une majoration pour le terme $\operatorname{En}$ (prenant $\varepsilon=\eta$) en vertu de \eqref{eq:Erterm}:
	\begin{equation}\label{eq:Er}
	\begin{split}
	\operatorname{Er}
	=O_{\varepsilon,\eta}\left(\frac{K^2 B^{2-\frac{1}{r}}}{N d^2}\right)+O_{\sigma}\left(\frac{K^\sigma B^\sigma N}{d^\sigma}\right).
	\end{split}
	\end{equation}
	En résumé de \eqref{eq:sumofTersilonkb} et \eqref{eq:Er}, on a démontré que pour $\alpha$ ayant la mesure d'irrationalité $2$ (pas nécessairement quadratique)
	\begin{equation}\label{eq:Tepsilondb}
	\sharp T_K(\varepsilon,\eta, d,B)=\frac{(\varepsilon-\eta)K^2}{2\alpha^2 d^2}B^{2-\frac{1}{r}}+O_{\varepsilon,\eta}\left(\frac{K^2B^{2-\frac{1}{r}}}{N d^2}\right)+O_{\sigma}\left(\frac{K^\sigma B^\sigma N}{d^\sigma}\right).
	\end{equation}
	
	Maintenant, supposons la validité de l'inégalité \eqref{eq:LiouvillealphaXi}. 
	Il suffit de modifier tous les termes d'erreur concernant la discrépance, à savoir, les termes qui contiennent $\sigma$. Plus précisément, en utilisant le Lemme \ref{le:boundforboundquotients} et en rappelant $\varDelta(\alpha)$ \eqref{eq:deltalpha}, on a
	$$\sum_{1\leqslant k\leqslant N}  \frac{kK}{\alpha d}\frac{B}{N}D_\alpha \left(\frac{kK}{\alpha d}\frac{B}{N}\right)=\sum_{1\leqslant k\leqslant N} O\left(\varDelta(\alpha)\log \left(\frac{kK}{\alpha d}\frac{B}{N}\right)\right)=O\left(\varDelta(\alpha)N\log\left(\frac{KB}{\alpha d}\right)\right)$$
	et donc, 
	\begin{align*}
	\sharp T_K(\varepsilon,\eta, d,B)&=\frac{(\varepsilon-\eta)K^2}{2\alpha^2 d^2}B^{2-\frac{1}{r}}+O_{\varepsilon,\eta}\left(\frac{K^2 B^{2-\frac{1}{r}}}{N d^2}\right)+O\left(\varDelta(\alpha)N\log\left(\frac{KB}{\alpha d}\right)\right)\\
	&=\frac{(\varepsilon-\eta)K^2}{2\alpha^2 d^2}B^{2-\frac{1}{r}}+O_{\varepsilon,\eta}\left(\frac{K^2 B^{2-\frac{1}{r}}}{N d^2}\right)+O\left(\varDelta(\alpha)N\log\left(KB\right)\right).\qedhere
	\end{align*}
\end{proof}
	\begin{proof}[Démonstration du Théorème \ref{th:localdistributionofrealnumbers}] 
			On peut supposer que $\alpha>0$. On fixe $\varepsilon>0,B\geqslant 1$. Soit $\alpha^\prime=\max(\alpha,1)$. \\
			\textbf{Cas $\frac{1}{2}<r\leqslant 1$.}	
			La première étape est de comparer le cardinal de $S(\varepsilon_1,\varepsilon_2,B)$ \eqref{eq:thesetS} avec celui de $T(\varepsilon_1,\varepsilon_2,B)$ qui consiste en les $(u,v)\in\ZZ\times\NN^*$ de sorte que 
			\begin{equation}\label{eq:thesetT}
			\pgcd(u,v)=1,\quad\varepsilon_2 B^{-\frac{1}{r}}<\left|\frac{u}{v}-\alpha\right|\leqslant \varepsilon_1 B^{-\frac{1}{r}},\quad
			v\leqslant B/\alpha^\prime
			\end{equation}
			 On prend $(u,v)\in S(\varepsilon_1,\varepsilon_2,B)$ (\eqref{eq:thesetS}). Alors \eqref{eq:thesetS} implique que
			 $$|u-\alpha v|\leqslant \varepsilon_1 B^{-\frac{1}{r}} v\leqslant \varepsilon_1.$$
			 Donc pour un tel $v$ choisi, il n'y a qu'un nombre fini de choix pour $u$.
			 De plus on a $$v\leqslant \alpha^{-1}(\varepsilon_1+|u|)\leqslant B/\alpha +\varepsilon_1 /\alpha.$$ 
			 On en conclut que $$\sharp (S(\varepsilon_1,\varepsilon_2,B)\setminus T(\varepsilon_1,\varepsilon_2,B))=O_{\varepsilon_i}(1).$$
			 Réciproquement, si l'on prend $(u,v)\in T(\varepsilon_1,\varepsilon_2,B)$, \eqref{eq:thesetT} implique que
			 $$|u|\leqslant v(\alpha+\varepsilon_1 B^{-\frac{1}{r}})\leqslant B+\varepsilon_1/\alpha.$$
			 On en conclut de façon analogue que 
			 $$\sharp (T(\varepsilon_1,\varepsilon_2,B)\setminus S(\varepsilon_1,\varepsilon_2,B))=O_{\varepsilon_i}(1),$$
			 et ainsi que
			 \begin{equation}\label{eq:thesetSd}
			 \sharp S(\varepsilon_1,\varepsilon_2,B)=\sharp T(\varepsilon_1,\varepsilon_2,B)+O_{\varepsilon_i}(1).
			 \end{equation}
		On est ramené à calculer le cardinal de l'ensemble $T(\varepsilon_1,\varepsilon_2,B)$. 
		À l'aide de l'inversion de Möbius, on définit pour $d\in\NN^*$,
\begin{equation}\label{eq:Mobius1}
		\sharp T(\varepsilon_1,\varepsilon_2,B)=\sum_{d\in\NN^*} \mu(d) \left(\sharp T_1(\varepsilon_1,\varepsilon_2,d,B)+\sharp T_2(\varepsilon_1,\varepsilon_2,d,B)\right),
\end{equation}
		où $T_1(\varepsilon_1,\varepsilon_2,d,B)$ consiste en les $(u,v)\in\ZZ\times\NN^*$ tels que 
		\begin{equation}\label{eq:thesetT1}
		\varepsilon_2 B^{-\frac{1}{r}}<\frac{u}{v}-\alpha\leqslant \varepsilon_1 B^{-\frac{1}{r}},\quad
		v\leqslant \frac{B}{\alpha^\prime d}
		\end{equation}
		et $T_2(\varepsilon_1,\varepsilon_2,d,B)$ consiste en les $(u,v)\in\ZZ\times\NN^*$ tels que 
		\begin{equation}\label{eq:thesetT2}
		-\varepsilon_1 B^{-\frac{1}{r}}\leqslant\frac{u}{v}-\alpha<-\varepsilon_2 B^{-\frac{1}{r}},\quad
		v\leqslant \frac{B}{\alpha^\prime d}
		\end{equation}
On obtient une borne sur $d$ plus fine que celle naïve $d\leqslant \frac{B}{\alpha^\prime}$ de la façon suivante.
D'une part,
\begin{equation}\label{eq:upperboundfordistz}
\dist(\alpha v,\ZZ)\leqslant |u-\alpha v|\leqslant \varepsilon_1 v B^{-\frac{1}{r}}\leqslant \varepsilon_1 \frac{B^{1-\frac{1}{r}}}{ d}.
\end{equation}
D'après l'hypothèse sur $\alpha$, pour tout $\lambda>0$, il existe une constante $C(\alpha,\lambda)>0$ telle que pour tout $n\in\ZZ^*$,
$$\dist(\alpha n,\ZZ)\geqslant \frac{C(\alpha,\lambda)}{n^{1+\lambda}}.$$
Donc
$$\dist(\alpha v,\ZZ)\geqslant \frac{C(\alpha,\lambda)}{v^{1+\lambda
	}}\geqslant\frac{C(\alpha,\lambda)d^{1+\lambda}}{B^{1+\lambda}}.$$
	En combinant les deux inégalités on obtient
	$$\varepsilon_1 \frac{B^{1-\frac{1}{r}}}{ d}\geqslant C(\alpha,\lambda)\frac{d^{1+\lambda}}{B^{1+\lambda}}.$$
	Donc on peut restreindre la somme en $d$ à ceux vérifiant
	\begin{equation}\label{eq:boundford}
	d\leqslant\left(\frac{\varepsilon_1}{C(\alpha,\lambda)}\right)^{\frac{1}{2+\lambda}}B^{1-\frac{1}{(2+\lambda)r}}.
	\end{equation}
	Si de plus \eqref{eq:LiouvillealphaXi} est valide, alors 
	$$\dist(\alpha v,\ZZ)\geqslant \frac{\Xi(\alpha)}{v}\geqslant \frac{\Xi(\alpha)d}{B}.$$
	Dans ce cas en combinant \eqref{eq:upperboundfordistz} on obtient
	\begin{equation}\label{eq:boundfordforgoodalpha}
	d\leqslant \left(\frac{\varepsilon_1}{\Xi(\alpha)}\right)^\frac{1}{2} B^{1-\frac{1}{2r}}.
	\end{equation}
	Nous allons nous concentrer sur le dénombrement de l'ensemble $T_1(\varepsilon_1,\varepsilon_2,d,B)$. Celui de $T_2(\varepsilon_1,\varepsilon_2,d,B)$ suit en imitant les lignes de la preuve (en changeant $\alpha$ par $-\alpha$, et $u$ par $-u$, notons que $\alpha$ et $-\alpha$ ont la même mesure d'irrationalité). On décompose $T_1(\varepsilon_1,\varepsilon_2,d,B)$ \eqref{eq:Mobius1} en des parties $T(\eta_{i-1},\eta_{i},d,B)$ où
\begin{equation}\label{eq:thesetS1}
T(\eta_{i-1},\eta_{i},d,B)=\left\{
\begin{aligned}
(u,v)\in\ZZ\times \NN^*
\end{aligned}
\left|
\begin{aligned}
&\eta_{i-1}B^{-\frac{1}{r}}<\left|\frac{u}{v}-\alpha\right|\leqslant \eta_i B^{-\frac{1}{r}}\\
&v\leqslant \frac{B}{\alpha^\prime d}
\end{aligned}
\right\}\right. 
\end{equation}	
avec $\varepsilon_2=\eta_0<\cdots<\eta_l=\varepsilon_1$ choisies de sorte que 
$$\eta_i-\eta_{i-1}<\alpha/4.$$
En appliquant la Proposition \ref{po:smallepsilon} à ces ensembles (avec $K=1,\varepsilon=\eta_i,\eta=\eta_{i-1}$ et $\alpha=\alpha^\prime$) on déduit que pour tout $\sigma>0$ et $N>\max(4\varepsilon_{1},1)$,
\begin{equation}\label{eq:T1var1var2}
\sharp T_1(\varepsilon_1,\varepsilon_2,d,B)=
\frac{\varepsilon_1-\varepsilon_2}{2(\alpha^\prime)^2 d^2}B^{2-\frac{1}{r}}+O_{\varepsilon_i}\left(\frac{B^{2-\frac{1}{r}}}{N d^2}\right)+O_{\sigma,\varepsilon_i}\left(\frac{B^\sigma N}{d^\sigma}\right),
\end{equation}
ce qui est valide aussi pour $T_2(\varepsilon_1,\varepsilon_2,d,B)$ (cf. la discussion au-dessus de \eqref{eq:thesetS1}). 
	Il reste à sommer sur tous les $d$. Rappelons d'après \eqref{eq:boundford} que l'on a pour tout $\lambda>0$,
	$$d= O_{\lambda,\alpha,\varepsilon_i}(B^{1-\frac{1}{(2+\lambda)r}}).$$
	On choisit $$\lambda=\frac{\sigma}{2}\left(1-\frac{1}{2r}\right)\text{ et } N=\frac{B^{(1-\frac{\sigma}{4})(1-\frac{1}{2r})}}{d^{1-\frac{\sigma}{2}}}.$$
	On vérifie que avec ce choix
	$$\left(1-\frac{1}{(2+\lambda)r}\right)\left(1-\frac{\sigma}{2}\right)<\left(1-\frac{\sigma}{4}\right)\left(1-\frac{1}{2r}\right)$$
	et donc $$N\gg B^{\left(1-\frac{\sigma}{4}\right)\left(1-\frac{1}{2r}\right)-\left(1-\frac{1}{(2+\lambda)r}\right)(1-\frac{\sigma}{2})}\gg \varepsilon_1$$ pour $B \gg_{\sigma,\varepsilon} 1$. On réécrit \eqref{eq:T1var1var2} comme
	$$\sharp T_1(\varepsilon_1,\varepsilon_2,d,B)=
	\frac{\varepsilon_1-\varepsilon_2}{2(\alpha^\prime)^2 d^2}B^{2-\frac{1}{r}}+O_{\varepsilon_i,\sigma}\left(\frac{B^{1-\frac{1}{2r}+\sigma(\frac{3}{4}+\frac{1}{8r})}}{d^{1+\frac{\sigma}{2}}}\right).$$
	En utilisant l'estimation élémentaire
	$$\sum_{n=1}^{M}\frac{\mu(n)}{n^2}=\frac{6}{\pi^2} +O\left(\frac{1}{M}\right),$$ 
	la somme du terme principal de \eqref{eq:T1var1var2} est égale à 
	\begin{align*}
		\sum_{d= O(B^{1-\frac{1}{(2+\lambda)r}})}\mu(d)\frac{\varepsilon_1-\varepsilon_2}{2(\alpha^\prime)^2 d^2}B^{2-\frac{1}{r}} &=\frac{3(\varepsilon_1-\varepsilon_2)}{\pi^2 (\alpha^\prime)^2}B^{2-\frac{1}{r}}+O_{\varepsilon_i,\lambda}(B^{1-\frac{1+\lambda}{(2+\lambda)r}})\\
		&=\frac{3(\varepsilon_1-\varepsilon_2)}{\pi^2 (\alpha^\prime)^2}B^{2-\frac{1}{r}}+O_{\varepsilon_i,\lambda}(B^{1-\frac{1}{2r}})
	\end{align*}
	Et celle du terme d'erreur est d'ordre de grandeur égal à
	\begin{align*}
&\sum_{d= O(B^{1-\frac{1}{(2+\lambda)r}})} O_{\varepsilon_i,\sigma}\left(\frac{B^{1-\frac{1}{2r}+\sigma(\frac{3}{4}+\frac{1}{8r})}}{d^{1+\frac{\sigma}{2}}}\right)=O_{\varepsilon_i,\sigma}(B^{1-\frac{1}{2r}+\sigma(\frac{3}{4}+\frac{1}{8r})}).
	\end{align*}
	Finalement en reportant dans \eqref{eq:Mobius1}, en rajoutant aussi la contribution de $T_2(\varepsilon_1,\varepsilon_2,d,B)$, on en conclut que
	\begin{equation}\label{eq:resultforsepsilonb}
	\sharp T(\varepsilon_1,\varepsilon_2,B)=\frac{6(\varepsilon_1-\varepsilon_2)}{\pi^2 (\alpha^\prime)^2}B^{2-\frac{1}{r}}+O_{\varepsilon_i,\sigma}(B^{1-\frac{1}{2r}+\sigma(\frac{3}{4}+\frac{1}{8r})}).
	\end{equation}
	Il ne reste qu'à prendre, pour $\tau>0$ donné, $\sigma>0$ tel que $\sigma(\frac{3}{4}+\frac{1}{8r})\leqslant \tau$.
	En reportant dans \eqref{eq:thesetSd}, cela démontre finalement que 
	\begin{equation}\label{eq:conclusion1}
	\begin{split}
	\delta_{\PP^1\alpha,B,r}(\chi(\varepsilon_1,\varepsilon_2))&=\sharp S(\varepsilon_1,\varepsilon_2,B)\\
	&=B^{2-\frac{1}{r}}\frac{3}{\pi^2 (\alpha^\prime)^2}\left(\int_{\varepsilon_2}^{\varepsilon_1}\operatorname{d}x+\int_{-\varepsilon_1}^{-\varepsilon_2}\operatorname{d}x\right)+O_{\varepsilon_i,\tau}(B^{1-\frac{1}{2r}+\tau})\\
	&=B^{2-\frac{1}{r}}\frac{3}{\pi^2 (\alpha^\prime)^2}\int \chi(\varepsilon_1,\varepsilon_2)\operatorname{d}x + O_{\varepsilon_i,\tau}(B^{1-\frac{1}{2r}+\tau}).
	\end{split}
	\end{equation}
	
 	Maintenant, supposons la validité de \eqref{eq:LiouvillealphaXi}. Rappelons la borne pour $d$ \eqref{eq:boundfordforgoodalpha}. Dans ce cas on prend
 	$$N=\frac{(16\varepsilon_1^\frac{3}{2}+2\varepsilon_1^\frac{1}{2})\Xi(\alpha)^{-\frac{1}{2}}B^{1-\frac{1}{2r}}}{d}>\max(4\varepsilon_1,1).$$
 	 Comme l'on suppose que $0<\Xi(\alpha)<1$, on déduit de \eqref{eq:Tstrongalpha} en utilisant \eqref{eq:T1var1var2},
 	 $$	\sharp T_1(\varepsilon_1,\varepsilon_2,d,B)=\frac{\varepsilon_1-\varepsilon_2}{ 2(\alpha^\prime)^2d^2}B^{2-\frac{1}{r}}+O_{\varepsilon_i}\left(\frac{\varDelta(\alpha)\Xi(\alpha)^{-\frac{1}{2}}}{d}B^{1-\frac{1}{2r}}\log (B)\right).$$
 	 De même pour $T_2(\varepsilon_1,\varepsilon_2,d,B)$.
 	 On somme sur tous les $d$ dans le terme principal et le terme d'erreur respectivement.
 	\begin{align*}
 	\sum_{d\leqslant \varepsilon_{1}^\frac{1}{2} \Xi(\alpha )^{-\frac{1}{2}}B^{1-\frac{1}{2r}}}\mu(d)\frac{\varepsilon_1-\varepsilon_2}{2(\alpha^\prime)^2 d^2}B^{2-\frac{1}{r}} &=\frac{6(\varepsilon_1-\varepsilon_2)}{\pi^2 (\alpha^\prime)^2}B^{2-\frac{1}{r}}+O_{\varepsilon_i}(\Xi(\alpha)^\frac{1}{2}B^{1-\frac{1}{2r}}),
 	\end{align*}
 	\begin{align*}
 		\sum_{d\leqslant \varepsilon_{1}^\frac{1}{2} \Xi(\alpha )^{-\frac{1}{2}}B^{1-\frac{1}{2r}}}O_{\varepsilon_i}\left(\frac{\varDelta(\alpha)\Xi(\alpha)^{-\frac{1}{2}}}{d}B^{1-\frac{1}{2r}}\log (B)\right)
 		=O_{\varepsilon_i}\left(\varDelta(\alpha)\Xi(\alpha)^{-\frac{1}{2}}B^{1-\frac{1}{2r}}\log (B)\log (\Xi(\alpha)^{-\frac{1}{2}}B^{1-\frac{1}{2r}})\right).
 	\end{align*}
 	On en conclut que
 	\begin{align*}
 		\sharp T(\varepsilon_1,\varepsilon_2,B)&=\frac{6(\varepsilon_1-\varepsilon_2)}{\pi^2 (\alpha^\prime)^2}B^{2-\frac{1}{r}}+O_{\varepsilon_i}(\Xi(\alpha)^\frac{1}{2}B^{1-\frac{1}{2r}})+O_{\varepsilon_i}\left(\varDelta(\alpha)\Xi(\alpha)^{-\frac{1}{2}}B^{1-\frac{1}{2r}}\log (B)\log (\Xi(\alpha)^{-\frac{1}{2}}B^{1-\frac{1}{2r}})\right)\\
 		&=\frac{3}{\pi^2 (\alpha^\prime)^2}B^{2-\frac{1}{r}}\int\chi(\varepsilon_1,\varepsilon_2)\operatorname{d}x+O_{\varepsilon_i}\left(\varDelta(\alpha)\Xi(\alpha)^{-\frac{1}{2}}B^{1-\frac{1}{2r}}\log (B)\log (\Xi(\alpha)^{-\frac{1}{2}}B^{1-\frac{1}{2r}})\right).
 	\end{align*}
 	Cela démontre la formule \eqref{eq:strongalpha}.\\
 	\textbf{Cas $r>1$.} 
 		Une inversion de Möbius nous fournit que 
 	$$\sharp S(\varepsilon_1,\varepsilon_2,B)=\sum_{d\leqslant B}\mu(d)(\sharp S(d,\varepsilon_1,B)-\sharp S(d,\varepsilon_2,B))$$
 	où 
 	\begin{equation*}
 	S(d,\varepsilon,B)=\left\{
 	(u,v)\in\ZZ\times \NN^*
 	\left|
 	\begin{aligned}
 	&\left|\frac{u}{v}-\alpha\right|\leqslant \varepsilon B^{-\frac{1}{r}}\\
 	&\sup(|u|,v)\leqslant B/\alpha^\prime d
 	\end{aligned}
 	\right\}\right. . 
 	\end{equation*}	
    On a, par la méthode de comparaison classique avec l'aire du domaine réel (du triangle dans la figure \ref{fig:area}, notons en effet qu'on a deux telles régions compte tenu du signe de $u$),
    $$\sharp S(d,\varepsilon,B)= B^{2-\frac{1}{r}}\frac{\varepsilon}{(\alpha^\prime)^2 d^2}+O_{|\alpha|}\left(\frac{B}{ d}+\frac{\varepsilon B^{1-\frac{1}{r}}}{d}\right) =B^{2-\frac{1}{r}}\frac{2\varepsilon}{(\alpha^\prime)^2 d^2}+O_{|\alpha|,\varepsilon}\left(\frac{B}{ d}\right).$$
    Puis on somme sur tous les $d$ possible \eqref{eq:boundford} et l'on obtient
    \begin{align*}
        \sharp S(\varepsilon_1,\varepsilon_2,B)&=\frac{6(\varepsilon_1-\varepsilon_2)}{\pi^2(\alpha^\prime)^2 }B^{2-\frac{1}{r}}+O_{|\alpha|,\varepsilon_i}(B^{1-\frac{1}{r}})+O_{|\alpha|,\varepsilon_i}(B\log B)\\
        &=\frac{3}{\pi^2(\alpha^\prime)^2 }B^{2-\frac{1}{r}}\int \chi(\varepsilon_1,\varepsilon_2)\operatorname{d}x+O_{|\alpha|,\varepsilon_i}(B\log B).\qedhere
    \end{align*}
 	 	\end{proof}
 	\subsection{Généralisation à un réseau}\label{se:gerelattice}
 	Les résultats sur l'approximation d'un nombre réel qu'on a démontrés précédemment peuvent être vu comme le fait d'approcher une droite par des points primitifs de pente donnée. Pour une utilisation ultérieure, nous allons présenter une version plus générale du Théorème \ref{th:localdistributionofrealnumbers}.
 	C'est-à-dire, on compte des points primitifs sur un réseau proche d'une droite dont la pente est un nombre quadratique donné. Le but est d'essayer de préciser la dépendance des constantes implicites, qui fait intervenir le déterminant du réseau. L'idée est qu'à l'aide d'une transformation linéaire liée au réseau, l'approximation sur le réseau de la droite équivaut l'approximation sur $\ZZ^2$ d'une autre droite de pente égale à un autre nombre quadratique. Comme l'on a vu, l'un des inconvénients de la démonstration ci-dessus est que la constante $C(\lambda,\alpha)$ qui apparaît dans \eqref{eq:boundford} est en général difficile à préciser, à cause de l'absence d'effectivité du théorème de Roth (ou bien de l'inégalité de Liouville). Cela mène aussi à une ineffectivité de la discrépance, car dans la démonstration du Corollaire \ref{co:controlforthediscrepancy}, on a utilisé le théorème de Roth (voir \cite{K-N}, p 123). Pour les nombres quadratiques, cette difficulté disparaît parce que l'on peut préciser facilement les constantes dans l'inégalité de Liouville \eqref{eq:liouvillelemma} et dans la majoration de la discrépance. On remarque que la technique ci-dessous fonctionne aussi pour les réels dont les quotients partiels dans l'expansion en fraction continue sont bornés (ceux qui vérifient l'inégalité de Liouville à puissance $2$). On ne rentrera pas dans les détails.
 	
 	\textbf{Notations et Conventions}: On fixe un réseau $\Lambda\subset\ZZ^2$ et $\varepsilon,K>0$. On suppose que $\Lambda$ contient un point primitif, donc $\Lambda\not\subset (d\ZZ)^2$ pour tout $d\in\NN^*$. Les nombres quadratiques auxquels on s'intéresse sont de la forme $\sqrt{\frac{b}{a}}\not\in\QQ$ avec $a<b$, que l'on notera toujours comme $\alpha$. $[m,n]$ désigne le plus petit multiple commun de deux entiers $m,n$. En pratique, les constantes $\varepsilon,\alpha$ sont bornées et l'on n'a pas besoin d'indiquer leur rôles dans les constante implicites. 
 		On note $\Lambda_d=\Lambda\cap d\ZZ^2$. et l'on définit
 		\begin{equation}\label{eq:GammaLambda}
 		\Theta(\Lambda)=\sum_{d\in\NN}\frac{\mu(d)}{\det(\Lambda_d)}.
 		\end{equation}
 		Il existe $\mathbf{f}_1,\mathbf{f}_2$ une base de $\ZZ^2$ telle que $\Lambda=\ZZ\mathbf{f}_1\oplus\ZZ\det(\Lambda)\mathbf{f}_2$ puisque $\Lambda$ contient un point primitif. On rappelle \eqref{eq:arithmeticfunction 1} et on a 
 		\begin{equation}\label{eq:Gammadet}
 		\begin{split}
 		\Theta(\Lambda)&=\sum_{d\in\NN}\frac{\mu(d)}{d[d,\det(\Lambda)]}\\
 		&=\prod_{p\nmid \det(\Lambda)}\left(1-\frac{1}{p^2}\right)\prod_{p|\det(\Lambda)}\left(1-\frac{1}{p}\right)\frac{1}{p^{v_p(\det(\Lambda))}}\\
 		&=\frac{1}{\det(\Lambda)}\frac{6}{\pi^2}\prod_{p|\det(\Lambda)}\frac{1-p^{-1}}{1-p^{-2}}=\frac{6}{\pi^2}\frac{\Psi_1(\det(\Lambda))}{\det(\Lambda)}.
 		\end{split}
 		\end{equation}
 		
 			On voudrait estimer le cardinal de l'ensemble suivant
 		\begin{equation}\label{SK}
 		S(\varepsilon,K,\Lambda,B)=\left\{
 		\begin{aligned}
 		&(u,v)\in\Lambda\cap\NN^2_{\operatorname{prem}}\\
 		\end{aligned}
 		\left|
 		\begin{aligned}
 		&0<\frac{u}{v}-\alpha\leqslant \varepsilon B^{-\frac{1}{r}}\\
 		&v\leqslant KB 
 		\end{aligned}
 		\right\}\right. .
 		\end{equation}
 		Pour réduire la difficulté technique et pour obtenir une formule asymptotique, nous allons imposer plusieurs conditions techniques. On ne prétend pas que les coefficients dans \eqref{eq:cond0} et \eqref{eq:cond4} soient optimaux mais ils suffisent pour l'utilisation ultérieure.
 		\begin{equation}\label{eq:cond0}
 		\frac{1}{2}<r<\frac{7}{10},
 		\end{equation}
 		\begin{equation}\label{eq:cond4}
 		K^{2} b\leqslant U(\alpha,\varepsilon)B^{\frac{4}{5}(\frac{1}{r}-1)-\frac{3}{5}(2-\frac{1}{r})}, \quad U(\alpha,\varepsilon)= (2^{21}\times 162 \alpha^2\varepsilon^2)^{-\frac{2}{5}}
 		\end{equation}
 		\begin{equation}\label{eq:cond5}
 		b\det(\Lambda)^2\leqslant K^2 B^{2-\frac{1}{r}}.
 		\end{equation}
 		En gros, le but d'imposer la condition \eqref{eq:cond4} est de réduire le problème de dénombrement à l'équidistribution modulo $1$, et celui de la condition \eqref{eq:cond5} est d'obtenir un meilleur terme d'erreur.
 	 \begin{theorem}\label{th:zoomingonlattices}
 		Avec les hypothèses ci-dessus, on a
  	 \begin{equation}\label{eq:estimatefors}
 \sharp S(\varepsilon,K,\Lambda,B)=\frac{\Theta(\Lambda)\varepsilon K^2}{2} B^{2-\frac{1}{r}}+O(K^\frac{3}{2}b^{\frac{1}{4}}\det(\Lambda)^\frac{1}{2} B^{\frac{3}{4}(2-\frac{1}{r})}\log B+Kb^{\frac{3}{2}}B^{1-\frac{1}{2r}}\log B).
 \end{equation}
 	\end{theorem}
 	Rappelons \eqref{eq:GammaLambda} que l'on peut aussi interpréter $\Theta(\Lambda)$, dans l'esprit de l'équidistribution, comme
 $$\Theta(\Lambda)=\frac{1}{\sharp \PP^1(\ZZ/\det(\Lambda)\ZZ)}\prod_{p}\left(1-\frac{1}{p^2}\right),$$
 qui correspond à la probabilité pour un point de $\PP^1(\QQ)$ de provenir d'un élément de $\Lambda$ primitif dans $\ZZ^2$.
 \begin{remark*}
 	Le deuxième terme d'erreur peut facilement dépasser le terme principal quand $\det(\Lambda)$ est trop petit et $b$ est trop grand, ce qui perd l'intérêt de la formule. Toutefois pour notre utilisation ultérieure il y aura des relations étroites entre $b,K,\det(\Lambda)$. Il s'avère qu'en fait ce sera le premier terme d'erreur qui contribuera plus que le deuxième.
 \end{remark*}
L'idée de la démonstration est que les points que l'on veut dénombrer sont dans un triangle. Toute transformation définie par une matrice l'envoie sur un autre triangle dont l'aire est celle du triangle initial divisé par le déterminant et la pente de l'une des bords correspond à un nouveau nombre quadratique à approcher, que l'on notera $\theta(\alpha)$, ce qui nous permet de se ramener au cas traité précédemment et d'appliquer la même technique.
 
Commençons par quelques préparations. 
 On prend un réseau $\Gamma\subset \ZZ^2$.
 	On choisit une base \begin{equation}\label{baseforgamma}
 	\be_1=(\lambda_1,\mu_1),\quad\be_2=(\lambda_2,\mu_2)
 	\end{equation} engendrant $\Gamma$ telle que (cf. \cite{Cassels} p. 135)
 	\begin{equation}\label{eq:choiceofbasis}
 	\|\be_1\|\leqslant 2 \nu_2,\quad \|\be_2\|\leqslant 2\nu_1,
 	\end{equation}
 	où $\nu_i$ désigne le $i$-ème minima successif de $\Gamma$ par rapport à la norme $\|\cdot\|$:
 	$$\|(x,y)\|=\max(|x|,|y|).$$
 	Le théorème de Minkowski (cf. \cite[Theorem V, VIII 4.3]{Cassels}) dit que
 	\begin{equation}\label{eq:Minkowski}
 	\nu_1^2\leqslant \nu_1 \nu_2\leqslant 4\det(\Gamma)=4|\lambda_2\mu_1-\lambda_1\mu_2|.
 	\end{equation}
 	On définit
 	\begin{equation}\label{eq:thetaalpha}
 	\theta=\theta(\alpha)=-\frac{\lambda_1-\alpha\mu_1}{\lambda_2-\alpha\mu_2}.
 	\end{equation}
 	Quitte à remplacer $\mathbf{e}_i$ par son opposé, on peut supposer que 
 	\begin{equation}\label{eq:supposebig0}
 	\lambda_2-\alpha\mu_2>0\quad \text{et}\quad\theta(\alpha)>0.
 	\end{equation}

 	On établit un théorème de Liouville \og effectif\fg, à savoir, avec des constantes explicites.
 	\begin{proposition}\label{po:effectiveliouville}
 		Pour tout $(u,v)\in(\NN\setminus \{0\})^2$, on a
 		\begin{equation}\label{eq:liouvilleeffective}
 		\left|\frac{u}{v}-\alpha\right|\geqslant\frac{\Xi(\alpha)}{v^2}\quad\text{et}\quad\left|\frac{u}{v}-\theta\right|\geqslant \frac{\xi(\theta)}{v^2},
 		\end{equation}
 		où on peut prendre 
 		\begin{equation}\label{eq:Xixi}
 		\Xi(\alpha)=(4\sqrt{ab})^{-1},\quad \xi(\theta)= (162b \det(\Gamma))^{-1}.
 		\end{equation}
 	\end{proposition}
 	\begin{proof}
 		On va reprendre la démonstration de l'inégalité de Liouville respectivement pour $\alpha$ et $\theta$.  
 		On note $f(x)=ax^2-b$ le polynôme minimal sur $\ZZ$ du nombre quadratique $\alpha=\sqrt{\frac{b}{a}}$.
 		Alors 
 		$$g(x)=a(\lambda_1+x\lambda_2)^2-b(\mu_1+x\mu_2)^2$$ est un polynôme entier qui annule $\theta$. D'une part comme $\alpha\not\in\QQ$, on a \begin{equation}\label{eq:ll1}
 		\left|f\left(\frac{u}{v}\right)\right|\geqslant \frac{1}{v^2},\quad\left|g\left(\frac{u}{v}\right)\right|\geqslant \frac{1}{v^2}.
 		\end{equation}
 		D'autre part, d'après le théorème de la valeur moyenne,  pour $z\in\mathopen]\alpha-1,\alpha+1\mathclose[$ et $y\in\mathopen]\theta-1,\theta+1\mathclose[$, 
 		\begin{equation}\label{eq:ll2}
 		\left|f\left(z\right)\right|\leqslant \sup_{|x-\alpha|<1}|f^\prime(x)|\left|z-\alpha\right|,\quad\left|g\left(y\right)\right|\leqslant \sup_{|x-\theta|<1}|g^\prime(x)|\left|y-\theta\right|.
 		\end{equation}
 		Or, on a les majorations
 		\begin{equation}\label{eq:ll3}
 		\sup_{|x-\alpha|<1}|f^\prime(x)|=\sup_{|x-\alpha|<1}2|ax|\leqslant 4\sqrt{ab},
 		\end{equation}
 		\begin{align*}
 \sup_{|x-\theta|<1}|g^\prime(x)|&= \sup_{|x-\theta|<1}|2|(a\lambda_2^2-b\mu_2^2)x+(a\lambda_1\lambda_2-b\mu_1\mu_2)|\\
 &\leqslant 2\sqrt{ab}|\lambda_2\mu_1-\lambda_1\mu_2|+2|a\lambda_2^2-b\mu_2^2|+4|a\lambda_1\lambda_2-b\mu_1\mu_2|.
 		\end{align*}
 	D'après les théorèmes de Minkowski \eqref{eq:Minkowski},
    en rappelant le choix de $\be_1,\be_2$ \eqref{eq:choiceofbasis} et $b>a$, il en découle que
    $$|a\lambda_2^2-b\mu_2^2|\leqslant b\|\be_2\|^2\leqslant 16b\det(\Gamma),$$
    \begin{equation}\label{eq:applicationofMinkowski1}
    |a\lambda_1\lambda_2-b\mu_1\mu_2|\leqslant 2b\|\be_1\|\|\be_2\|\leqslant 32 b\det(\Gamma),
    \end{equation}
    d'où 
    \begin{equation}\label{eq:ll4}
     \sup_{|x-\theta|<1}|g^\prime(x)|\leqslant 162 b\det(\Gamma).
    \end{equation}
    Les inégalités \eqref{eq:liouvilleeffective} sont triviales si 
    $$\left|\frac{u}{v}-\alpha\right|\geqslant 1, \quad \text{ ou }\quad \left|\frac{u}{v}-\theta\right|\geqslant 1.$$
   Dans le cas contraire les inégalités \eqref{eq:liouvilleeffective} résultent de \eqref{eq:ll1},\eqref{eq:ll2},\eqref{eq:ll3},\eqref{eq:ll4}.
 	\end{proof}
 On en déduit l'encadrement suivant qui sera utilisé fréquemment.
 \begin{corollary}\label{co:rangeoflambdamu2}
 	On a
 	$$(16b\sqrt{\det(\Gamma)})^{-1}<\lambda_2-\alpha\mu_2<8\alpha \sqrt{\det(\Gamma)}.$$
 \end{corollary}
\begin{proof}
	Premièrement,
	$$\lambda_2-\alpha\mu_2<2\alpha \|e_2\|\leqslant 8\alpha\sqrt{\det(\Gamma)}.$$
	Pour la minoration on utilise les inégalités \eqref{eq:liouvilleeffective}:
	$$\lambda_2-\alpha\mu_2>\frac{\Xi(\alpha)}{|\mu_2|}\geqslant \frac{\Xi(\alpha)}{\|\mathbf{e}_2\|}>\frac{1}{16b\sqrt{\det(\Gamma)}}.$$
\end{proof}
 	\begin{corollary}\label{co:discrepancyfortheta}
 		Pour $N> 1$, on a
 		$$N D_{\alpha}(N)=O\left(\frac{b}{\log b}\log N\right)=O(b \log N),$$
 		$$ND_\theta(N)=O\left(\frac{b\det(\Gamma)}{\log(b\det(\Gamma))}\log N\right)=O(b\det(\Gamma)\log N).$$
 	\end{corollary}
 \begin{proof}
 	Ceci résulte du Lemme \ref{le:boundforboundquotients} et de la Proposition \ref{po:effectiveliouville}.
 \end{proof}
 	\begin{proof}[Démonstration du Théorème \ref{th:zoomingonlattices}]
Avant de commencer de dénombrer l'ensemble \eqref{SK}, on présente deux conditions supplémentaires qui sont des conséquences des conditions \eqref{eq:cond0} et \eqref{eq:cond4} pour $B\gg_{\alpha,\varepsilon} 1$,
 	\begin{equation}\label{eq:cond2}
 	K^6b^3\leqslant (2^{10}\varepsilon)^{-4}B^{4(\frac{1}{r}-1)-(2-\frac{1}{r})},
 	\end{equation}
 	\begin{equation}\label{eq:cond3}
 	K^2b\leqslant(2^6 \alpha)^{-4}B^{4-3(2-\frac{1}{r})},
 	\end{equation}
 	puisque
 	$$0<\frac{4}{5}\left(\frac{1}{r}-1\right)-\frac{3}{5}\left(2-\frac{1}{r}\right)<\min\left(\frac{4}{3}\left(\frac{1}{r}-1\right)-\frac{1}{3}\left(2-\frac{1}{r}\right),4-3\left(2-\frac{1}{r}\right)\right).$$
 	Premièrement, comme toujours, on utilise l'inversion de Möbius pour éliminer la coprimalité:
 	$$\sharp S(\varepsilon,K,\Lambda,B)=\sum_{d\in\NN^*}\mu(d) \sharp S(\varepsilon,K,\Lambda_d,B),$$
 	où 
 	 \begin{equation}\label{eq:SKprime}
 	S(\varepsilon,K,\Lambda_d,B)=\left\{
 	\begin{aligned}
 	(u,v)\in\Lambda_d\cap\NN^2
 	\end{aligned}
 	\left|
 	\begin{aligned}
 	&0<u-v\alpha\leqslant \varepsilon v B^{-\frac{1}{r}}\\
 	&v\leqslant KB
 	\end{aligned}
 	\right\}\right. .
 	\end{equation} 
   	Pour déduire une borne précise pour $d$ nous suivons \eqref{eq:boundford}. Pour tout $(u,v)\in S(\varepsilon,K,\Lambda_d,B)$, on note $u^\prime=u/d,v^\prime=v/d$.
   	En utilisant la Proposition \ref{po:effectiveliouville}, on a
   	$$\frac{d\Xi(\alpha)}{KB}\leqslant\frac{\Xi(\alpha)}{v^\prime}\leqslant u^\prime-v^\prime\alpha\leqslant\varepsilon v^\prime B^{-\frac{1}{r}}\leqslant \frac{\varepsilon K }{d}B^{1-\frac{1}{r}},$$
   d'où
   	\begin{equation}\label{eq:boundford22}
   	d\leqslant \left(\frac{\varepsilon K^2}{\Xi(\alpha)}\right)^\frac{1}{2} B^{1-\frac{1}{2r}}.
   	\end{equation}
   	Notons que 
   \begin{equation}\label{eq:Dalpha}
   	D(\alpha,K)=\left(\frac{\varepsilon K^2}{\Xi(\alpha)}\right)^\frac{1}{2}\leqslant 2\sqrt{\varepsilon}K b^\frac{1}{2} \ll Kb^{\frac{1}{2}}.
   \end{equation}

   	\textbf{Cas I. $d$ est petit.} 
 	L'étape suivante consiste à éliminer la condition de réseau sur $u,v$ pour les $d$ petits. Pour contrôler la contribution des $d$ grands on peut utiliser la majoration donnée par \eqref{eq:Tstrongalpha}. On utilisera les notations \eqref{baseforgamma},\eqref{eq:choiceofbasis},\eqref{eq:thetaalpha} pour le réseau $\Gamma=\Lambda_d$, c'est-à-dire
 	$\Lambda_d=\ZZ \be_1\oplus\ZZ\be_2$. On écrit
 	$$u=n_1\lambda_1+n_2\lambda_2,\quad v=n_1\mu_1+n_2\mu_2,$$ avec $n_1,n_2\in\ZZ$.
 	On voudrait se ramener au dénombrement sur $(n_1,n_2)$ puisqu'il n'y a plus de contraintes sur la divisibilité de $n_1,n_2$. Avec les notations ci-dessus l'ensemble \eqref{eq:SKprime} s'écrit
 	 	\begin{equation*}
 	 	S(\varepsilon,K,\Lambda_d,B)=\left\{
 	 	\begin{aligned}
 	 	(n_1,n_2)\in\ZZ^2
 	 	\end{aligned}
 	 	\left|
 	 	\begin{aligned}
 	 	&0<n_2-n_1\theta\leqslant\frac{n_1\mu_1+n_2\mu_2}{\lambda_2-\alpha\mu_2}\varepsilon  B^{-\frac{1}{r}}\\
 	 	&0<n_1\mu_1+n_2\mu_2\leqslant KB
 	 	\end{aligned}
 	 	\right\}\right. .
 	 	\end{equation*}
 	On va estimer le cardinal de $S(\varepsilon,K,\Lambda_d,B)$ quand $d$ vérifie la condition
 	 \begin{equation}\label{eq:condd}
 	 d\leqslant \frac{K^\frac{1}{2}}{b^\frac{1}{4}\det(\Lambda)^\frac{1}{2}}B^{\frac{1}{4}(2-\frac{1}{r})}.
 	 \end{equation}
 	 La condition \eqref{eq:cond5} assure que tels $d$ existent. 
 	 D'abord on détermine les signes de $n_1,n_2$. En utilisant le Corollaire \ref{co:rangeoflambdamu2} et en rappelant la condition \eqref{eq:cond2}, on a
 	 \begin{equation}\label{eq:estimaten2n11}
 	 \begin{split}
 	 0<n_2-n_1\theta&\leqslant\frac{n_1\mu_1+n_2\mu_2}{\lambda_2-\alpha\mu_2}\varepsilon  B^{-\frac{1}{r}}\leqslant \frac{K\varepsilon}{\lambda_2-\alpha\mu_2}B^{1-\frac{1}{r}}\leqslant 16K\varepsilon b \sqrt{\det(\Lambda_d)}B^{1-\frac{1}{r}}\\
 	 &\leqslant 16K\varepsilon b d\sqrt{\det(\Lambda)}B^{1-\frac{1}{r}} 
 	 \leqslant 2^4 \varepsilon K^{\frac{3}{2}} b^{\frac{3}{4}}  B^{1-\frac{1}{r}+\frac{1}{4}(2-\frac{1}{r})}\leqslant \frac{1}{64}.
 	  	 \end{split}
 	 \end{equation}
 	 Cela implique non seulement que pour $n_1$ fixé, il y a au plus un $n_2$ tel que $(n_1,n_2)\in S(\varepsilon,K,\Lambda_d,B)$, mais aussi que, pour un réseau $\Lambda_d$ fixé, on a $n_1,n_2\geqslant 0 \text{ ou } n_1,n_2\leqslant 0$.
 	 On ne peut avoir qu'une seule possibilité parce que le signe de $n_1,n_2$ détermine aussi le signe de $v$, que l'on a supposé être positif. Donc dans la suite, sans perte de généralité on suppose que $n_1,n_2$ sont positifs. En fait on a dans ce cas $n_1,n_2>0$ (sinon on aurait $n_1=n_2=0$ et avec $v=0$). 
 	  On définit les ensembles $R_1(\varepsilon,K,\Lambda_d,B),R_2(\varepsilon,K,\Lambda_d,B)$ des couples $(n_1,n_2)\in\NN^*\times\NN^*$ vérifiant respectivement
 	 \begin{equation}\label{eq:RepsilonLambdad}
 	 \begin{split}
 	 \frac{\det(\Lambda_d)n_1}{(\lambda_2-\alpha\mu_2)^2}\varepsilon  B^{-\frac{1}{r}}-2^{10} K\varepsilon^2 b^2\det(\Lambda_d)^{\frac{3}{2}}B^{1-\frac{2}{r}}&<n_2-n_1\theta\leqslant\frac{\det(\Lambda_d)n_1}{(\lambda_2-\alpha\mu_2)^2}\varepsilon  B^{-\frac{1}{r}}+2^{10} K\varepsilon^2 b^2\det(\Lambda_d)^{\frac{3}{2}}B^{1-\frac{2}{r}}\\
 	 n_1&\leqslant 2\frac{\lambda_2-\alpha\mu_2}{\det(\Lambda_d)} KB
 	 \end{split}
 	 \end{equation}
 	 et
 	 \begin{equation}\label{eq:RepsilonLambdad2}
 	 \begin{split}
 	 &0<n_2-n_1\theta\leqslant\frac{3}{2}\frac{\det(\Lambda_d)n_1}{(\lambda_2-\alpha\mu_2)^2}\varepsilon  B^{-\frac{1}{r}}\\
 	 &\frac{\lambda_2-\alpha\mu_2}{\det(\Lambda_d)} KB-\frac{\alpha}{2}\leqslant n_1\leqslant \frac{\lambda_2-\alpha\mu_2}{\det(\Lambda_d)} KB+\frac{\alpha}{2}
 	 \end{split}
 	 \end{equation}
 	 et l'ensemble	\begin{equation}\label{eq:TepsilonLambdad}
 	 T(\varepsilon,K,\Lambda_d,B)=\left\{
 	 \begin{aligned}
 	 (n_1,n_2)\in\NN^*\times\NN^*
 	 \end{aligned}
 	 \left|
 	 \begin{aligned}
 	 &0<n_2-n_1\theta\leqslant\frac{\det(\Lambda_d)n_1}{(\lambda_2-\alpha\mu_2)^2}\varepsilon  B^{-\frac{1}{r}}\\
 	 &n_1\leqslant\frac{\lambda_2-\alpha\mu_2}{\det(\Lambda_d)} KB
 	 \end{aligned}
 	 \right\}\right. .
 	 \end{equation}
 	 Comme dans \eqref{eq:estimaten2n11}, on a que pour $(n_1,n_2)\in T(\varepsilon,K,\Lambda_d,B)$, 
 	 \begin{equation}\label{eq:estimaten2n12}
 	 0<n_2-n_1\theta\leqslant \frac{\det(\Lambda_d)\varepsilon B^{-\frac{1}{r}}}{(\lambda_2-\alpha\mu_2)^2}\times \frac{(\lambda_2-\alpha\mu_2)KB}{\det(\Lambda_d)}=\frac{K\varepsilon}{\lambda_2-\alpha\mu_2}B^{1-\frac{1}{r}}\leqslant \frac{1}{64}.
 	 \end{equation}
 	 
 	 L'étape suivante est de comparer les ensembles $S(\varepsilon,K,\Lambda_d,B)$ et $T(\varepsilon,K,\Lambda_d,B)$, ainsi qu'à établir (rappelons la notation $\Delta$ \eqref{eq:delta})
\begin{equation}\label{eq:STdiff}
 	 	S(\varepsilon,K,\Lambda_d,B)\Delta  T(\varepsilon,K,\Lambda_d,B)\subseteq R_1(\varepsilon,K,\Lambda_d,B)\cup R_2(\varepsilon,K,\Lambda_d,B).
\end{equation}

 	  		Tout d'abord, rappelons l'égalité
 	  		$$|\mu_1+\theta \mu_2|=\frac{\det(\Lambda_d)}{\lambda_2-\alpha\mu_2}.$$
 Soit $(n_1,n_2)\in S(\varepsilon,K,\Lambda_d,B)\cup T(\varepsilon,K,\Lambda_d,B)$. On a
\begin{equation}\label{eq:S1}
 	 n_1|\mu_1+\theta \mu_2|=|n_1\mu_1+n_2\mu_2-(n_2-n_1\theta)\mu_2|\leqslant KB+|\mu_2|(n_2-n_1\theta),
\end{equation}
D'après \eqref{eq:estimaten2n11} et le Corollaire \ref{co:rangeoflambdamu2}, on a, puisque 
$$|\mu_2|\leqslant\|e_2\|\leqslant 4\sqrt{\det(\Lambda_d)},$$
\begin{equation}\label{eq:S2}
	 n_2-n_1\theta\leqslant \frac{K \varepsilon }{\lambda_2-\alpha \mu_2} B^{1-\frac{1}{r}}\leqslant 16K\varepsilon b\sqrt{\det(\Lambda_d)}B^{1-\frac{1}{r}}.
\end{equation}
	 Donc d'après \eqref{eq:estimaten2n11} et \eqref{eq:estimaten2n12},
 	 $$\frac{|\mu_2|(n_2-n_1\theta)}{|\mu_1+\theta\mu_2|}=\frac{|\mu_2|(n_2-n_1\theta)(\lambda_2-\alpha\mu_2)}{\det(\Lambda_d)}\leqslant  32 \alpha(n_2-n_1\theta)\leqslant \frac{\alpha}{2}.$$
 	Notons aussi que, d'après le Corollaire \ref{co:rangeoflambdamu2}, \eqref{eq:cond3}, et \eqref{eq:condd}, 
 	$$\frac{\lambda_2-\alpha\mu_2}{\det(\Lambda_d)} KB > \frac{KB}{16b\det(\Lambda_d)^\frac{3}{2}} \geqslant\frac{KB}{16b\det(\Lambda)^\frac{3}{2}d^3}\geqslant \frac{B^{1-\frac{3}{4}(2-\frac{1}{r})}}{2^7 K^\frac{1}{2} b^\frac{1}{4}}\geqslant \frac{\alpha}{2}.$$
Donc 
\begin{equation}\label{eq:n1upper}
n_1\leqslant \frac{KB+|\mu_2|(n_2-n_1\theta)}{|\mu_1+\theta \mu_2|}\leqslant \frac{\lambda_2-\alpha\mu_2}{\det(\Lambda_d)} KB+\frac{\alpha}{2}\leqslant 2\frac{\lambda_2-\alpha\mu_2}{\det(\Lambda_d)} KB.
\end{equation}
Si $(n_1,n_2)\in T(\varepsilon,K,\Lambda_d,B)$ est tel que $n_1\mu_1+n_2\mu_2\geqslant KB$, on a comme dans \eqref{eq:S1},
$$|n_1\mu_1+n_2\mu_2|\leqslant n_1|\mu_1+\theta\mu_2|+|\mu_2|(n_2-n_1\theta),$$
et donc
\begin{equation}\label{eq:n1lower}
n_1\geqslant \frac{|n_1\mu_1+n_2\mu_2|}{|\mu_1+\theta\mu_2|}-\frac{|\mu_2|(n_2-n_1\theta)}{|\mu_1+\theta\mu_2|}\geqslant  \frac{\lambda_2-\alpha\mu_2}{\det(\Lambda_d)} KB-\frac{\alpha}{2}.
\end{equation}
Deuxièmement, comme d'après la définition de $\theta$ \eqref{eq:thetaalpha},
 	 $$\frac{n_1\mu_1+n_2\mu_2}{\lambda_2-\alpha\mu_2}=\frac{n_1(\lambda_2\mu_1-\lambda_1\mu_2)}{(\lambda_2-\alpha\mu_2)^2}+\frac{\mu_2(n_2-\theta n_1)}{\lambda_2-\alpha\mu_2},$$
 	on va montrer que pour tout $(n_1,n_2)\in S(\varepsilon,K,\Lambda_d,B)\cup  T(\varepsilon,K,\Lambda_d,B)$,
 	 \begin{equation}\label{eq:halfseparate}
 	\left|\frac{(\lambda_2\mu_1-\lambda_1\mu_2)n_1}{(\lambda_2-\alpha\mu_2)^2}\right|\varepsilon B^{-\frac{1}{r}}\geqslant 2^{11} K\varepsilon^2 b^2\det(\Lambda_d)^{\frac{3}{2}}B^{1-\frac{2}{r}}\geqslant 2\left|\frac{\mu_2(n_2-n_1 \theta)}{\lambda_2-\alpha\mu_2}\right|\varepsilon B^{-\frac{1}{r}}. 
 	\end{equation}
 	 	En particulier, comme on a $\det(\Lambda_d)=\lambda_2\mu_1-\lambda_1\mu_2$, cela implique que
 	\begin{align}\label{eq:n1n2diff}
 		\left|\frac{\det(\Lambda_d)n_1}{(\lambda_2-\alpha\mu_2)^2}\varepsilon  B^{-\frac{1}{r}}- \frac{n_1\mu_1+n_2\mu_2}{\lambda_2-\alpha\mu_2}\varepsilon B^{-\frac{1}{r}} \right|=\left|\frac{\mu_2(n_2-\theta n_1)}{\lambda_2-\alpha\mu_2} \varepsilon B^{-\frac{1}{r}}\right|
 		\leqslant 2^{10} K\varepsilon^2 b^2\det(\Lambda_d)^{\frac{3}{2}}B^{1-\frac{2}{r}},
 	\end{align}
 	et donc, en rappelant la définition des ensembles $R_1(\varepsilon,K,\Lambda_d,B)$ et $R_2(\varepsilon,K,\Lambda_d,B)$ \eqref{eq:RepsilonLambdad}, \eqref{eq:RepsilonLambdad2}, l'inclusion \eqref{eq:STdiff} découle des estimations \eqref{eq:n1lower}, \eqref{eq:n1upper} et \eqref{eq:n1n2diff}.
 	
 	Pour tout $(n_1,n_2)\in S(\varepsilon,K,\Lambda_d,B)\cup T(\varepsilon,K,\Lambda_d,B)$, d'après la Proposition \ref{po:effectiveliouville},
 	 $$\lambda_2-\alpha \mu_2\geqslant \frac{\Xi(\alpha)}{|\mu_2|}, \quad n_2-n_1\theta \geqslant \frac{\xi(\theta)}{n_1},$$
 	 on déduit de \eqref{eq:S2} que
 	 \begin{equation}\label{eq:minorationsurn1}
 	 n_1\geqslant\frac{B^{\frac{1}{r}-1}\xi(\theta)}{16K\varepsilon b\sqrt{\det(\Lambda_d)}}=\frac{B^{\frac{1}{r}-1}}{16\times 162 K\varepsilon b^2 \det(\Lambda_d)^{\frac{3}{2}}},
 	 \end{equation}
 	 et
 	 \begin{align*}
 	 	\left|\frac{\mu_2(n_2-n_1 \theta)}{\lambda_2-\alpha\mu_2}\right|\varepsilon B^{-\frac{1}{r}}&\leqslant (\Xi(\alpha))^{-1}\mu_2^2(n_2-n_1\theta)\varepsilon B^{-\frac{1}{r}}\\
 	 	&\leqslant 2^{10} \varepsilon^2 K b^2 \det(\Lambda_d)^\frac{3}{2}B^{1-\frac{2}{r}}.
 	 \end{align*}
 	 	 	 Pour que \eqref{eq:halfseparate} soit vraie, il reste à démontrer
 	 	 	 $$\frac{\det(\Lambda_d)n_1}{(\lambda_2-\alpha \mu_2)^2}\varepsilon B^{-\frac{1}{r}}\geqslant 2^{11} K\varepsilon^2b^2\det(\Lambda_d)^{\frac{3}{2}}B^{1-\frac{2}{r}}.$$ 
 	 	 	 Pour cela il suffit de démontrer 
 	 	 	 \begin{equation}\label{eq:minorationsurn12}
 	 	 	 n_1\geqslant 2^{17} \alpha^2 \varepsilon Kb^2\det(\Lambda_d)^\frac{3}{2}B^{1-\frac{1}{r}},
 	 	 	 \end{equation}
 	 	 	 car d'après la deuxième inégalité du Corollaire \ref{co:rangeoflambdamu2},
 	 	 	 $$2^{17} \alpha^2 \varepsilon Kb^2\det(\Lambda_d)^\frac{3}{2}B^{1-\frac{1}{r}}\geqslant 2^{11} \varepsilon(\lambda_2-\mu_2\alpha)^2 K b^2\det(\Lambda_d)^\frac{1}{2}B^{1-\frac{1}{r}}.$$
 	 	 	 D'après la condition \eqref{eq:cond4} et la condition sur $d$ \eqref{eq:condd}, on a
 	 	 	 \begin{align*}
 	 	 	 	2^{21}\times 162 \alpha^2 \varepsilon^2 K^2 b^4 \det(\Lambda_d)^3 &\leqslant 	
 	 	 	 	2^{21}\times 162 \alpha^2 \varepsilon^2 K^2 b^4 \det(\Lambda)^3 d^6\\
 	 	 	 	&\leqslant 2^{21}\times 162 \alpha^2\varepsilon^2 K^5 b^\frac{5}{2} B^{\frac{3}{2}(2-\frac{1}{r})}\\
 	 	 	 		&\leqslant B^{2(\frac{1}{r}-1)},
 	 	 	 \end{align*}
 	 	 	 ce qui implique que
 	 	 	 $$\frac{B^{\frac{1}{r}-1}}{16\times 162 K\varepsilon b^2 \det(\Lambda_d)^{\frac{3}{2}}}\geqslant 2^{17} \alpha^2 \varepsilon K b^2\det(\Lambda_d)^{\frac{3}{2}}B^{1-\frac{1}{r}}.$$
 	         Donc l'inégalité dans \eqref{eq:minorationsurn12} qui fallait démontrer découle de \eqref{eq:minorationsurn1}. La formule \eqref{eq:STdiff} est finalement achevée.

 	 Maintenant on découpe l'intervalle $$\left[0,\frac{(\lambda_2-\alpha\mu_2)K}{\det(\Lambda_d)} B \right]$$ en $N$ pièces (avec $N$ à déterminer) 
 	 \begin{equation}\label{eq:intervalsN}
 	 \left[\frac{(\lambda_2-\alpha\mu_2)(k-1)K}{N\det(\Lambda_d)} B,\frac{(\lambda_2-\alpha\mu_2)kK}{N\det(\Lambda_d)} B\right],\quad 1\leqslant k\leqslant N.
 	 \end{equation}
 	 Avant de dénombrer l'ensemble $T(\varepsilon,K,\Lambda_d,B)$, observons que pour tout $(n_1,n_2)\in T(\varepsilon,K,\Lambda_d,B)$, nous avons $n_2-n_1\theta=1-\{n_1\theta\}$ grâce à \eqref{eq:estimaten2n12}.
 	 Cela nous permet d'écrire
 	 \begin{align*}
 	 	\sharp T(\varepsilon,K,\Lambda_d,B)
 	 	&=\sum_{1\leqslant k\leqslant N}\sum_{\frac{(\lambda_2-\alpha\mu_2)(k-1)K}{N\det(\Lambda_d)}B<n\leqslant\frac{(\lambda_2-\alpha\mu_2)kK}{N\det(\Lambda_d)}B} \mathtt{1}_{1-\{n\theta\}\leqslant\frac{\det(\Lambda_d)n}{(\lambda_2-\alpha\mu_2)^2}\varepsilon  B^{-\frac{1}{r}}}\\
 	 	&=\sum_{1\leqslant k\leqslant N}\sum_{\frac{(\lambda_2-\alpha\mu_2)(k-1)K}{N\det(\Lambda_d)}B<n\leqslant\frac{(\lambda_2-\alpha\mu_2)kK}{N\det(\Lambda_d)}B}\left(\mathtt{1}_{1-\{n\theta\}\leqslant\frac{k\varepsilon K}{N(\lambda_2-\mu_2\alpha)}B^{1-\frac{1}{r}}}-\mathtt{1}_{\frac{\det(\Lambda_d)n}{(\lambda_2-\alpha\mu_2)^2}\varepsilon  B^{-\frac{1}{r}}<1-\{n\theta\}\leqslant \frac{k\varepsilon K}{N(\lambda_2-\mu_2\alpha)}B^{1-\frac{1}{r}}}\right)\\
 	 	&=\sum_{1\leqslant k\leqslant N}\left(\sum_{n\leqslant\frac{(\lambda_2-\alpha\mu_2)kK}{N\det(\Lambda_d)}B}-\sum_{n\leqslant\frac{(\lambda_2-\alpha\mu_2)(k-1)K}{N\det(\Lambda_d)}B}\right)\mathtt{1}_{1-\{n\theta\}\leqslant\frac{k\varepsilon K}{N(\lambda_2-\mu_2\alpha)}B^{1-\frac{1}{r}}}-\sum_{1\leqslant k\leqslant N}\operatorname{Er}_k,
 	 \end{align*}
 	 où pour tout $1\leqslant k\leqslant N$,
 	 $$\operatorname{Er}_k=\sum_{\frac{(\lambda_2-\alpha\mu_2)(k-1)K}{N\det(\Lambda_d)}B<n\leqslant\frac{(\lambda_2-\alpha\mu_2)kK}{N\det(\Lambda_d)}B}\mathtt{1}_{\frac{\det(\Lambda_d)n}{(\lambda_2-\alpha\mu_2)^2}\varepsilon  B^{-\frac{1}{r}}<1-\{n\theta\}\leqslant \frac{k\varepsilon K}{N(\lambda_2-\mu_2\alpha)}B^{1-\frac{1}{r}}}.$$
 	 Pour $1\leqslant k\leqslant N$, on définit $$\phi_k(u)=\mathtt{1}_{1-\frac{k\varepsilon K}{N(\lambda_2-\mu_2\alpha)}B^{1-\frac{1}{r}}\leqslant u< 1}(u).$$ 
 	 On rappelle que d'après \eqref{eq:estimaten2n11}, $$\frac{K\varepsilon}{\lambda_2-\alpha\mu_2}B^{1-\frac{1}{r}}\leqslant \frac{1}{64},$$
 	 et donc la fonction $\phi_k$ est à support dans $[0,1[$ et on a
 	 	$$\int_{0}^{1}\phi_k(u)\operatorname{d}u=\frac{k\varepsilon K}{N(\lambda_2-\mu_2\alpha)}B^{1-\frac{1}{r}}.$$
 	 Donc d'après le Théorème \ref{th:koksma-denjoy} avec le terme d'erreur précis (Corollaire \ref{co:discrepancyfortheta}),
 	 \begin{equation}
 	 \begin{split}
 	 &\sum_{n\leqslant\frac{(\lambda_2-\alpha\mu_2)kK}{N\det(\Lambda_d)}B}\mathtt{1}_{1-\{n\theta\}\leqslant\frac{k\varepsilon K}{N(\lambda_2-\mu_2\alpha)}B^{1-\frac{1}{r}}}\\
 	 &=\frac{k^2\varepsilon K^2}{N^2\det(\Lambda_d)}B^{2-\frac{1}{r}}+O\left(V(\phi)b\det(\Lambda_d)\log\left(\frac{kK(\lambda_2-\alpha\mu_2)}{N\det(\Lambda_d)}B\right)\right)\\
 	 &=\frac{k^2\varepsilon K^2}{N^2\det(\Lambda_d)}B^{2-\frac{1}{r}}+O\left(b\det(\Lambda_d) \log\left(\frac{kK(\lambda_2-\alpha\mu_2)}{N\det(\Lambda_d)}B\right)\right),
 	 \end{split}
 	 \end{equation}
 	 \begin{align*}
 	  	 &\sum_{n\leqslant\frac{(\lambda_2-\alpha\mu_2)(k-1)K}{N\det(\Lambda_d)}B}\mathtt{1}_{1-\{n\theta\}\leqslant\frac{k\varepsilon K}{N(\lambda_2-\mu_2\alpha)}B^{1-\frac{1}{r}}}\\
 	  	 &=\frac{k(k-1)\varepsilon K^2}{N^2\det(\Lambda_d)}B^{2-\frac{1}{r}}+O\left(b\det(\Lambda_d) \log\left(\frac{(k-1)K(\lambda_2-\alpha\mu_2)}{N\det(\Lambda_d)}B\right)\right).
 	 \end{align*}
Donc 
\begin{align*}
	&\sum_{\frac{(\lambda_2-\alpha\mu_2)(k-1)K}{N\det(\Lambda_d)}B<n\leqslant\frac{(\lambda_2-\alpha\mu_2)kK}{N\det(\Lambda_d)}B}\mathtt{1}_{1-\{n\theta\}\leqslant\frac{k\varepsilon K}{N(\lambda_2-\mu_2\alpha)}B^{1-\frac{1}{r}}}\\
	&=\frac{k\varepsilon K^2}{N^2\det(\Lambda_d)}B^{2-\frac{1}{r}}+O\left(b\det(\Lambda_d) \log\left(\frac{kK(\lambda_2-\alpha\mu_2)}{N\det(\Lambda_d)}B\right)\right).
\end{align*}
On calcule maintenant le terme d'erreur $\operatorname{Er}_k$ provenant du changement de $n_1$ en les valeurs du bord des intervalles \eqref{eq:intervalsN}.  En appliquant le Théorème \ref{th:koksma-denjoy} à la fonction auxiliaire
$$\mathtt{1}_{1-\frac{k\varepsilon K}{N(\lambda_2-\mu_2\alpha)}B^{1-\frac{1}{r}}\leqslant u\leqslant1-\frac{(k-1)\varepsilon K}{N(\lambda_2-\mu_2\alpha)}B^{1-\frac{1}{r}}},$$
on obtient
\begin{align*}
\operatorname{Er}_k
&\leqslant\sum_{\frac{(\lambda_2-\alpha\mu_2)(k-1)K}{N\det(\Lambda_d)}B<n\leqslant\frac{(\lambda_2-\alpha\mu_2)kK}{N\det(\Lambda_d)}B} \mathtt{1}_{\frac{(k-1)\varepsilon K}{N(\lambda_2-\mu_2\alpha)}B^{1-\frac{1}{r}}<1-\{n\theta\}\leqslant \frac{k\varepsilon K}{N(\lambda_2-\mu_2\alpha)}B^{1-\frac{1}{r}}}\\
&=\left(\sum_{n\leqslant\frac{(\lambda_2-\alpha\mu_2)kK}{N\det(\Lambda_d)}B}-\sum_{n\leqslant\frac{(\lambda_2-\alpha\mu_2)(k-1)K}{N\det(\Lambda_d)}B}  \right)  \mathtt{1}_{\frac{(k-1)\varepsilon K}{N(\lambda_2-\mu_2\alpha)}B^{1-\frac{1}{r}}<1-\{n\theta\}\leqslant \frac{k\varepsilon K}{N(\lambda_2-\mu_2\alpha)}B^{1-\frac{1}{r}}}\\
&=\frac{\varepsilon K^2}{N^2 \det(\Lambda_d)}B^{2-\frac{1}{r}}+O\left(b\det(\Lambda_d) \log\left(\frac{kK(\lambda_2-\alpha\mu_2)}{N\det(\Lambda_d)}B\right)\right).
\end{align*}
On en conclut que 
\begin{align*}
&\sharp T(\varepsilon,K,\Lambda_d,b)\\
&=\sum_{1\leqslant k\leqslant N}\left(\frac{k\varepsilon K^2}{N^2\det(\Lambda_d)}B^{2-\frac{1}{r}}+O\left(\frac{ K^2}{N^2 \det(\Lambda_d)}B^{2-\frac{1}{r}}\right)+O\left(b\det(\Lambda_d) \log\left(\frac{kK(\lambda_2-\alpha\mu_2)}{N\det(\Lambda_d)}B\right)\right)\right).
\end{align*}
Grâce à la condition \eqref{eq:cond0} et \eqref{eq:cond4},
$$\log (KB)\leqslant \log (Kb^\frac{1}{2}B)\ll \log B.$$
 	 En sommant sur tous les $k$ pour le premier terme (le terme principal) 
 	 $$\sum_{k=1}^{N}\frac{k\varepsilon K^2}{N^2\det(\Lambda_d)}B^{2-\frac{1}{r}}=\frac{\varepsilon K^2}{2\det(\Lambda_d)}B^{2-\frac{1}{r}}+O\left(\frac{K^2}{N\det(\Lambda_d)}B^{2-\frac{1}{r}}\right),$$
 	 puis sur les termes d'erreur
 	 \begin{align*}
 	 	&\sum_{k=1}^N\left(\frac{ K^2}{N^2 \det(\Lambda_d)}B^{2-\frac{1}{r}}+ b\det(\Lambda_d)\log\left(\frac{kK(\lambda_2-\alpha\mu_2)}{N\det(\Lambda_d)}B\right)\right)\\
 	 	&\ll \frac{ K^2}{N \det(\Lambda_d)}B^{2-\frac{1}{r}}+b \det(\Lambda_d)\sum_{1\leqslant k\leqslant N}\log\left(\frac{kKB}{N}\right)\\
 	 	&\ll \frac{ K^2}{N \det(\Lambda_d)}B^{2-\frac{1}{r}}+b\det(\Lambda_d)N \log(KB)\\
 	 	&\ll \frac{ K^2}{N \det(\Lambda_d)}B^{2-\frac{1}{r}}+b\det(\Lambda_d)N \log(B),
 	 \end{align*}
 	 on obtient que
 	 $$\sharp T(\varepsilon,K,\Lambda_d,b)=\frac{\varepsilon K^2}{2\det(\Lambda_d)}B^{2-\frac{1}{r}}+O\left(\frac{ K^2}{N \det(\Lambda_d)}B^{2-\frac{1}{r}}+b\det(\Lambda_d)N \log(B)\right).$$
 	 Il nous reste à majorer le terme d'erreur venant du cardinal de $R_1(\varepsilon,K,\Lambda_d,B)$ \eqref{eq:RepsilonLambdad} et $R_2(\varepsilon,K,\Lambda_d,B)$ \eqref{eq:RepsilonLambdad2}. Pour cela on utilise encore une fois le Théorème \ref{th:koksma-denjoy}.
 	 \begin{align*}
 	 	&\sharp R_1(\varepsilon,K,\Lambda_d,B)\\
 	 	&=\sum_{n\leqslant 2\frac{\lambda_2-\mu_2\alpha}{\det(\Lambda_d)}KB} \mathtt{1}_{\frac{\det(\Lambda_d)n}{(\lambda_2-\alpha\mu_2)^2}\varepsilon  B^{-\frac{1}{r}}-2^{10}K\varepsilon^2 b^2\det(\Lambda_d)^{\frac{3}{2}}B^{1-\frac{2}{r}}<1-\{n\theta\}\leqslant\frac{\det(\Lambda_d)n}{(\lambda_2-\alpha\mu_2)^2}\varepsilon  B^{-\frac{1}{r}}+2^{10}K\varepsilon^2 b^2\det(\Lambda_d)^{\frac{3}{2}}B^{1-\frac{2}{r}}}\\
 	 	&=O\left(K\varepsilon^2b^2\det(\Lambda_d)^{\frac{3}{2}}B^{1-\frac{2}{r}}\times\frac{\lambda_2-\mu_2\alpha}{\det(\Lambda_d)}KB+b \det(\Lambda_d)\log\left(\frac{\lambda_2-\mu_2\alpha}{\det(\Lambda_d)}KB\right)\right)\\
 	 	&=O(K^2 b^2\det(\Lambda_d)B^{2-\frac{2}{r}}+b\det(\Lambda_d)\log (KB))\\
 	 	&=O(b\det(\Lambda_d)\log (B)).
 	 \end{align*}
 	  	grâce à la condition \eqref{eq:cond4}.
 	 Quant à $R_2(\varepsilon,K,\Lambda_d,B)$, le même raisonnement que \eqref{eq:estimaten2n11} donne que 
 	 $$\sharp R_2(\varepsilon,K,\Lambda_d,B)=O(1).$$
 	 On en conclut que
 	 $$S(\varepsilon,K,\Lambda_d,B)=\frac{\varepsilon K^2}{2\det(\Lambda_d)}B^{2-\frac{1}{r}}+O\left(\frac{ K^2}{N \det(\Lambda_d)}B^{2-\frac{1}{r}}+b\det(\Lambda_d)N \log(B)\right)$$
 	  On choisit 
 	 $$N=\frac{2KB^{1-\frac{1}{2r}}}{b^\frac{1}{2}\det(\Lambda_d)}.$$
 	 Comme $\det(\Lambda_d)\leqslant d^2\det(\Lambda)$, la condition \eqref{eq:condd} assure que $N>1$. On conclut que
 	 $$S(\varepsilon,K,\Lambda_d,B)=\frac{\varepsilon K^2}{2\det(\Lambda_d)}B^{2-\frac{1}{r}}+O(Kb^\frac{1}{2}B^{1-\frac{1}{2r}}\log(B)).$$
 	 Maintenant on somme sur tous les $d$ petits \eqref{eq:condd}. D'abord faisons-le pour le terme principal. En rappelant la constante $\Theta(\Lambda)$ \eqref{eq:GammaLambda} et en remarquant que $\det(\Lambda_d)\geqslant d^2$,
 	 \begin{equation}\label{eq:error1}
 	 \begin{split}
 	 	&\sum_{1\leqslant d\leqslant K^\frac{1}{2}b^{-\frac{1}{4}}\det(\Lambda)^{-\frac{1}{2}}B^{\frac{1}{4}(2-\frac{1}{r})}} \mu(d)\frac{\varepsilon K^2}{2\det(\Lambda_d)}B^{2-\frac{1}{r}}\\
 	 	&=\frac{\Theta(\Lambda)\varepsilon K^2}{2} B^{2-\frac{1}{r}} +\sum_{d> K^\frac{1}{2}b^{-\frac{1}{4}}\det(\Lambda)^{-\frac{1}{2}}B^{\frac{1}{4}(2-\frac{1}{r})}} O\left(\frac{K^2}{d^2}B^{2-\frac{1}{r}}\right)\\
 	 	&=\frac{\Theta(\Lambda)\varepsilon K^2}{2} B^{2-\frac{1}{r}}+O\left(K^\frac{3}{2}b^\frac{1}{4}\det(\Lambda)^\frac{1}{2}B^{\frac{3}{4}(2-\frac{1}{r})}\right).
 	 	 	 \end{split}
 	 \end{equation}
 	 Ensuite pour le terme d'erreur
\begin{align}\label{eq:error2}
	 	 \sum_{1\leqslant d\leqslant K^\frac{1}{2}b^{-\frac{1}{4}}\det(\Lambda)^{-\frac{1}{2}}B^{\frac{1}{4}(2-\frac{1}{r})}}Kb^\frac{1}{2}B^{1-\frac{1}{2r}}\log(B)=O\left(\frac{K^\frac{3}{2}b^\frac{1}{4}}{\det(\Lambda)^\frac{1}{2}}B^{\frac{3}{4}(2-\frac{1}{r})}\log B\right).
\end{align}

 	 	 \textbf{Cas II. $d$ est grand.} Rappelons \eqref{eq:boundford22} et \eqref{eq:Dalpha}. Pour les $d$ vérifiant
 	 	 $$\frac{K^\frac{1}{2}}{b^\frac{1}{4}\det(\Lambda)^\frac{1}{2}}B^{\frac{1}{4}(2-\frac{1}{r})}<d\leqslant D(\alpha,K)B^{1-\frac{1}{2r}},$$ 
 	 	  on utilise la formule \eqref{eq:Tstrongalpha} de la Proposition \ref{po:smallepsilon} avec $N=1+\frac{4\varepsilon K}{\alpha d}$ puisque $\alpha=\sqrt{\frac{b}{a}}$ est quadratique, notant que les conditions \eqref{eq:cond4} \eqref{eq:cond5} impliquent
 	 	  $$(2^{21}\times 162)^2 K^4b^8\det(\Lambda)^6\alpha^4\varepsilon^4 B^{\frac{3}{r}-6}\leqslant(2^{21}\times 162)^2 K^{10} b^5\alpha^4\varepsilon^4\leqslant B^{\frac{7}{r}-10},$$
 	 	  en particulier
 	 	  $$(2^{21}\times 162)^2 K^4\varepsilon^4 \leqslant B^{\frac{4}{r}-4},$$
 	 	  et l'hypothèse \eqref{eq:epsiloncondition12} de la proposition est donc vérifiée. Rappelons aussi $\varDelta(\alpha)$ \eqref{eq:deltalpha} et $\Xi(\alpha)$ \eqref{eq:Xixi}. On obtient (quitte à remplacer $B$ par $\frac{B}{\alpha}$ et $\varepsilon$ par $\varepsilon \alpha^{\frac{1}{r}}$, on rappelle que le terme d'erreur peut dépendre de $\alpha$ et $\varepsilon$)
 	 	 \begin{align*}
 	 	 	\sharp S(\varepsilon,K,\Lambda_d,B)&
 	 	 	\leqslant \sharp \left\{
 	 	 	\begin{aligned}
 	 	 	(u^\prime,v^\prime)\in\NN^2
 	 	 	\end{aligned}
 	 	 	\left|
 	 	 	\begin{aligned}
 	 	 	&0<\frac{u^\prime}{v^\prime}-\alpha\leqslant \varepsilon B^{-\frac{1}{r}}\\
 	 	 	&v^\prime\leqslant \frac{KB }{ d}
 	 	 	\end{aligned}
 	 	 	\right\}\right. \\
 	 	 	&=O\left(\frac{K^2}{d^2}B^{2-\frac{1}{r}}+\varDelta(\alpha)\left(1+\frac{K}{d}\right) \log(KB))\right)\\
 	 	 	&=O\left(\frac{K^2}{d^2}B^{2-\frac{1}{r}}+b\left(1+\frac{K}{d}\right)\log B\right).
 	 	 \end{align*}
 	 	 Donc la contribution totale provenant de ces $d$ est d'ordre de grandeur, compte tenu de \eqref{eq:Dalpha},
 	 	 \begin{equation}\label{eq:error3}
 	 	 \begin{split}
 	 	 	 	 	 &\sum_{K^\frac{1}{2}b^{-\frac{1}{4}}\det(\Lambda)^{-\frac{1}{2}}B^{\frac{1}{4}(2-\frac{1}{r})}<d\leqslant D(\alpha,K)B^{1-\frac{1}{2r}}} \left(\frac{K^2}{d^2}B^{2-\frac{1}{r}}+b\left(1+\frac{K}{d}\right)\log B\right)\\
 	 	 	 	 	 &=O(K^\frac{3}{2}b^{\frac{1}{4}}\det(\Lambda)^\frac{1}{2} B^{\frac{3}{4}(2-\frac{1}{r})}+bD(\alpha,K)B^{1-\frac{1}{2r}}\log B+Kb\log (D(\alpha,K)B^{1-\frac{1}{2r}})\log B)\\
 	 	 	 	 	 &=O(K^\frac{3}{2}b^{\frac{1}{4}}\det(\Lambda)^\frac{1}{2} B^{\frac{3}{4}(2-\frac{1}{r})}+Kb^{\frac{3}{2}}B^{1-\frac{1}{2r}}\log B+Kb(\log B)^2)\\
 	 	 	 	 	 &=O(K^\frac{3}{2}b^{\frac{1}{4}}\det(\Lambda)^\frac{1}{2} B^{\frac{3}{4}(2-\frac{1}{r})}+Kb^{\frac{3}{2}}B^{1-\frac{1}{2r}}\log B)
 	  	 	 \end{split}
 	 	 \end{equation}
 	 	 La formule \eqref{eq:estimatefors} est déduite de \eqref{eq:error1}, \eqref{eq:error2} et \eqref{eq:error3}.		
  	 \end{proof}
	\subsection{Perspective}
	Le Théorème \ref{th:secondtheorem} ne couvre pas le zoom critique des nombres algébriques de degré plus grand que $2$. Les experts semblent penser que l'expansion en fraction continue d'un nombre algébrique $\alpha$ de degré $\geqslant 3$ n'a pas de quotients partiels bornés. C'est-à-dire que pour tout réel $C>0$,  il existe toujours des nombres rationnels $\frac{p}{q}$ vérifiant (cf. par exemple \cite{Chu})
	$$\left|\alpha -\frac{p}{q}\right|<\frac{C}{q^2}.$$ 
	En terme de la distribution locale (critique), 
	on pourrait interpréter cela comme: les nombres rationnels proches de $\alpha$ seraient \og beaucoup plus nombreux \fg\ que dans le cas quadratique et on n'aurait plus de phénomène de \og trou\fg\ .
	  
\section{Distribution locale sur la surface torique $Y_4$}\label{se:toricsurface}

On considère la surface torique $Y_4$ définie sur $\QQ$ obtenue en éclatant $\mathbb{P}^1\times\mathbb{P}^1$ en $4$ points invariants par l'action du tore:
$$P_1=[1:0]\times [1:0],\quad P_2=[0:1]\times[1:0],$$ 
$$P_3=[1:0]\times[0:1],\quad P_4=[0:1]\times[0:1].$$ 
On désigne par $\pi$ le morphisme d'éclatement. L'éventail de $Y_4$ est dans la Figure \ref{fig:y4}. 

\subsection{Géométrie de $Y_4$ et courbes rationnelles sur $Y_4$}

 On note $\mathcal{O}(1,0),\mathcal{O}(0,1)$ les tirés en arrière de sections hyperplans dans chaque $\mathbb{P}^1$, et $E_i~(1\leqslant i\leqslant 4)$ les diviseurs exceptionnels tels que $E_i=\pi^{-1}(P_i)$.
 En dehors de $\cup_{i=1}^4 E_i$ on utilise encore les coordonnées $[x:y]\times[s:t]$ de $\mathbb{P}^1\times\mathbb{P}^1$. 

Parmi les courbes qui rencontrent l'orbite ouverte, il y en a $4$ familles de degré anticanonique $2$ dont les classes sont celles de
$$\mathcal{O}(1,0),\quad \mathcal{O}(0,1),\quad \mathcal{O}(1,1)-E_2-E_3,\quad \mathcal{O}(1,1)-E_1-E_4$$
dans le groupe de Picard. Appartenant à chacun de ces fibrés il existe une unique courbe irréductible passant par $Q=[1:1]\times[1:1]$ d'équations respectives
\begin{equation}\label{fourspeciallines}
x=y,\quad s=t,\quad xs=yt,\quad \text{et}\quad xt=ys.
\end{equation}
On les appellera courbes spéciales et l'on les notera $Z_i(1\leqslant i\leqslant 4)$ 
Il y a $4$ familles de courbes rationnelles de degré $3$ relativement au fibré anticanonique passant par $Q$ qui sont lisses, dont les classes sont
\begin{equation}\label{eq:generallines}
\mathcal{O}(1,1)-E_i,\quad (1\leqslant i\leqslant 4).
\end{equation}
 Le diviseur anticanonique $$\omega_{Y_4}^{-1}=\mathcal{O}(2,2)-E_1-E_2-E_3-E_4,$$ dont les sections globales sont des combinaisons linéaires des monômes
$$x^2 st,\quad y^2 st,\quad t^2 xy,\quad s^2 xy,\quad xyst,$$ est gros et engendré par ses sections globales mais il n'est pas ample (car, les fibrés $L_1=\mathcal{O}(1,0)-E_3-E_4,L_2=\mathcal{O}(0,1)-E_1-E_2$ représentent des courbes effectives et $\langle\omega_{Y_4}^{-1},L_1\rangle=\langle\omega_{Y_4}^{-1},L_2\rangle=0$). Il vérifie donc la propriété de Northcott dans l'orbite ouverte (cf. Définition \ref{def:Northcott}) car il définit un morphisme $Y_4\to \PP^4$ qui est un isomorphisme autour de $Q$ (cf. Section \ref{se:thinset} \emph{infra}).
Sur l'ouvert
$(s\neq 0)\cap( x\neq 0)$, en utilisant les coordonnées $(w,z)=(\frac{y}{x},\frac{t}{s})$, on identifie localement l'espace tangent en $Q$ à un voisinage de $(0,0)\in\RR^2$ par le difféomorphisme 
\begin{equation}\label{eq:diffeomorphism}
\rho:[x:y]\times [s:t]\longmapsto \left(\frac{y}{x}-1,\frac{t}{s}-1\right)=(w-1,z-1).
\end{equation}
Avec cette identification, les $4$ courbes spéciales s'écrivent comme $3$ droites et une hyperbole:
\begin{equation}\label{eq:fourlines}
z=1,\quad w=1,\quad z=w,\quad zw=1.
\end{equation}
Maintenant on associe une hauteur de Weil à $\omega_{Y_4}^{-1}$. Tout d'abord pour un point 
$P=[x:y]\times[s:t]\in (Y_4\setminus (\cup_{i=1}^4 E_i))(\QQ)$ avec $x,y,s,t\in\ZZ,\pgcd(x,y)=\pgcd(s,t)=1$ n'appartenant pas aux diviseurs exceptionnels,
\begin{align*}
&\pgcd(x^2 st,y^2 st,t^2 xy,s^2 xy,xyst)\\
=&\pgcd(st\pgcd(x^2,y^2),xy\pgcd(t^2,s^2),xyst)\\
=&\pgcd(st,xy)\\
=&\pgcd(x,s)\pgcd(x,t)\pgcd(y,s)\pgcd(y,t).
\end{align*}  
Donc on peut prendre comme hauteur de Weil
$$\h_{\omega_{Y_4}^{-1}}(P)=\frac{\max(|x^2 st|,| y^2 st|,| t^2 xy|,|s^2 xy|,|xyst|)}{\pgcd(x^2 st,y^2 st,t^2 xy,s^2 xy,xyst)}
=\frac{\max(|x^2 st|,| y^2 st|,| t^2 xy|,|s^2 xy|)}{\pgcd(x,s)\pgcd(x,t)\pgcd(y,s)\pgcd(y,t)}.$$
Les courbes \eqref{eq:fourlines} divisent $\RR^2$ en $8$ régions. Puisque la surface $Y_4$ possède un \og gros\fg{}\ groupe d'automorphisme, dans lequel les permutations de coordonnées échangent les points dans ces $8$ régions tout en préservant la hauteur $\h_{\omega_{Y_4}^{-1}}$. On peut donc se ramener à l'une des ces régions
\begin{equation}\label{eq:theregionR}
R=\{(w,z)\in\RR^2:z>w>1\}.
\end{equation}
On note
\begin{equation}\label{eq:theregionV}
V=\rho^{-1}(R)\subset U=Y_4\setminus\cup_{i=1}^4 Z_i.
\end{equation}
Pour $P=[x:y]\times[s:t]\in V$, la hauteur se calcule comme
$$\h_{\omega_{Y_4}^{-1}}(P)=\frac{t^2 xy}{\pgcd(x,s)\pgcd(x,t)\pgcd(y,s)\pgcd(y,t)}.$$
La distance que l'on va utiliser est 
\begin{equation}\label{eq:dist}
\dist(P)=\dist(\rho(P),\rho(Q))=\max(|z-1|,|w-1|)=z-1.
\end{equation}
\subsection{Détermination des constantes d'approximation}\label{se:cab}
\subsubsection{Borne inférieure uniforme}
Nous allons montrer que la meilleure constante d'approximation est $2$ par une estimation directe.
\begin{proposition}\label{po:lowerbound1}
	$$\alpha(Q,Y_4)=2.$$
\end{proposition}
\begin{proof}
	 Pour $P=[x:y]\times[s:t]\neq Q$ satisfaisant à $t\neq s$ (les cas où $x\neq y$ se démontre de façon analogue), on a
	\begin{equation}\label{eq:htdistlowery4}
	\begin{split}
	 \h_{\omega_{Y_4}^{-1}}(P) d(P)^2&\geqslant\frac{t^2 xy}{\pgcd(x,s)\pgcd(x,t)\pgcd(y,s)\pgcd(y,t)} \left(\frac{t}{s}-1\right)^2\\
	&=\frac{ xy}{\pgcd(x,s)\pgcd(x,t)\pgcd(y,s)\pgcd(y,t)}\frac{t^2}{s^2}(t-s)^2\\
	&\geqslant 1.
	\end{split}
	\end{equation}
	Cela montre que $\alpha(Q,Y_4)\geqslant 2$ (cf. Proposition-Définition \ref{pd:appconstant}), en rappelant que $\omega_{Y_4}^{-1}$ vérifie la propriété de Northcott. Mais les courbes spéciales $Z_i$ \eqref{fourspeciallines} donnent des sous-variétés avec $\alpha(Q,Z_i)=2$ (remarque \ref{rmk:thmofmckinnon}).
	Ceci clôt la démonstration.
\end{proof}
\begin{remark*}
	Cette majoration uniforme indique un phénomène de \og trou\fg\ en dimension $2$ dans le zoom critique. C'est-à-dire, il existe $\varepsilon_0>0$ tel que pour tout $\varepsilon<\varepsilon_0$, et pour tout $B$, on ait $\delta_{U,Q,B,\frac{1}{2}}(\chi(\varepsilon))=0$. Ceci peut aussi s'interpréter comme une inégalité du type Liouville: $$d(P)\geqslant \frac{1}{ \h_{\omega_{Y_4}^{-1}}(P) ^\frac{1}{2}}, \quad \forall P\in Y_4(\QQ)\setminus \{Q\}.$$
\end{remark*}
\subsubsection{Constante d'approximation essentielle}
Considérons les relevés stricts des courbes $C_{a,b}$ dans $Y_4$ définie par
$$C_{a,b}:axy(t-s)^2=bst(y-x)^2 \quad (a,b)\in \NN^{*2},\pgcd(a,b)=1.$$
Avec les coordonnées $(w,z)=\left(\frac{y}{x},\frac{t}{s}\right)$, cette équation s'écrit
$$aw(z-1)^2=bz(w-1)^2.$$
Ces équations définissent donc une famille de courbes cubiques dans $\PP^1\times \PP^1$ passant par les $4$ points invariants $P_i,1\leqslant i\leqslant 4$ avec un point singulier en $Q$. Donc la classe de $C_{a,b}$ dans le groupe de Picard est la même que celle de $\omega_{Y_4}^{-1}$ et on a
$$\deg_{\omega_{Y_4}^{-1}}C_{a,b}=4.$$ Si $a=b$, la courbe
$$xy(t-s)^2=st(y-x)^2$$
a en fait deux composantes irréductibles
$$yt=xs\quad \text{et}\quad xt=ys.$$
\begin{lemma}
Lorsque $a\neq b$, la courbe $\cab$ est une courbe cubique géométriquement intègre et nodale en $Q$. 
\end{lemma}
\begin{proof}
	Si $a\neq b$, le polynôme $aw(z-1)^2=bz(w-1)^2$ est irréductible sur $\CC[z,w]$.
	Une courbe cubique intègre singulière ne peut pas avoir qu'un point singulier qui est nodal ou cuspidal. 
	Les tangentes au point $Q$ sont de pente
	$$\lim_{(w,z)\to(1,1)}\frac{z-1}{w-1}=\pm \sqrt{\frac{b}{a}}\lim_{(w,z)\to(1,1)}\sqrt{\left|\frac{z}{w}\right|}=\pm \sqrt{\frac{b}{a}}.$$
	Donc $Q$ est un point nodal de $\cab$.
\end{proof}
\begin{lemma}\label{le:curvescoverallpoints}
	Pour tout point rationnel $P=(w_0,z_0)\in R$ où $z_0,w_0\in\QQ$, il existe une unique courbe $\cab$ passant par $\rho^{-1}(P)$.
\end{lemma}
\begin{proof}
	On écrit (de façon unique) le quotient en une fraction positive réduite:
\begin{equation}\label{eq:cabparametrisation}
	\frac{z_0(w_0-1)^2}{w_0(z_0-1)^2}=\frac{a_0}{b_0}.
\end{equation}
	Alors $\rho^{-1}(P)\in C_{a_0,b_0}(\QQ)$.
\end{proof}

Géométriquement, on considère $\widetilde{\pi}:\widetilde{Y_4}\to Y_4$ l'éclatement de $Y_4$ en $Q$, et on note $\widetilde{C_{a,b}}$ la transformation stricte de $\cab$. Alors $\widetilde{C_{a,b}}$ est une normalisation de $\cab$ et $\widetilde{\pi}^{-1}(Q)=\{Q_1,Q_2\}$. 
On peut voir $\pm\sqrt {\frac{b}{a}}$ comme l'une des coordonnées de $Q_i$.
Approcher $Q$ sur la courbe $C_{a,b}$ revient à approcher l'un des $Q_1,Q_2$ sur $\widetilde{\cab}$, ceci étant équivalent à approcher $\sqrt{\frac{b}{a}}$ ou $-\sqrt{\frac{b}{a}}$ par des points rationnels sur $\PP^1_\QQ$.
\begin{center}
	\textbf{Cas I.} $a,b$ sont deux carrés de nombres entiers.\end{center}
Alors $\sqrt{\frac{b}{a}}\in\QQ$ et $Q_1,Q_2\in\widetilde{C_{a,b}}(\QQ)$. Il s'agit d'approximer un nombre rationnel par des nombres rationnels. La remarque \ref{rmk:thmofmckinnon} dit que
\begin{equation}\label{eq:absquare}
\alpha(Q,C_{a,b})=\alpha(Q_i,\widetilde{\cab})=\deg_{\omega_{Y_4}^{-1}}C_{a,b}=4.
\end{equation}
\begin{center}\textbf{Cas II.}
	l'un des $a,b$ n'est pas le carré d'un entier, (noté $(a,b)\not\in\square^2$)
\end{center}
Comme $a$ et $b$ sont premiers entre eux, $\sqrt{\frac{b}{a}}\not\in \QQ$.
Aucun des points $Q_1,Q_2$ n'est défini sur $\QQ$. 
L'approximation du point $Q$ le long une branche de $\cab$ est équivalente à l'approximation du point \textit{irrationnel quadratique réel} $Q_i$ (dépendant de la branche choisie) le long $\widetilde{C_{a,b}}$. Donc dans ce cas d'après la remarque \ref{rmk:thmofmckinnon},
\begin{equation}\label{eq:abnotsquare}
\alpha(Q,C_{a,b})=\alpha(Q_i,\widetilde{\cab})=\frac{\deg_{\omega_{Y_4}^{-1}}C_{a,b}}{2}=2.
\end{equation}
Cependant, pour les $4$ courbes spéciales $Z_i$, l'approximation au point $Q$ équivaut à une approximation d'un point rationnel défini sur $\QQ$ dans $\PP^1$, à savoir $\alpha(Q,Z_i) =2$. Bien que les valeur de leurs constante d'approximations soient la même, le nombre des points rationnels que l'on trouve dans l'opération de zoom sont de grandeur très différente.
C'est cette différence qui explique les phénomènes radicalement différents entre les deux types de courbes ci-dessus.
On en conclut 
\begin{theorem}\label{th:approximationessential}
	On a
	$$\aess(Q)=\alpha(Q,Y_4)=2.$$
	Par conséquent, il n'y a pas de sous-variétés localement accumulatrices (Définition \ref{df:essential}).
\end{theorem}
\begin{proof}
	Cela résulte du fait que l'ensemble des points rationnels dans $\cab$ dont le paramètre $a,b$ vérifient la condition dans le cas II est dense (même pour la topologie analytique)
	et de la borne inférieure que l'on a établie précédemment (Proposition \ref{po:lowerbound1})
\end{proof}
Donc la famille $(\cab)_{(a,b)\not\in\square^2}$ donnent une autre façon de paramétrer les points rationnels bien adaptée à notre problème car la constante d'approximation sur celles correspondant au cas II est $2$, plus petite que celle donnée par les droites générales, i.e. les sections irréductibles de \eqref{eq:generallines}, valant $3$.
\subsection{Zoom sur la surface $Y_4$}
Nous utiliserons les courbes nodales $\cab$ pour paramétrer les points rationnels autour de $Q$. Nous avons vu qu'elles n'ont pas la même constante d'approximation et nous avons constaté que l'approximation essentielle est prise sur celles vérifiant la condition $(a,b)\not\in\square^2$. Regardons d'abord ce que la prédiction naïve nous donne (cf. \eqref{eq:naivepredictionofpowers} dans \S\ref{se:zoomoperation}). On note $r$ le facteur de zoom et $B$ la borne de la hauteur. Alors dans le zoom de facteur $r$ le nombre de points rationnels dans le voisinage de diamètre $\varepsilon$ devrait être 
\begin{equation}\label{eq:prediction}
c_{Y_4} B(\log B)^{\operatorname{rg}\operatorname{Pic}(Y_4)-1} \times O(B^{-\frac{\dim Y_4}{r}})=O(B^{1-\frac{2}{r}}(\log B)^5).
\end{equation}
Si $r>2=\aess(Q,Y_4)$, on devrait pouvoir trouver \og beaucoup\fg{}\ de points, c'est-à-dire de cardinal au moins une puissance de $B$. Si $r=2$, \textit{a priori} cette heuristique prédit qu'il existe \og très peu\fg{}\ de points dans ce voisinage (le nombre étant d'ordre de grandeur $\log $). Nous allons montrer qu'en fait il y a au plus un nombre fini de points en dehors des courbes spéciales dans ce dernier cas.
\subsubsection{Paramétrage par des courbes nodales}\label{se:parametrisation}
Nous nous bornons à la région $R\subset \RR^2$ \eqref{eq:theregionR}. Puisque les droites $w=z, w=1$
forment le bord de $R$, il reste à trouver la distribution dans l'intérieur de $R$. Grâce au Lemme \ref{le:curvescoverallpoints}, les courbes nodales donnent un paramétrage local des points rationnels. Donc le dénombrement des points rationnels dans l'intérieur de $R$ peut se décomposer en des sommes des points sur chaque $\cab$. On va calculer la formule de la hauteur restreinte à $\cab$ en choisissant un paramétrage rationnel par les coordonnées de $\PP^1$.
 
Par la formule \eqref{eq:cabparametrisation}, seulement les courbes $\cab$ avec $a<b$ intersectent $R$ puisque $z_0>w_0>1$ implique 
$$z_0(w_0-1)^2<w_0(z_0-1)^2.$$
On fixe un couple $(a,b)\in \NN_{\operatorname{prem}}^{*2}$ vérifiant la condition $a<b$.
D'abord on change les coordonnées $(w,z)$ en $(w^\prime,z^\prime)=(w-1,z-1)$.
Alors l'équation de $\cab$ devient 
\begin{equation}\label{eq:cabR}
\cab: a z^ {\prime 2}(w^\prime+1)=bw^{\prime 2}(z^\prime+1).
\end{equation}
D'après le principe de Bézout, une droite générale intersecte $\cab$ en $3$ points (comptant la multiplicité). 
La droite $D_\lambda:z^\prime=\lambda w^\prime$ passe par $(0,0)$ 
La multiplicité d'intersection $\cab\cap D_\lambda $ en $(0,0)$ est $2$. Donc cette droite intersecte $\cab$ en un autre point rationnel $(w^\prime_\lambda,z^\prime_\lambda)$ différent de $(0,0)$ avec $z^\prime_\lambda>w^\prime_\lambda>0$. Un calcul nous donne
\begin{equation}\label{eq:theparametrization}
w^\prime_\lambda=\frac{a\lambda^2-b}{\lambda(b-\lambda a)},\quad z^\prime_\lambda=\frac{a\lambda^2-b}{b-\lambda a},\quad \left(\sqrt{\frac{b}{a}}<\lambda<\frac{b}{a}\right).
\end{equation}
Ceci nous permet de retrouver le paramétrage sous les coordonnées $(w,z)$:
\begin{equation*}
w_\lambda=\frac{b(\lambda-1)}{\lambda(b-\lambda a)},\quad z_\lambda=\frac{\lambda a(\lambda-1)}{b-\lambda a},\quad\left(\sqrt{\frac{b}{a}}<\lambda<\frac{b}{a}\right).
\end{equation*}
Donc on obtient un morphisme de paramétrage
$\psi_{a,b}:\PP^1\longrightarrow \cab$ défini pour $(u,v)\in\NN^{*2}_{\operatorname{prem}}$,
\begin{equation}\label{eq:parapsiab}
(u,v)\longmapsto \left(\frac{b(\frac{u}{v}-1)}{\frac{u}{v}(b-\frac{u}{v} a)},\frac{\frac{u}{v} a(\frac{u}{v}-1)}{b-\frac{u}{v} a}\right)=\left(\frac{bv(u-v)}{u(bv-ua)},\frac{ua(u-v)}{v(bv-ua)}\right).
\end{equation}
On introduit les notations
\begin{align}
&d_1=\pgcd(u,b),\quad d_2=\pgcd(v,a),\quad d_3=\pgcd(u-v,b-a);\nonumber\\
& D_1=\pgcd(u^2,b),\quad D_2=\pgcd(v^2,a).\label{eq:D1D2}
\end{align}
Alors comme l'on a supposé que $\pgcd(u,v)=\pgcd(a,b)=1$,
\begin{align}
&\pgcd(bv(u-v),u(bv-ua))\nonumber\\
=&\pgcd(b,u(bv-ua))\pgcd(v,bv-ua)\pgcd(u-v,bv-ua)\label{eq:pgcd1}\\
=&\pgcd(u^2,b)\pgcd(v,a)\pgcd(u-v,b-a)\nonumber\\
=&D_1d_2d_3\nonumber;
\end{align}
\begin{align}
&\pgcd(ua(u-v),v(bv-ua))\nonumber\\
=&\pgcd(a,v(bv-ua))\pgcd(u,bv-ua)\pgcd(u-v,bv-ua)\label{eq:pgcd2}\\
=&\pgcd(v^2,a)\pgcd(u,b)\pgcd(u-v,b-a)\nonumber\\
=&d_1D_2d_3\nonumber.
\end{align}
L'égalité \eqref{eq:pgcd1} se découle car
$$\pgcd(b,u(bv-ua))=\pgcd(b,u^2)\mid u^2,\quad \pgcd(v,u(bv-ua))\mid v,\quad \pgcd(u-v,u(bv-ua))\mid u-v$$
sont premiers deux à deux. Le même raisonnement s'applique à \eqref{eq:pgcd2}.
On trouve le paramétrage suivant pour les coordonnées primitives dans $\PP^1\times\PP^1$.
\begin{equation}
x=\frac{u(bv-ua)}{D_1d_2d_3},\quad y=\frac{bv(u-v)}{D_1d_2d_3},\quad s=\frac{v(bv-ua)}{d_1D_2d_3},\quad t=\frac{ua(u-v)}{d_1D_2d_3}.
\end{equation}
On a alors, puisque $D_1\mid d_1^2,D_2\mid d_2^2$,
\begin{align*}
	\pgcd(y,t)=\frac{u-v}{d_3}\pgcd\left(\frac{bv}{D_1d_2},\frac{ua}{d_1D_2}\right)=\frac{u-v}{d_3};
\end{align*}
$$\pgcd(y,s)=\frac{v}{d_2}\pgcd\left(\frac{b(u-v)}{D_1d_3},\frac{(bv-ua)d_2}{d_1D_2d_3}\right)=\frac{v}{d_2};$$
$$\pgcd(x,t)=\frac{u}{d_1}\pgcd\left(\frac{(bv-ua)d_1}{D_1d_2d_3},\frac{a(u-v)}{D_2d_3}\right)=\frac{u}{d_1};$$
$$\pgcd(x,s)=\frac{bv-ua}{d_1d_2d_3}\pgcd\left(\frac{ud_1}{D_1},\frac{vd_2}{D_2}\right)=\frac{bv-ua}{d_1d_2d_3}.$$
Enfin on obtient la formule de la hauteur restreinte à $\cab(\QQ)$ induite par l'image inverse sur $\PP^1$ dans la région $R$: pour $[u:v]\in \PP^1(\ZZ)$ satisfaisant à $\sqrt{\frac{b}{a}}<\frac{u}{v}<\frac{b}{a}$, 
\begin{align*}
\h_{\omega_{Y_4}^{-1}}(\psi_{a,b}([u:v]))&=\frac{t^2 xy}{\pgcd(x,s)\pgcd(x,t)\pgcd(y,s)\pgcd(y,t)}\\
&=\frac{b(ua(u-v))^2}{(D_1D_2d_3)^2}.
\end{align*}

Pour $\varepsilon>0,B>0$ fixés, on prend la fonction de test 
$$\chi(\varepsilon)=\mathtt{1}(\{(w^\prime,z^\prime)\in\RR^2:\max(|w^\prime|,|z^\prime|)\leqslant \varepsilon\})$$ 
et l'on regarde la distribution locale dans le voisinage de diamètre $\varepsilon$ intersectant la région $R$, c'est-à-dire la quantité
\begin{equation}\label{eq:thesetEepsilonB}
\delta_{V,Q,B,r}(\chi(\varepsilon))=\sharp \left\{
\begin{aligned}
&P=[x:y]\times[s:t]\\
&x,y,s,t\in\NN^*;\frac{t}{s}>\frac{x}{y}>1\\
&\pgcd(x,y)=\pgcd(s,t)=1
\end{aligned}
\left|
\begin{aligned}
&B^{\frac{1}{r}}\dist(P)=B^{\frac{1}{r}}\left(\frac{t}{s}-1\right)\leqslant \varepsilon \\
&\h_{\omega_{Y_4}^{-1}}(P)=\frac{t^2 xy}{\pgcd(x,s)\pgcd(x,t)\pgcd(y,s)\pgcd(y,t)}\leqslant B
\end{aligned}
\right\}\right..
\end{equation}
On le décrit comme un problème de dénombrement. 
Pour un couple $(a,b)\in \NN_{\operatorname{prem}}^{*2}$ satisfaisant à $a<b$, considérons l'ensemble $E(a,b,\varepsilon,B,r)$ des $(u,v)\in\NN_{\operatorname{prem}}^{*2}$ vérifiant (cf. \eqref{eq:theparametrization})
\begin{equation}\label{eq:zoomoncabdistance}
\sqrt{\frac{b}{a}}<\frac{u}{v}<\frac{b}{a},\quad B^{\frac{1}{r}}\dist(P)=B^{\frac{1}{r}}\frac{\frac{u^2}{v^2} -\frac{b}{a}}{\frac{b}{a}-\frac{u}{v} }\leqslant \varepsilon ,
\end{equation}
\begin{equation}\label{eq:heightoncab}
\h_{\omega_{Y_4}^{-1}}(\psi_{a,b}[u:v])=\frac{b(ua(u-v))^2}{(D_1D_2d_3)^2}\leqslant B .
\end{equation}
Alors
\begin{equation}\label{eq:decompositiononcab}
\delta_{V,Q,B,r}(\chi(\varepsilon))=\sum_{a,b\in\NN_{\operatorname{prem}}^{*2},a<b}\sharp E(a,b,\varepsilon,B,r).
\end{equation}

Avant tout on poursuit quelques manipulations simples pour obtenir des encadrements des paramètres.
En factorisant $v$, l'inégalité \eqref{eq:heightoncab} s'écrit 
\begin{equation*}
v^4ba^2\left(\left(\frac{u}{v}\right)^2\left(\frac{u}{v}-1\right)^2\right)\leqslant BD_1^2D_2^2d_3^2.
\end{equation*}
Comme on a supposé que $u\neq v$, on a donc
\begin{align}\label{eq:heightboundforv}
v^4&\leqslant\frac{ BD_1^2D_2^2d_3^2}{ba^2\left(\left(\frac{u}{v}\right)^2\left(\frac{u}{v}-1\right)^2\right)}\leqslant \frac{ BD_1^2D_2^2d_3^2}{b^2a\left(\sqrt{\frac{b}{a}}-1\right)^2}.
\end{align}
Donc en combinant avec \eqref{eq:zoomoncabdistance}, on en déduit
\begin{align*}
|au^2-bv^2|\leqslant \varepsilon av^2\left(\frac{b}{a}-\frac{u}{v}\right) B^{-\frac{1}{r}}
\leqslant \varepsilon a\left(\frac{b}{a}-\sqrt{\frac{b}{a}}\right)\frac{D_1D_2d_3}{b\sqrt{a}\left(\sqrt{\frac{b}{a}}-1\right)}B^{\frac{1}{2}-\frac{1}{r}}
=\frac{\varepsilon D_1D_2d_3}{\sqrt{b}}B^{\frac{1}{2}-\frac{1}{r}},
\end{align*}
d'où
\begin{equation}\label{eq:thenumberofnodalcurves}
b\leqslant \varepsilon^2 \left(\frac{|au^2-bv^2|}{D_1D_2d_3}\right)^{-2}B^{1-\frac{2}{r}}.
\end{equation}

La condition \eqref{eq:zoomoncabdistance} implique aussi
\begin{equation}\label{eq:distinduced}
\frac{u}{v}-\sqrt{\frac{b}{a}}\leqslant \frac{ba^{-1}-uv^{-1}}{uv^{-1}+\sqrt{ba^{-1}}}\varepsilon B^{-\frac{1}{r}}\leqslant \frac{ba^{-1}-\sqrt{ba^{-1}}}{2\sqrt{ba^{-1}}}\varepsilon B^{-\frac{1}{r}}=\frac{1}{2}\left(\sqrt{\frac{b}{a}}-1\right)\varepsilon B^{-\frac{1}{r}}.
\end{equation}

\subsubsection{Zoom critique: la finitude}
\begin{theorem}\label{th:finitenessofcriticzoom}
	Pour tout $\varepsilon>0$, on a que, pour tout $B\gg_{\varepsilon}1$,
	$$\delta_{V,Q,B,2}(\chi(\varepsilon))=\sum_{\substack{(a,b)\not\in\square^2\\
			a<b\leqslant \varepsilon^2}}\sharp E(a,b,\varepsilon,B,r)=O_\varepsilon(1).$$
Par conséquent, la sous-variété $\cup_{i=1}^4 Z_i$ est celle localement faiblement accumulatrice (définition \ref{df:weaklocal}). Il n'existe pas de mesure limite pour le zoom critique. 
\end{theorem}
\begin{proof}
On va démontrer que, pour $\varepsilon>0$ fixé, en utilisant le paramétrage ci-dessus, dans le cas critique ($r=2$), il n'y a qu'un nombre fini de courbes $\cab$ qui interviennent. 

Notons que $D_1,D_2,d_3$ sont premiers deux à deux. Une observation importante est que
$$D_1D_2 | au^2-bv^2,\quad d_3|a(u^2-v^2)-(b-a)v^2=au^2-bv^2,$$
et donc $$D_1D_2d_3|au^2-bv^2.$$ 
Compte tenu de \eqref{eq:thenumberofnodalcurves}, ceci nous donne la majoration du paramètre $b$ pour $r=2$: 
$$b\leqslant \varepsilon^2 \left(\frac{|au^2-bv^2|}{D_1D_2d_3}\right)^{-2}\leqslant \varepsilon^2,$$
ce qui signifie que dans un voisinage fixé après zoom, le nombre de courbes nodales est uniformément majoré, à savoir
$$\delta_{V,Q,B,2}(\chi(\varepsilon))=\sum_{a<b\leqslant \varepsilon^2}\sharp E(a,b,\varepsilon,B,r).$$
Parmi ces courbes $\cab$ qui interviennent, si $(a,b)\in\square^2$, on a $\alpha(Q,\cab)=4$ \eqref{eq:absquare}. Alors la Proposition \ref{po:weakzoom} nous donne que
$$\delta_{\cab\cap V,Q,B,2}(\chi(\varepsilon))=\sharp E(a,b,\varepsilon,B,2)=0$$ pour tout $B\gg_{\varepsilon}1$.
Pour tous les $(a,b)\not\in\square^2$, sur $\PP^1$ (l'image inverse de $\psi_{a,b}$) on a une distance et une hauteur induites de celles sur $Y_4$, la hauteur étant équivalente à la hauteur de Weil canonique $\h_{\mathcal{O}_{\PP^1}(4)}$ (cf. \eqref{eq:zoomoncabdistance}, \eqref{eq:heightoncab}). Comme \eqref{eq:abnotsquare}, le zoom induit sur $\cab$ avec le facteur $r=2$ est critique sur $\cab$.
En prenant $r=2$ dans \eqref{eq:distinduced}, la majoration découle donc de la Proposition \ref{po:controlofepsilon}.
D'après le Théorème \ref{th:criticzoomforquadratics}, pour tout $\varepsilon>\eta>0$ suffisamment proches, on a
$$\liminf_{B} \delta_{V,Q,B,2}(\chi(\varepsilon,\eta))=\liminf_{B} \delta_{\bigcup_{\substack{(a,b)\not\in\square^2\\b\leqslant \varepsilon^2}}\cab\cap V,Q,B,2}(\chi(\varepsilon,\eta))=0,$$
alors que pour tout $(a,b)\not\in\square^2$, il existe certains couples $(\varepsilon,\eta)$ tels que
$$ \limsup_{B}\delta_{V,Q,B,2}(\chi(\varepsilon,\eta))\geqslant\limsup_{B}\delta_{\cab\cap V,Q,B,2}(\chi(\varepsilon,\eta))>0.$$
Cela démontre la non-existence de mesure limite.
Le fait que $Z_i$ est localement faiblement accumulatrice découle du Théorème \ref{th:pagelot} puisque $\h_{\omega_{Y_4}^{-1}}|_{Z_i}\sim \h_{\mathcal{O}_{\PP^1}(2)}$ sur $Z_i$ et donc
\begin{equation*}
\delta_{Z_i,Q,B,2}(\chi(\varepsilon))\gg\ll_\varepsilon B^{\frac{1}{2}}.\qedhere
\end{equation*}
\end{proof}

\subsubsection{Zoom sous-critique: une borne inférieure}
 On va démontrer une borne inférieure du nombre de points rationnels dans le zoom sous-critique ($r>2$) basé sur le fait que sur certaines courbes nodales on trouve \og beaucoup\fg\ de points avec la hauteur et la distance induites (Théorème \ref{th:zoomingonlattices}). 
 
 Fixons $\varepsilon_1>\varepsilon_2>0,\tau_1>\tau_2>1$. Nous nous bornons à la région
\begin{equation}\label{eq:regionW}
W=W(\varepsilon_1,\varepsilon_2,\tau_1,\tau_2)=\left\{(w^\prime,z^\prime)\in\RR^2:\varepsilon_2<z^\prime\leqslant\varepsilon_1,\tau_2\leqslant\frac{z^\prime}{w^\prime}\leqslant\tau_1\right\}\subset R.
\end{equation}
L'énoncé précis est le suivant. Rappelons que $V=\rho^{-1}(R)\subset Y_4\setminus \cup_{i=1}^4 Z_i$.
\begin{theorem}\label{th:lowerbound}
	Pour \begin{equation}\label{eq:conditionforeta}
	2<r<\frac{144}{55},\quad 0<\eta< \frac{1}{35},
	\end{equation}on a
	\begin{align*}
		&\delta_{V,Q,B,r}(\chi(W(\varepsilon_1,\varepsilon_2,\tau_1,\tau_2)))\\
		\geqslant &B^{(1+\eta)\left(\frac{1}{2}-\frac{1}{r}\right)}(\log B)^3 \left( C_2 \int\chi(W(\varepsilon_1,\varepsilon_2,\tau_1,\tau_2))\frac{\operatorname{d}w^\prime\operatorname{d}z^\prime}{z^\prime}+O_{\tau_i,\varepsilon_i}\left(\frac{1}{\log B}\right)\right),
	\end{align*}
	où 
\begin{equation}\label{eq:C2}
	C_2=C_2(r)=\frac{6}{\pi^2}\left(\eta\left(1-\frac{2}{r}\right)\right)^3\prod_{p}\left(1-\frac{1}{p}\right)^3\left(1+\frac{3}{p}-\frac{1}{p^2}-\frac{18}{p(p+2)}\right).
\end{equation}
\end{theorem}
\subsubsubsection{Comptage sur les $\cab$}
Au vu de \eqref{eq:decompositiononcab}, on va appliquer le Théorème \ref{th:zoomingonlattices} pour compter le nombre de points dans le zoom de facteur $r$ induit sur chaque $\cab$, 
avec un terme d'erreur précis. 

\begin{proposition}\label{po:zoompointsoncab}
	Supposons que 
	\begin{equation}\label{eq:conditionforr}
		2<r<\frac{144}{55}.
	\end{equation}
	Alors pour tout $\varepsilon>0$, pour tout couple $(a,b)\in\NN^{*2}$ satisfaisant à 
	\begin{equation}\label{eq:conditiononab}
	\pgcd(a,b)=1,\quad a<b,\quad (a,b)\not\in\square^2,
	\end{equation} 
	et
	\begin{equation}\label{eq:conditionforbandr1}
\tau_2^2<ba^{-1}<\tau_1^2\quad\text{et}\quad b ^{\frac{17}{2}}\leqslant \frac{16 \tau_2}{\tau_1-1}B^{\frac{1}{2}-\frac{1}{r}}.
	\end{equation} 
	on a (rappelons l'ensemble $E(a,b,\varepsilon,B,r)$ défini par \eqref{eq:zoomoncabdistance} et \eqref{eq:heightoncab} et les fonctions $\Psi$ \eqref{eq:arithmeticfunctionkey} et $\Phi$ \eqref{eq:arithmeticfunctionPhi})
	\begin{equation}\label{eq:asymptoticoncab}
	\begin{split}
	\sharp E(a,b,\varepsilon,B,r)
	=\frac{3}{2\pi^2}\frac{\Phi(b)\Phi(a)\Psi(b-a)}{ba^{\frac{1}{2}}}\varepsilon B^{\frac{1}{2}-\frac{1}{r}}+O_{\tau_i,\varepsilon,\delta}\left(b^{\frac{23}{8}+\delta}B^{\frac{3}{4}(\frac{1}{2}-\frac{1}{r})}\log B \right)
	\end{split}
		\end{equation}
		pour tout $\delta>0$.
\end{proposition}
\begin{remark*}
	Si l'on identifie localement une branche de $\cab$ avec l'espace tangent de $\PP^1$ en le point $\sqrt{\frac{b}{a}}$, la Proposition \ref{po:zoompointsoncab} indique qu'il existe une équidistribution locale des points rationnels sur chaque $\cab$ (avec $(a,b)\not\in\square^2$) avec la hauteur et la distance induites (cf. \eqref{eq:zoomoncabdistance}, \eqref{eq:heightoncab}).
\end{remark*}
On utilisera le lemme suivant dû à Heath-Brown \cite[Lemma 2]{H} nous permet de contrôler le nombre des points entiers primitifs dans une région planaire autour de l'origine. Voir aussi \cite[Lemma 1]{B-B2}.
\begin{lemma}\label{le:countinglatticepointsinsmalldomains}
	Soit $\Lambda\subset \ZZ^2$ un réseau de rang $2$. Soit $E\subset\RR^2$ une partie convexe dont le bord est lisse par morceaux avec $(0,0)\in \overline{E}$. Alors
	$$\sharp \{(x_1,x_2)\in\Lambda\cap E:\pgcd(x_1,x_2)=1\}\ll 1+\frac{\operatorname{vol}(E)}{\det(\Lambda)}.$$
\end{lemma}
\begin{proof}[Démonstration de la Proposition \ref{po:zoompointsoncab}]
La condition \eqref{eq:heightoncab} correspond à la restriction de la hauteur $H_{\omega_{Y_4}^{-1}} $ sur $\cab$, et une inversion de Möbius conduit à une condition de réseau. 
Dans un premier temps on voudrait réduire la condition \eqref{eq:zoomoncabdistance} en un zoom avec une distance induite sur l'image inverse de $\PP^1\to\cab$ pour que l'on puisse appliquer le Théorème \ref{th:zoomingonlattices}. On va approcher le cardinal de l'ensemble $E(a,b,\varepsilon,B,r)$ par celui de l'ensemble $F(a,b,\varepsilon,B,r)$ des $(u,v)\in\NN^{*2}_{\operatorname{prem}}$ satisfaisant à (rappelons les notations \eqref{eq:D1D2})
\begin{equation}\label{eq:zoomdistanceF}
\sqrt{\frac{b}{a}}<\frac{u}{v}<\frac{b}{a},\quad \frac{u}{v}-\sqrt{\frac{b}{a}} \leqslant \frac{\varepsilon}{2}\left(\sqrt{\frac{b}{a}}-1\right) B^{-\frac{1}{r}},
\end{equation}
\begin{equation}\label{eq:heightF}
v^4\leqslant \frac{B D_1^2D_2^2d_3^2}{b^2a\left(\sqrt{ba^{-1}}-1\right)^2}.
\end{equation}
On compare maintenant ces deux ensembles.
En combinant \eqref{eq:distinduced} avec la majoration de $v$ \eqref{eq:heightboundforv} obtenue à partir de la condition \eqref{eq:heightoncab}, on conclut que
$$E(a,b,\varepsilon,B,r)\subset F(a,b,\varepsilon,B,r).$$
De plus en rappelant l'hypothèse \eqref{eq:conditionforbandr1} sur $(a,b)$, pour $(u,v)\in F(a,b,\varepsilon,B,r)$, on a d'après \eqref{eq:zoomdistanceF}, 
$$\frac{u}{v}-\sqrt{\frac{b}{a}}\leqslant \frac{\varepsilon(\tau_1-1)}{2}B^{-\frac{1}{r}}.$$
Donc on a
\begin{align*}
	\frac{ba^{-1}-uv^{-1}}{uv^{-1}+\sqrt{ba^{-1}}}\varepsilon B^{-\frac{1}{r}}&=\frac{\varepsilon}{2}\left(\sqrt{\frac{b}{a}}-1\right)\varepsilon B^{-\frac{1}{r}}-\left(\frac{uv^{-1}-\sqrt{ba^{-1}}}{uv^{-1}+\sqrt{ba^{-1}}}+\frac{ba^{-1}-\sqrt{ba^{-1}}}{2\sqrt{ba^{-1}}}-\frac{ba^{-1}-\sqrt{ba^{-1}}}{uv^{-1}+\sqrt{ba^{-1}}}\right)\varepsilon B^{-\frac{1}{r}}\\
	&>\frac{\varepsilon}{2}\left(\sqrt{\frac{b}{a}}-1\right)\varepsilon B^{-\frac{1}{r}}-\left(\frac{uv^{-1}-\sqrt{ba^{-1}}}{2\tau_2}+\frac{(uv^{-1}-\sqrt{ba^{-1}})(\tau_1^2-\tau_1)}{4\tau_2}\right)\varepsilon B^{-\frac{1}{r}}\\
	&>\frac{\varepsilon}{2}\left(\sqrt{\frac{b}{a}}-1\right)\varepsilon B^{-\frac{1}{r}}-\frac{\varepsilon^2(\tau_1^2-\tau_1+2)(\tau_1-1)}{8\tau_2}B^{-\frac{2}{r}}.
\end{align*}
Et aussi
\begin{align*}
	\frac{B D_1^2D_2^2d_3^2}{b^2a\left(\sqrt{ba^{-1}}-1\right)^2} -\frac{ BD_1^2D_2^2d_3^2}{ba^2\left(\left(\frac{u}{v}\right)^2\left(\frac{u}{v}-1\right)^2\right)}
	=&\frac{B D_1^2D_2^2d_3^2}{ba^2}\left(\frac{1}{\frac{b}{a}\left(\sqrt{\frac{b}{a}}-1\right)^2}-\frac{1}{\left(\frac{u}{v}\right)^2\left(\frac{u}{v}-1\right)^2}\right)\\
	\leqslant & \frac{B D_1^2D_2^2d_3^2}{b^2a}\times \frac{(\tau_1^2 -1)^3 \tau_1+\tau_1^2(\tau_1-1)(\tau_1^2+\tau_1-2)}{2\tau_2^2(\tau_2-1)^4}\varepsilon B^{-\frac{1}{r}}.
\end{align*}
Cela implique que $$\quad F(a,b,\varepsilon,B,r)\setminus E(a,b,\varepsilon,B,r)\subset G(a,b,\varepsilon,B,r),$$ 
où $G(a,b,\varepsilon,B,r)=G_1(a,b,\varepsilon,B,r)\cup G_2(a,b,\varepsilon,B,r)$, avec 
\begin{equation*}
G_1(a,b,\varepsilon,B,r)=\left\{
\begin{aligned}
&(u,v)\in \NN^{*2}_{\text{prem}}\\
&\sqrt{\frac{b}{a}}<\frac{u}{v}<\frac{b}{a}
\end{aligned}
\left|
\begin{aligned}
&\frac{\varepsilon}{2}\left(\sqrt{\frac{b}{a}}-1\right) B^{-\frac{1}{r}}-A(\varepsilon,\tau_1,\tau_2)B^{-\frac{2}{r}}\leqslant\frac{u}{v}-\sqrt{\frac{b}{a}} \leqslant \frac{\varepsilon}{2}\left(\sqrt{\frac{b}{a}}-1\right) B^{-\frac{1}{r}} \\
&v^4\leqslant \frac{B D_1^2D_2^2d_3^2}{b^2a\left(\sqrt{ba^{-1}}-1\right)^2} 
\end{aligned}
\right\}\right. ,
\end{equation*}
\begin{equation*}
G_2(a,b,\varepsilon,B,r)=\left\{
\begin{aligned}
&(u,v)\in \NN^{*2}_{\text{prem}}\\
&\sqrt{\frac{b}{a}}<\frac{u}{v}<\frac{b}{a}
\end{aligned}
\left|
\begin{aligned}
&\frac{u}{v}-\sqrt{\frac{b}{a}} \leqslant \frac{\varepsilon}{2}\left(\sqrt{\frac{b}{a}}-1\right) B^{-\frac{1}{r}} \\
&\frac{B D_1^2D_2^2d_3^2}{b^2a\left(\sqrt{ba^{-1}}-1\right)^2}\left(1-A_2(\varepsilon,\tau_1,\tau_2)B^{-\frac{1}{r}}\right)\leqslant v^4\leqslant \frac{B D_1^2D_2^2d_3^2}{b^2a\left(\sqrt{ba^{-1}}-1\right)^2} 
\end{aligned}
\right\}\right. ,
\end{equation*}
où $$A_1(\varepsilon,\tau_1,\tau_2)=\frac{\varepsilon^2(\tau_1^2-\tau_1+2)(\tau_1-1)}{8\tau_2},$$
$$A_2(\varepsilon,\tau_1,\tau_2)=\frac{\varepsilon(\tau_1^2 -1)^3 \tau_1+\tau_1^2(\tau_1-1)(\tau_1^2+\tau_1-2)}{2\tau_2^2(\tau_2-1)^4}.$$
Ensuite on va borner le cardinal de $G(a,b,\varepsilon,B,r)$. 
Par une inversion de Möbius, on a pour $l=1,2$,
$$\sharp G_l(a,b,\varepsilon,B,r)=\sum_{\substack{e_1,e_2,e_3,f_1,f_2,f_3\in\NN^*\\e_1f_1|b,e_2f_2|a,e_3f_3|b-a}}\left(\prod_{i=1}^{3}\mu(e_i)\right)\sharp G_l(e_1,e_2,e_3,f_1,f_2,f_3,a,b,\varepsilon,B,r)$$
où $G_1(e_1,e_2,e_3,f_1,f_2,f_3,a,b,\varepsilon,B,r),G_2(e_1,e_2,e_3,f_1,f_2,f_3,a,b,\varepsilon,B,r)$ sont respectivement les ensembles
\begin{align*}
\left\{
\begin{aligned}
&(u,v)\in \NN^{*2}_{\text{prem}}\\
&\sqrt{ba^{-1}}<uv^{-1}<ba^{-1}\\
&e_1f_1|u^2,e_2f_2|v^2,e_3f_3|u-v
\end{aligned}
\left|
\begin{aligned}
&\frac{\varepsilon}{2}\left(\sqrt{\frac{b}{a}}-1\right) B^{-\frac{1}{r}}-A_1(\varepsilon,\tau_1,\tau_2)B^{-\frac{2}{r}}\leqslant\frac{u}{v}-\sqrt{\frac{b}{a}} \leqslant \frac{\varepsilon}{2}\left(\sqrt{\frac{b}{a}}-1\right) B^{-\frac{1}{r}} \\
&v^4\leqslant \frac{B f_1^2f_2^2f_3^2}{b^2a\left(\sqrt{ba^{-1}}-1\right)^2} 
\end{aligned}
\right\}\right. ,
\end{align*}
\begin{align*}
\left\{
\begin{aligned}
&(u,v)\in \NN^{*2}_{\text{prem}}\\
&\sqrt{ba^{-1}}<uv^{-1}<ba^{-1}\\
&e_1f_1|u^2,e_2f_2|v^2,e_3f_3|u-v
\end{aligned}
\left|
\begin{aligned}
&\frac{u}{v}-\sqrt{\frac{b}{a}} \leqslant \frac{\varepsilon}{2}\left(\sqrt{\frac{b}{a}}-1\right) B^{-\frac{1}{r}} \\
&\frac{B f_1^2f_2^2f_3^2}{b^2a\left(\sqrt{ba^{-1}}-1\right)^2}\left(1-A_2(\varepsilon,\tau_1,\tau_2)B^{-\frac{1}{r}}\right) \leqslant v^4\leqslant \frac{B f_1^2f_2^2f_3^2}{b^2a\left(\sqrt{ba^{-1}}-1\right)^2} 
\end{aligned}
\right\}\right. .
\end{align*}
Le points dans le premier ensemble se trouvent dans l'intersection d'un triangle, dont l'origine est l'un de ses sommets et l'aire est de grandeur $O_{\tau_i,\varepsilon}\left(B^{\frac{1}{4}}(f_1f_2f_3)^\frac{1}{2}b^{-\frac{3}{4}}\times B^{\frac{1}{4}-\frac{2}{r}}(f_1f_2f_3)^\frac{1}{2}b^{-\frac{3}{4}} \right)=O_{\tau_i,\varepsilon}\left(B^{\frac{1}{2}-\frac{2}{r}}f_1f_2f_3b^{-\frac{3}{2}}\right)$, avec le réseau
\begin{equation}\label{eq:thelattice}
\begin{split}
\Lambda_{\substack{e_1,e_2,e_3\\f_1,f_2,f_3}}&=\{(x,y)\in\ZZ^2:e_1f_1|x^2,e_2f_2|y^2,e_3f_3|y-x\}\\
&=\{(x,y)\in\ZZ^2:g(e_1f_1)|x,g(e_2f_2)|y,e_3f_3|y-x\}.
\end{split}
\end{equation}
Puisque $\pgcd(e_if_i,e_jf_j)=1$ pour $i\neq j$, on vérifie que (cf. par exemple \cite[Proposition 4.3]{Huang1})
$$\det\left(\Lambda_{\substack{e_1,e_2,e_3\\f_1,f_2,f_3}}\right)=g(e_1f_1)g(e_2f_2)e_3f_3\leqslant e_1e_2e_3f_1f_2f_3.$$
Quant à l'ensemble $G_2(e_1,e_2,e_3,f_1,f_2,f_3,a,b,\varepsilon,B,r)$, ses points sont contenu dans un trapèze dont la longueur du bord est $O_{\tau_i,\varepsilon}\left(B^{\frac{1}{4}-\frac{1}{r}}(f_1f_2f_3)^\frac{1}{2}b^{-\frac{3}{4}}\right)$ et l'aire est 
$O_{\tau_i,\varepsilon}(B^{\frac{1}{2}-\frac{1}{r}}f_1f_2f_3b^{-\frac{3}{2}})$.
En utilisant le Lemme \ref{le:countinglatticepointsinsmalldomains}, on obtient que pour tout $0<\delta_1<1$, comme $\tau(n)\ll  _{\delta_1} n^{\delta_1}$ (cf. \cite[\S I.5.2]{tenenbaum}), $\sigma_\kappa(n)=\sum_{d\mid n}d^\kappa \ll_{\delta_1} n^{\kappa+\delta_1}$ (cf. \cite[\S I.5.5]{tenenbaum}),
\begin{align*}
	 \sharp G(a,b,\varepsilon,B,r)&\ll_{\tau_i,\varepsilon}\sum_{e_1f_1|b,e_2f_2|a,e_3f_3|b-a}\left(1+\frac{B^{\frac{1}{2}-\frac{2}{r}}}{b^{\frac{3}{2}}e_1e_2e_3}+\frac{B^{\frac{1}{4}-\frac{1}{r}}(f_1f_2f_3)^\frac{1}{2}}{b^{\frac{3}{4}}}+\frac{B^{\frac{1}{2}-\frac{1}{r}}f_1f_2f_3}{b^{\frac{3}{2}}}\right)\\
	 &\ll_{\tau_i} \sum_{n_1|b,n_2|a,n_3|b-a} \tau(n_1)\tau(n_2)\tau(n_3)\\
	 &\ll  _{\delta_1,\tau_i}\sum_{n_1|b,n_2|a,n_3|b-a} (n_1n_2n_3)^{\delta_1}\\
	 &\ll_{\tau_i,\delta_1} \sigma_{\delta_1}(ba(b-a))\ll _{\tau_i,\delta_i}b^{\delta_2},
 \end{align*}
 où $\delta_1,\delta_2$ peuvent être arbitrairement petits. On en conclut que
 $$\sharp  E(a,b,\varepsilon,B,r)=\sharp F(a,b,\varepsilon,B,r)+O_{\tau_i,\delta}(b^\delta).$$
Maintenant on compte le cardinal de $F(a,b,\varepsilon,B,r)$. Comme on a fait pour $G(a,b,\varepsilon,B,r)$, et par une inversion de Möbius on obtient
\begin{equation}\label{eq:Fepsilonabdecomposition}
	\begin{split}
	\sharp F(a,b,\varepsilon,B,r)&=\sum_{e_1f_1|b,e_2f_2|a,e_3f_3|b-a}\left(\prod_{i=1}^3\mu(e_i)\right)\sharp F(e_1,e_2,e_3,f_1,f_2,f_3,a,b,\varepsilon,B,r),
	\end{split}
\end{equation} où
\begin{equation}\label{eq:thesetF}
 F(e_1,e_2,e_3,f_1,f_2,f_3,a,b,\varepsilon,B,r)=\left\{
 \begin{aligned}
&(u,v)\in \NN^{*2}_{\text{prem}}\\
&e_1f_1|u^2,e_2f_2|v^2,e_3f_3|u-v
 \end{aligned}
\left|
\begin{aligned}
&0<\frac{u}{v}-\sqrt{\frac{b}{a}} \leqslant \frac{\varepsilon}{2}\left(\sqrt{\frac{b}{a}}-1\right) B^{-\frac{1}{r}}\\
&v^4\leqslant \frac{B f_1^2f_2^2f_3^2}{b^2a\left(\sqrt{ba^{-1}}-1\right)^2}
\end{aligned}
\right\}\right. .
\end{equation}

Nous allons appliquer le Théorème \ref{th:zoomingonlattices}à \eqref{eq:thesetF}. À cette fin on prend dans \eqref{eq:thesetF} (rappelons le réseau \eqref{eq:thelattice})
$$\Lambda^\prime=\Lambda_{\substack{e_1,e_2,e_3\\f_1,f_2,f_3}},\quad K^\prime=\frac{(f_1f_2f_3)^{\frac{1}{2}}}{b^\frac{1}{2}a^{\frac{1}{4}}\left(\sqrt{ba^{-1}}-1\right)^{\frac{1}{2}}},\quad \varepsilon^\prime=\frac{\varepsilon}{2}\left(\sqrt{\frac{b}{a}}-1\right).$$
Notons qu'ici la signification de $B$ et $r$ sont différentes. On prend
$$B^\prime=B^{\frac{1}{4}},\quad r^\prime=\frac{r}{4}.$$
Enfin d'appliquer le Théorème \ref{th:zoomingonlattices}, nous vérifions que les conditions \eqref{eq:conditionforeta} et \eqref{eq:conditionforbandr1} sur $r$ et $b$ impliquent \eqref{eq:cond0} - \eqref{eq:cond5}.
En effet, la condition \eqref{eq:cond5} se traduit en
	$$bg(e_1f_1)^2 g(e_2f_2)^2 e_3^2f_3^2\leqslant \frac{\prod_{i=1}^3 f_i}{b a^\frac{1}{2}(\sqrt{ba^{-1}}-1)}B^{\frac{1}{2}-\frac{1}{r}},$$
	qui est elle-même impliquée par les conditions équivalentes
	$$b\prod_{i=1}^{3}e_i^2f_i^2\leqslant \frac{\prod_{i=1}^3 f_i}{b a^\frac{1}{2}(\sqrt{ba^{-1}}-1)}B^{\frac{1}{2}-\frac{1}{r}}\Leftrightarrow b^{\frac{5}{2}}\sqrt{ab^{-1}}(\sqrt{ba^{-1}}-1) \prod_{i=1}^{3}e_i^2 f_i\leqslant B^{\frac{1}{2}-\frac{1}{r}}.$$
	En utilisant la majoration $$\prod_{i=1}^{3} e_i\leqslant \prod_{i=1}^3 e_i f_i\leqslant ba(b-a),$$
	la condition \eqref{eq:conditionforbandr1} sur $b$ entraîne que (on rappelle que $0<a<b$ et donc $\frac{a}{b}\left(1-\frac{a}{b}\right)\leqslant \left(\frac{1}{2}\left(\frac{a}{b}+1-\frac{a}{b}\right)\right)^2=\frac{1}{4}$)
	\begin{align*}
		b^{\frac{5}{2}}\sqrt{ab^{-1}}(\sqrt{ba^{-1}}-1) \prod_{i=1}^{3}e_i^2 f_i&\leqslant b^{\frac{5}{2}}\frac{\tau_1-1}{\tau_2}(ba(b-a))^2\\
		&=b^{\frac{17}{2}}\frac{\tau_1-1}{\tau_2}\left(\frac{a}{b}\left(1-\frac{a}{b}\right)\right)^2\\
		&\leqslant b^{\frac{17}{2}}\frac{\tau_1-1}{16\tau_2}\leqslant B^{\frac{1}{2}-\frac{1}{r}},
	\end{align*}
	 d'où la condition \eqref{eq:cond5}.
	De plus, \eqref{eq:conditionforeta} sur $r$ entraîne que
	$\frac{1}{2}<r^\prime<\frac{36}{55}<\frac{7}{10}$ et pour un tel $r$ on a
$$\frac{2}{17}\left(\frac{1}{2}-\frac{1}{r}\right)<\frac{2}{5}\left(\frac{4}{5}(\frac{1}{r}-\frac{1}{4})-\frac{3}{5}(\frac{1}{2}-\frac{1}{r})\right).$$
Donc \eqref{eq:conditionforbandr1} implique la condition suivante vérifiées de la même façon, qui elle-même implique \eqref{eq:cond4} pour tout $B\gg_{\tau_i,\varepsilon} 1$,
$$ b^\frac{5}{2}\leqslant \tau_1^2U(\tau_2,\varepsilon(\tau_2-1)/2) B^{\frac{4}{5}(\frac{1}{r}-\frac{1}{4})-\frac{3}{5}(\frac{1}{2}-\frac{1}{r})}.$$

Les conditions \eqref{eq:cond0} - \eqref{eq:cond5} étant vérifiées, on peut appliquer le Théorème \ref{th:zoomingonlattices} avec les quantités $\Lambda_{\substack{e_1,e_2,e_3\\f_1,f_2,f_3}},K^\prime,\varepsilon^\prime,B^\prime,r^\prime$ ci-dessus, et on obtient, compte-tenu du calcul \eqref{eq:Gammadet},
	\begin{equation}
		\begin{split}
		&\sharp F(e_i,f_i,a,b,\varepsilon,B,r)\\
	&=\frac{\Theta\left(\Lambda_{\substack{e_1,e_2,e_3\\f_1,f_2,f_3}}\right)f_1f_2f_3}{4ba^{\frac{1}{2}}}\varepsilon B^{\frac{1}{2}-\frac{1}{r}}+O_{\tau_i,\varepsilon}\left((K^\prime)^\frac{3}{2}b^\frac{1}{4} \det\left(\Lambda_{\substack{e_1,e_2,e_3\\f_1,f_2,f_3}}\right)^{\frac{1}{2}}B^{\frac{3}{4}(\frac{1}{2}-\frac{1}{r})}\log B +K^\prime b^\frac{3}{2}B^{\frac{1}{4}-\frac{1}{2r}}\log B\right)\\
	&=\frac{3}{2\pi^2}\frac{\Psi_1\left(\det\left(\Lambda_{\substack{e_1,e_2,e_3\\f_1,f_2,f_3}}\right)\right)f_1f_2f_3}{ba^{\frac{1}{2}}\det\left(\Lambda_{\substack{e_1,e_2,e_3\\f_1,f_2,f_3}}\right)}\varepsilon B^{\frac{1}{2}-\frac{1}{r}}+O_{\tau_i,\varepsilon}\left( \frac{(e_1e_2e_3)^{\frac{1}{2}}(f_1f_2f_3)^\frac{5}{4}}{b^{\frac{7}{8}}} B^{\frac{3}{4}(\frac{1}{2}-\frac{1}{r})}\log B+ b^\frac{3}{4}(f_1f_2f_3)^\frac{1}{2}B^{\frac{1}{4}-\frac{1}{2r}}\log B\right).
		\end{split}
	\end{equation}
	En reportant dans \eqref{eq:Fepsilonabdecomposition}, on calcule la somme du coefficient du terme principal
	\begin{align*}
	&\sum_{e_1f_1|b,e_2f_2|a,e_3f_3|b-a}(\prod_{i=1}^3\mu(e_i))\frac{\Psi_1\left(\det\left(\Lambda_{\substack{e_1,e_2,e_3\\f_1,f_2,f_3}}\right)\right)f_1f_2f_3}{\det\left(\Lambda_{\substack{e_1,e_2,e_3\\f_1,f_2,f_3}}\right)}\\
	&=\left(\sum_{e_1f_1|b}\frac{\Psi_1(g(e_1f_1))e_1f_1}{g(e_1f_1)}\frac{\mu(e_1)}{e_1}\right)\left(\sum_{e_2f_2|a}\frac{\Psi_1(g(e_2f_2))e_2f_2}{g(e_2f_2)}\frac{\mu(e_2)}{e_2}\right)\left(\sum_{e_3f_3|b-a}\Psi_1(e_3f_3)\frac{\mu(e_3)}{e_3}\right)\\
	&=\Phi(b)\Phi(a)\Psi(b-a).
	\end{align*}
	Rappelons la majoration pour la fonction $\sigma_\kappa$ \cite[Theorem 5.7]{tenenbaum},
	$$\sigma_\kappa(n)\ll n^\kappa,(\kappa>1),\quad \sigma_{\kappa}(n)\ll_\delta n^{\kappa+\delta},(0<\kappa<1,\forall\delta>0).$$ 
	On peut majorer le terme d'erreur de la façon suivante. Pour tout $\delta>0$, 
	\begin{align*}
		\sum_{e_1f_1|b,e_2f_2|a,e_3f_3|b-a}\frac{(e_1e_2e_3)^{\frac{1}{2}}(f_1f_2f_3)^\frac{5}{4}}{b^{\frac{7}{8}}} &=b^{-\frac{7}{8}}\sum_{n_1|b,n_2|a,n_3|b-a}(n_1n_2n_3)^{\frac{1}{2}}\sigma_{\frac{3}{4}}(n_1)\sigma_{\frac{3}{4}}(n_2)\sigma_{\frac{3}{4}}(n_3)\\
		&\ll_\delta b^{-\frac{7}{8}}\sigma_{\frac{5}{4}+\delta}(b)\sigma_{\frac{5}{4}+\delta}(a)\sigma_{\frac{5}{4}+\delta}(b-a)\\
	    &\ll_\delta b^{\frac{15}{4}-\frac{7}{8}+\delta}=b^{\frac{23}{8}+\delta}.
	\end{align*}
	Et pour tout $\frac{5}{8}\geqslant\delta_2>3\delta_1>0$, 
	\begin{align*}
			\sum_{e_1f_1|b,e_2f_2|a,e_3f_3|b-a} b^\frac{3}{4}(f_1f_2f_3)^\frac{1}{2}
			&=b^\frac{3}{4} \sum_{n_1|b,n_2|a,n_3|b-a} \sigma_{\frac{1}{2}}(n_1)\sigma_{\frac{1}{2}}(n_2)\sigma_{\frac{1}{2}}(n_3)\\
			&\ll_{\delta_1} b^\frac{3}{4} \sum_{n_1|b,n_2|a,n_3|b-a}  (n_1n_2n_3)^{\frac{1}{2}+\delta_1}\\
			&\ll_{\delta_1} b^\frac{3}{4}\sigma_{\frac{1}{2}+\delta_1}(ba(b-a))\\
			&\ll_{\delta_2} b^{\frac{9}{4}+\delta_2}.
	\end{align*}
Compte tenu le terme d'erreur provenant du cardinal de $G(a,b,\varepsilon,B,r)$, on en conclut la formule asymptotique \eqref{eq:asymptoticoncab} sur chaque courbe $\cab$.
\end{proof}
	\subsubsubsection{Obtention de la borne inférieure}
	Comme étant une étape de clé, avant tout on a besoin du lemme suivant, qui nous conduira au problème de diviseurs des formes en deux variables.
	\begin{lemma}\label{le:upperandlowerboundforphi}
		Rappelons les fonctions arithmétiques $\tau$ \eqref{eq:thefunctiontau}, $\Psi$ \eqref{eq:arithmeticfunctionkey} et $\Phi$ \eqref{eq:arithmeticfunctionPhi}. On a pour tout $n\in\NN_{\geqslant 1}$,
		\begin{equation*}
		\Psi(n)\leqslant \tau(n) \quad \text{et}\quad\Psi(n)\leqslant \Phi(n)\leqslant \sqrt{n}\Psi(n).
		\end{equation*}
	\end{lemma}
	\begin{proof}
		Puisque $\tau,\Psi,\Phi$ sont positives et multiplicatives, il suffit de comparer leur valeurs en les puissances des nombre premiers. Fixons un nombre premier $p$ et $k\in\NN$, on a
		$$\Psi(p^k)=1+k\frac{1-p^{-1}}{1+p^{-1}}\leqslant 1+k=\tau(p^k).$$
		Quant à $\Phi$, on a
		\begin{equation*}
		\begin{split}
		\Phi(p^k)=\sum_{l=0}^{k}\Psi_1(p^{\lceil \frac{l}{2}\rceil})p^{l-\lceil \frac{l}{2}\rceil}\phi(p^l)
		=1+\frac{1-p^{-1}}{1+p^{-1}}\sum_{l=1}^k p^{l-\lceil \frac{l}{2}\rceil}
		\geqslant 1+k\frac{1-p^{-1}}{1+p^{-1}}=\Psi(p^k),
		\end{split}
		\end{equation*}
		alors que
		\begin{equation*}
		\Phi(p^k)=1+\frac{1-p^{-1}}{1+p^{-1}}\sum_{l=1}^k p^{l-\lceil \frac{l}{2}\rceil}
		\leqslant 1+\frac{1-p^{-1}}{1+p^{-1}}\sum_{l=1}^k p^{\frac{l}{2}}
		\leqslant p^{\frac{k}{2}}\left(1+k\frac{1-p^{-1}}{1+p^{-1}}\right)
		=p^{\frac{k}{2}}\Psi(p^k).\qedhere
		\end{equation*}
	\end{proof}
	On voit ailleurs que la différence entre $\Psi$ et $\tau$ est \og petite\fg.
	Le terme principal \eqref{eq:asymptoticoncab} sur chaque courbe nodale admet donc une minoration de la forme
	\begin{equation}\label{eq:keylowerbound}
	\frac{\Phi(b)\Phi(a)\Psi(b-a)}{ba^{\frac{1}{2}}}\geqslant \frac{\Psi(b)\Psi(a)\Psi(b-a)}{ba^{\frac{1}{2}}},
	\end{equation}
	qui fait disparaître la fonction $\Phi$ et laisser la fonction $\Psi$ qui ressemble à la fonction $\tau$ au sens de la convolution, dont l'ordre moyen de ce type est connu grâce à une série de travaux de R. de la Bretèche et T. D. Browning (\cite{B-B2}, \cite{B-B1}, \cite{Browning}). On donnera les détails dans les appendices, où nous décrivons ces résultats. 
	
	\begin{proof}[Démonstration du Théorème \ref{th:lowerbound}]
	Tout d'abord rappelons le difféomorphisme local $\rho$ \eqref{eq:diffeomorphism} et les notations de coordonnées de l'espace tangent
	$$(w^\prime,z^\prime)=(w-1,z-1).$$
	On déduit un encadrement des $(a,b)$ tels que la courbe $\cab$  intervienne dans le dénombrement \eqref{eq:thesetEepsilonB} quand on prend la fonction caractéristique $\chi(W(\varepsilon_1,\varepsilon_2,\tau_1,\tau_2))$ de la région $W=W(\varepsilon_1,\varepsilon_2,\tau_1,\tau_2)$ \eqref{eq:regionW} dans $R$ \eqref{eq:theregionR}. 
	La condition de zoom dit que
\begin{equation}\label{eq:zoomconditiononzw}
	\max(w^\prime,z^\prime)=z^\prime=z-1\leqslant \varepsilon_1 B^{-\frac{1}{r}}.
\end{equation}
	Prenons un couple $(a,b)$ vérifiant la condition \eqref{eq:conditiononab}, 
	d'après l'équation \eqref{eq:cabR} définissant l'image de la courbe $\cab$, on a 
	$$\frac{b}{a}=\frac{(z^\prime)^2(w^\prime+1)}{(w^\prime)^2(z^\prime+1)}.$$
	On note
	$$\delta(\varepsilon,B)=1+\varepsilon B^{-\frac{1}{r}}.$$
	S'il existe $(w^\prime,z^\prime)\in B^{-\frac{1}{r}}W\cap \rho(\cab)$, alors d'après \eqref{eq:conditionforbandr1} on a
	$$\frac{\tau_2^2}{\delta(\varepsilon_1,B)}=\frac{\tau_2^2}{1+\varepsilon_1 B^{-\frac{1}{r}}}<\frac{b}{a}< \tau_1^2(1+\varepsilon_1 B^{-\frac{1}{r}})=\tau_1^2\delta(\varepsilon_1,B),$$
	Maintenant prenons un couple $(a,b)$ vérifiant \eqref{eq:conditiononab} et 
	$$\tau_2^2\delta(\varepsilon_1,B)<\frac{b}{a}<\frac{\tau_1^2}{\delta(\varepsilon_1,B)}.$$
	Puisqu'un point $(w^\prime,z^\prime)\in\rho( \cab)$ vérifie 
	$$\frac{(z^\prime)^2}{(w^\prime)^2}=\frac{b(z^\prime+1)}{a(w^\prime+1)},$$
	la condition de zoom implique que sur la courbe $\cab$, si $\max(w^\prime,z^\prime)\leqslant \varepsilon_1 B^{-\frac{1}{r}}$, on a
	$$\tau_2^2=\frac{\tau_2^2\delta(\varepsilon_1,B)}{1+\varepsilon_1 B^{-\frac{1}{r}}}<\frac{(z^\prime)^2}{(w^\prime)^2}<\frac{\tau_1^2(1+\varepsilon_1 B^{-\frac{1}{r}})}{\delta(\varepsilon_1,B)}=\tau_1^2.$$
Donc $(w^\prime,z^\prime)$ donne un point de $B^{-\frac{1}{r}}W\cap\rho(\cab)$.

On restreint le dénombrement sur les $(a,b)$ tels que $(a,b)\not\in\square^2$ car c'est sur de telles $\cab$ qu'on trouve une équidistribution de points pour le zoom sous-critique (Proposition \ref{po:zoompointsoncab}).
On en conclut que 
\begin{equation}\label{eq:deuxparts}
	\begin{split}
	\delta_{U,Q,B,r}(\chi(W\varepsilon_1,\varepsilon_2,\tau_1,\tau_2))
	\geqslant \sum^*_{\substack{ \tau_2^2\leqslant\frac{b}{a}\leqslant\tau_1^2}}\sharp E(\varepsilon_1,\varepsilon_2,a,b,B,r)+O\left(\sum^*_{\substack{ \frac{\tau_2^2}{\delta(\varepsilon_1,B)}<\frac{b}{a}<\tau_2^2\delta(\varepsilon_1,B) \text{ ou }\\ \frac{\tau_1^2}{\delta(\varepsilon_1,B)}<\frac{b}{a}<\tau_1^2\delta(\varepsilon_1,B) }}\sharp E(\varepsilon_1,\varepsilon_2,a,b,B,r)\right),
	\end{split}
\end{equation}
où le symbole $*$ signifie que la somme est prise sur les $(a,b)\in\NN_{\geqslant 1}^2$ satisfaisant à	\eqref{eq:conditiononab}, et les ensembles
$E(a,b,\varepsilon_1,\varepsilon_2,B,r)$ sont définis de manière analogue à $E(a,b,\varepsilon,B,r)$ en remplaçant \eqref{eq:zoomoncabdistance} par $\varepsilon_2<B^{\frac{1}{r}}d(P)\leqslant \varepsilon_1$.

		Pour obtenir un terme principal satisfaisant, on restreint la somme sur les courbes $\cab$ des paramètres $(a,b)$ vérifiant $$b\leqslant B^{\eta\left(1-\frac{2}{r}\right)}=o(B^{\frac{2}{17}(\frac{1}{2}-\frac{1}{r})}),$$
		qui correspond à \eqref{eq:conditionforbandr1}. Tout d'abord d'après la Proposition \ref{po:zoompointsoncab}, on obtient, pour chaque tel $(a,b)$ fixé, la minoration
		\begin{equation}\label{eq:lowerboundforstar}
		\begin{split}
			\sharp E(a,b,\varepsilon_1,\varepsilon_2,B,r)&=\frac{3}{2\pi^2}\frac{\Phi(b)\Phi(a)\Psi(b-a)}{ba^{\frac{1}{2}}}(\varepsilon_1-\varepsilon_2)B^{\frac{1}{2}-\frac{1}{r}}+O_{\tau_i,\varepsilon_i,\delta}\left(b^{\frac{23}{8}+\delta}B^{\frac{3}{4}(\frac{1}{2}-\frac{1}{r})}\log B \right)\\
			&\geqslant \frac{3}{2\pi^2}\frac{\Psi(b)\Psi(a)\Psi(b-a)}{ba^{\frac{1}{2}}}(\varepsilon_1-\varepsilon_2)B^{\frac{1}{2}-\frac{1}{r}}+O_{\tau_i,\varepsilon_i,\delta}\left(b^{\frac{23}{8}+\delta}B^{\frac{3}{4}(\frac{1}{2}-\frac{1}{r})}\log B \right),
			\end{split}
		\end{equation}
		en utilisant le Lemme \ref{le:upperandlowerboundforphi}.
		 Afin d'appliquer la Proposition \ref{po:asymptoticforPsiquotient} sur l'ordre moyen des diviseurs de formes binaires, on étend la somme \eqref{eq:deuxparts} sur les $(a,b)$ vérifiant la condition $\eqref{eq:conditiononab}$ en rajoutant les $(a,b)$ vérifiant la condition $(a,b)\in\square^2$. Le terme d'erreur correspondant est de grandeur
		\begin{align*}
		\sum_{\substack{1\leqslant\max(a,b)\leqslant B^{\eta\left(1-\frac{2}{r}\right)}\\ a<b,(a,b)\in\square^2}}\frac{\Psi(b)\Psi(a)\Psi(b-a)}{ba^{\frac{1}{2}}}B^{\frac{1}{2}-\frac{1}{r}}&\ll \sum_{\substack{1\leqslant\max(c,d)\leqslant B^{\eta\left(\frac{1}{2}-\frac{1}{r}\right)}\\c>d}}\frac{\Psi(c^2)\Psi(d^2)\Psi(c^2-d^2)}{c^2d}B^{\frac{1}{2}-\frac{1}{r}}\\
		&\ll \sum_{\substack{1\leqslant\max(c,d)\leqslant B^{\eta\left(\frac{1}{2}-\frac{1}{r}\right)}\\c>d}}\frac{\tau(c^2)\tau(d^2)\tau(c^2-d^2)}{c^2d}B^{\frac{1}{2}-\frac{1}{r}}\\ &\ll _{\delta_1}  \sum_{\substack{1\leqslant\max(c,d)\leqslant B^{\eta\left(\frac{1}{2}-\frac{1}{r}\right)} }}\frac{1}{c^{2-\delta_1}d}B^{\frac{1}{2}-\frac{1}{r}}\\
		&\ll_{\delta_1} B^{\frac{1}{2}-\frac{1}{r}}\log B,
		\end{align*}
		pour tout $\delta_1\in\mathopen]0,1\mathclose[$ puisque $\tau(n)\ll n^\delta,\forall \delta>0$ (cf. \cite[\S I.5.2]{tenenbaum}).
	D'après la formule \ref{po:asymptoticforPsiquotient} de la Proposition \ref{po:problemofdivisor}, la somme des termes principaux dans \eqref{eq:lowerboundforstar} est
	 	\begin{align*}
	 &\sum_{\substack{a<b\leqslant B^{\eta(1-\frac{2}{r})}\\\tau_2^2\leqslant\frac{b}{a}\leqslant\tau_1^2\\ \pgcd(a,b)=1}}\frac{3}{2\pi^2}\frac{\Psi(b)\Psi(a)\Psi(b-a)}{ba^{\frac{1}{2}}}(\varepsilon_1-\varepsilon_2)B^{\frac{1}{2}-\frac{1}{r}}\\
	 &=\frac{6}{\pi^2}C_1 \left(\eta(1-\frac{2}{r})\right)^3 \left(\frac{1}{\tau_2}-\frac{1}{\tau_1}\right)(\varepsilon_1-\varepsilon_2) B^{(1+\eta)\left(\frac{1}{2}-\frac{1}{r}\right)}(\log B)^3 +O( B^{(1+\eta)\left(\frac{1}{2}-\frac{1}{r}\right)}(\log B)^2 )\\
	 &=\left(C_2 \int_{\substack{z^\prime\in\mathopen]\varepsilon_2,\varepsilon_1\mathclose]\\\theta\in\mathopen]\tau_2,\tau_1\mathclose[}}\frac{1}{\theta^{2}} \operatorname{d}\theta\operatorname{d}z^\prime\right)B^{(1+\eta)\left(\frac{1}{2}-\frac{1}{r}\right)}(\log B)^3 +O( B^{(1+\eta)\left(\frac{1}{2}-\frac{1}{r}\right)}(\log B)^2 )\\
	 &=\left(C_2\int \chi(W(\varepsilon_1,\varepsilon_2,\tau_1,\tau_2))\frac{\operatorname{d }w^\prime\operatorname{d}z^\prime}{z^\prime}\right)B^{(1+\eta)\left(\frac{1}{2}-\frac{1}{r}\right)}(\log B)^3+O( B^{(1+\eta)\left(\frac{1}{2}-\frac{1}{r}\right)}(\log B)^2 ),
	 \end{align*}
	 où $C_1$ est définie par \eqref{eq:C1} et $C_2$ est \eqref{eq:C2}.
	 La contribution de la somme des termes d'erreur dans \eqref{eq:lowerboundforstar} est majorée de la façon suivante.
	 \begin{align*}
	 	\sum_{a<b\leqslant B^{\eta(1-\frac{2}{r})}} b^{\frac{23}{8}+\delta}B^{\frac{3}{4}(\frac{1}{2}-\frac{1}{r})}\log B& 
	 	\ll \sum_{b\leqslant B^{\eta(1-\frac{2}{r})}} b^{\frac{23}{8}+1+\delta}B^{\frac{3}{4}(\frac{1}{2}-\frac{1}{r})}\log B\\
	 	&\ll B^{(\frac{39}{4}\eta+\frac{3}{4}+2\eta\delta)(\frac{1}{2}-\frac{1}{r})}\log B= O(B^{(1+\eta)(\frac{1}{2}-\frac{1}{r})}\log B),
	 \end{align*}
	 grâce à la condition \eqref{eq:conditionforeta} sur $r$ et pour $$0<8\delta\leqslant \frac{1}{\eta}-35.$$
Pour contrôler la somme du terme d'erreur dans \eqref{eq:deuxparts}, on utilise la majoration pour la fonction $\Phi$ donné par le Lemme \ref{le:upperandlowerboundforphi}. Le terme principal \eqref{eq:asymptoticoncab} admet donc la majoration de la forme
$$	\frac{\Phi(b)\Phi(a)\Psi(b-a)}{ba^{\frac{1}{2}}}\leqslant \frac{\Psi(b)\Psi(a)\Psi(b-a)}{b^{\frac{1}{2}}}\leqslant \tau(b)\tau(a)\tau(b-a).$$
Les $(a,b)\in\NN_{\geqslant 1}$ vérifiant la condition
\begin{equation}\label{eq:errortermconditionab}
\frac{\tau_2^2}{\delta(\varepsilon_1,B)}<\frac{b}{a}<\tau_2^2\delta(\varepsilon_1,B) \quad \text{ou}\quad \frac{\tau_1^2}{\delta(\varepsilon_1,B)}<\frac{b}{a}<\tau_1^2\delta(\varepsilon_1,B),  \quad \text{et} \quad a<b\leqslant B^{\eta(1-\frac{2}{r})}
\end{equation}
se trouvent dans deux triangles de l'aire 
$$B^{2\eta(1-\frac{2}{r})}O_{\tau_i}\left(\delta(\varepsilon_{1},B)-\frac{1}{\delta(\varepsilon_{1},B)}\right)=O_{\tau_i}(B^{4\eta(\frac{1}{2}-\frac{1}{r})-\frac{1}{r}}).$$
Rappelons le Lemme \ref{le:countinglatticepointsinsmalldomains} et la majoration pour la fonction $\tau$ \cite[\S I.5.2]{tenenbaum}, on obtient que, pour tout $\delta>0$,
\begin{align*}
	&\sum_{\substack{(a,b)\text{ vérifie} \eqref{eq:errortermconditionab},\\\pgcd(a,b)=1}}\sharp E(\varepsilon_1,\varepsilon_2,a,b,B)\\
		&\ll_{\varepsilon_i,\tau_i}  \sum_{\substack{(a,b)\text{ vérifie} \eqref{eq:errortermconditionab},\\\pgcd(a,b)=1}}\tau(b)\tau(a)\tau(b-a)B^{\frac{1}{2}-\frac{1}{r}}+ 	\sum_{a<b\leqslant B^{\eta(1-\frac{2}{r})}} b^{\frac{23}{8}+\delta}B^{\frac{3}{4}(\frac{1}{2}-\frac{1}{r})}\log B\\		
		&\ll_{\varepsilon_i,\tau_i,\delta} B^{\frac{1}{2}-\frac{1}{r}+\delta}(B^{4\eta(\frac{1}{2}-\frac{1}{r})-\frac{1}{r}}+1)+B^{(1+\eta)(\frac{1}{2}-\frac{1}{r})}\log B\\
		&\ll_{\varepsilon_i,\tau_i,\delta} B^{\frac{1}{2}-\frac{1}{r}+\delta}+ B^{(4\eta+1)(\frac{1}{2}-\frac{1}{r})+\delta-\frac{1}{r}}+B^{(1+\eta)(\frac{1}{2}-\frac{1}{r})}\log B.
\end{align*}
Pour que cette majoration soit satisfaisante, il faut que
$$ (4\eta+1)\left(\frac{1}{2}-\frac{1}{r}\right)-\frac{1}{r}<(1+\eta)\left(\frac{1}{2}-\frac{1}{r}\right)\Leftrightarrow 3\eta<\frac{2}{r-2}.$$
Cela est valide à cause de la condition \eqref{eq:conditionforeta}.
La preuve du Théorème \ref{th:lowerbound} est achevée.
	\end{proof}
\subsubsubsection{Remarque}
Nous expliquons une raison pour laquelle nous n'avons pas réussi à établir une formule asymptotique pour le zoom sous-critique. Comme l'on a vu, la prédiction naïve affirme que la puissance de $B$ devrait être $1-\frac{2}{r}$. En fait on a
$$b=O_\varepsilon(B^{1-\frac{2}{r}})$$
d'après \eqref{eq:thenumberofnodalcurves}. Le terme principal de la borne inférieure (Théorème \ref{th:lowerbound}) deviendrait $B^{1-\frac{2}{r}}(\log B)^3$ si l'on pourrait prendre la constante $\eta=1$. 
Une difficulté se cache sur l'exactitude de la formule \eqref{eq:asymptoticoncab} pour les paramètres $(a,b)$ grands (c'est-à-dire $a,b>B^\lambda$ pour certain $0<\lambda<1-\frac{2}{r}$). Dans ce cas le zoom sur la courbe nodale $\cab$ compte au plus un point et la technique utilisée dans ce texte ne permet pas de déterminer s'il est non-nul ou pas. Dans \cite{Huang1}, nous avons surmonté un obstacle similaire à l'aide de la transformation de Cremona. Mais cette astuce n'est plus applicable ici car la transformation de Cremona préserve les courbes $\cab$.

\section{Interprétation en terme d'ensemble mince}\label{se:thinset}
Le but de cette section est de signaler le fait qu'il existe une $2:1$-application qui envoie les $\QQ$-points de $Y_4$ sur un ensemble mince de $\PP^1\times\PP^1$.
Notons que, outre le morphisme d'implosion utilisé jusqu'à maintenant, la surface $Y_4$ possède un autre morphisme vers $\PP^1\times\PP^1$ comme suit.

Le diviseur $\omega_{Y_4}^{-1}$ définit un morphisme birationnel $$f:Y_4\to \mathcal{V}\subset \PP^4$$
dont l'image $\mathcal{V}$ est une surface torique de del Pezzo de degré $4$ singulière de type $4\mathbf{A}_1$ définie comme l'intersection de deux quadriques dans $\PP^4$:
\begin{equation}\label{eq:intersectionoftwoquadrics}
x_0x_1=x_2x_3=x_4^2,
\end{equation}
et l'éventail est un \og croix\fg\ qui ressemble à celui de $\PP^1\times\PP^1$. Nous invitons le lecteur à consulter \cite{Derenthal} pour plus de détails sur le sujet des surfaces de del Pezzo singulières.
\begin{figure}[h]
	\centering
	\includegraphics[scale=0.8]{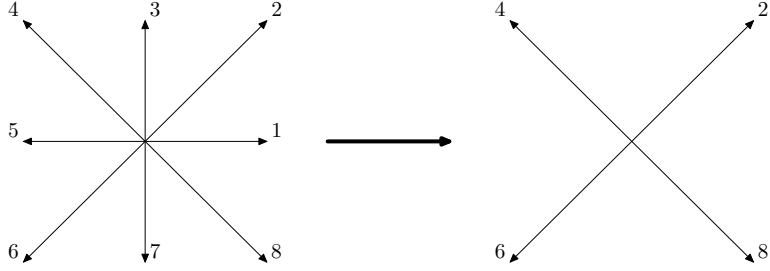}
	\caption{La désingularisation $Y_4\to \mathcal{V}$}
\end{figure}

Le morphisme $f$ est en fait la désingularisation minimale torique de $\mathcal{V}$ puisque les $4$ diviseurs au bord $(1\leqslant i\leqslant 4)$ ont le nombre d'auto-intersection $-2$. La première équation \eqref{eq:intersectionoftwoquadrics}
implique (comme pour toute surface de del Pezzo de degré  $4$ fibrée en conique) que $Y_4$ possède deux projections vers la droite projective, et donc un morphisme vers $\PP^1\times\PP^1$ qui est fini de degré générique $2$ et défini de la façon suivante. D'abord la projection $\wp:\PP^4\dashrightarrow\PP^3$ depuis le point $[0:0:0:0:1]$ est bien définie sur $\mathcal{V}$, dont l'image est définie par $x_0x_1=x_2x_3$, si l'on utilise les coordonnées $[x_0:x_1:x_2:x_3]$ venant de $\PP^4$. C'est une surface quadrique isomorphisme à la variété de produit $\PP^1\times\PP^1$ puisqu'elle s'injecte sur $\PP^3$ par
$$[u:v]\times[s:t]\longmapsto [us:vt:ut:vs].$$  
Donc $\wp$ induit un morphisme $g:\mathcal{V}\to\PP^1\times\PP^1$ de degré générique $2$. L'image de $\mathcal{V}(\QQ)$ est un ensemble \textit{mince} (cf. \cite[\S 9.1]{Serre}) de $(\PP^1\times\PP^1)(\QQ)$:
\begin{equation}\label{eq:thinset}
\{[u:v]\times [s:t]\in(\PP^1\times\PP^1)(\QQ):u,v,s,t\in\ZZ,uvst\in\square\}.
\end{equation} En les composant, on voit que le morphisme $h=g\circ f:Y_4\to\PP^1\times \PP^1$ est un revêtement lisse de degré générique $2$. De plus, on a
$$\omega_{Y_4}^{-1}=f^*(\omega_\mathcal{V}^{-1})=h^*(\mathcal{O}_{\PP^1\times\PP^1}(1,1)).$$
Le lieu de ramification étant sur les diviseurs au bord, autour du point $[1:1]\times[1:1]$, $f$ est un difféomorphisme. Par la fonctorialité de la hauteur on peut se ramener au cas de $\PP^1\times\PP^1$ restreinte à l'ensemble mince \eqref{eq:thinset}. Ceci pourrait réexpliquer le fait que l'approximation rationnelle des nombres rationnels est équivalente à l'approximation rationnelle des nombres quadratiques. 
\section{Autres variétés toriques}\label{se:othervarieties}
Dans cette section, on produira plus de variétés toriques de dimension supérieure sur lesquelles en dehors d'un fermé de Zariski les meilleurs approximants se trouvent dans une famille de courbes nodales, comme pour $Y_4$. Elles sont construites en fait comme produits de $Y_4$ avec les espaces projectifs.

Pour tout $n\in\NN^*$, considérons la variété $V_n=Y_4\times\PP^n_{\QQ}$. On note $\pi_1,\pi_2$ les morphismes de projection vers $Y_4$ et $\PP^n$. \begin{proposition}
	Pour $Q\in V_n(\QQ)$ dans l'orbite ouverte, on a
	$$\aess_{\omega_{V_n}^{-1}}(Q)=\aess_{\omega_{Y_4}^{-1}}(\pi_1(Q))+\aess_{\omega_{\PP^n}^{-1}}(\pi_2(Q))=n+3,$$
	qui peut être atteinte sur une famille de courbes nodales couvrant un ouvert dense de $V_n$.
\end{proposition}
\begin{proof}
On peut supposer que $\pi_1(Q)$ a pour coordonnées $[1:1]\times [1:1]$ et $\pi_2(Q)=[1:\cdots:1]$. Tout d'abord, comme $\omega_{V_n}^{-1}=\pi_1^*(\omega_{Y_4}^{-1})\otimes \pi_2^*(\omega_{\PP^n}^{-1})$, on choisit la hauteur de Weil associée à $\omega_{V_n}^{-1}$ définie pour $x\in Y_4(\QQ),[y_0:\cdots:y_n]\in\PP^n(\QQ)$ tel que $y_0,\cdots,y_n\in\ZZ,\pgcd(y_0,\cdots,y_n)=1$, $$H_{\omega_{V_n}^{-1}}(x,[y_0:\cdots:y_n])=H_{\omega_{Y_4}^{-1}}(x)H_{\omega_{\PP^n}^{-1}}([y_0:\cdots:y_n]), \quad H_{\omega_{\PP^n}^{-1}}([y_0:\cdots:y_n])=\max_{0\leqslant i\leqslant n}(|y_i|^{n+1}),$$ et la distance $d$ comme le maximum de celle définie par \eqref{eq:dist}, notée ici $d_1$ sur $Y_4$ et $d_2$ sur $\PP^n$ définie par la formule (pour $y_0\neq 0$), 
$$d_2(\pi_2(Q),[y_0:\cdots:y_n])=\max_{1\leqslant i\leqslant n}\left|\frac{y_i}{y_0}-1\right|.$$
Par définition, si $[y_0:\cdots:y_n]\neq \pi_2(Q)$, 
$$d_2(\pi_2(Q),[y_0:\cdots:y_n])^{n+1}H_{\omega_{\PP^n}^{-1}}([y_0:\cdots:y_n])\geqslant 1.$$
Donc pour un point général $P\in V_n(\QQ)$ tel que $d_1(\pi_1(Q),\pi_1(P)),d_2(\pi_2(Q),\pi_2(P)))\neq 0$, on a, d'après le calcul \eqref{eq:htdistlowery4} dans la Proposition \ref{po:lowerbound1},
$$d(Q,P)^{n+3} H_{\omega_{V_n}^{-1}}(P)\geqslant d_1(\pi_1(Q),\pi_1(P)))^2 H_{\omega_{Y_4}^{-1}}(\pi_1(P)) \times d_2(\pi_2(Q),\pi_2(P)))^{n+1} H_{H_{\omega_{\PP^n}^{-1}}}(\pi_2(P))\geqslant 1.$$
Cela démontre que $\aess_{\omega_{V_n}^{-1}}(Q)\geqslant n+3$.

Ensuite on construit des courbes nodales dans $V_n$ à partir de celles dans $Y_4$.
Avec les notations dans la Section \ref{se:cab}, pour $(a,b)\in\NN^{*2}_{\text{prem}},a<b$, on a construit une courbe nodale $\cab$ dans $Y_4$ avec le morphisme de paramétrage $\psi_{a,b}$ \eqref{eq:parapsiab}. Maintenant on prend une droite $l$ dans $\PP^n$ avec un paramétrage $\psi_l:\PP^1\to l$ qui envoie $[0:1]$ sur $\pi_2(Q)$. On définit un morphisme de paramétrage comme suit
$$\varPsi_{a,b,l}:\PP^1\to V_n,$$
$$\varPsi_{a,b,l} ([u:v])=(\psi_{a,b}([u:v]),\psi_l\circ \theta_{a,b}([u:v])), $$
où $\theta_{a,b}:\PP^1\to\PP^1$ est défini comme le revêtement double
$$\theta_{a,b}([u:v])=[au^2-bv^2:v^2].$$
On voit que l'image $D_{a,b,l}$ de $\varPsi_{a,b,l}$ est une courbe rationnelle nodale en $Q=\varPsi_{a,b,l}([\pm \sqrt{\frac{b}{a}}:1])$ avec
$$\deg_{\omega_{V_n}^{-1}}(D_{a,b,l})=\deg_{\omega_{Y_4}^{-1}}(\pi_1(D_{a,b,l}))+2\deg_{\omega_{\PP^n}^{-1}}(\pi_2(D_{a,b,l}))=4+2(n+1)=2n+6,$$
puisque $\theta_{a,b}^*(\mathcal{O}_{\PP^1}(\deg_{\omega_{\PP^n}^{-1}}(\pi_2(D_{a,b,l}))))=\theta_{a,b}^*(\mathcal{O}_{\PP^1}(n+1))=\mathcal{O}_{\PP^1}(2n+2)$.
En vertu de la remarque \ref{rmk:thmofmckinnon}, nous obtenons
$$\alpha_{\omega_{V_n}^{-1}}(Q,D_{a,b,l})=\frac{\deg_{\omega_{V_n}^{-1}}(D_{a,b,l})}{2}=n+3.$$
Comme la réunion la famille $(D_{a,b,l})$ avec $(a,b)\in\NN^{*2}_{\operatorname{prem}},a<b,(a,b)\not\in \square^2$ et $l$ variant est dense pour la topologie de Zariski dans $V_n$, ceci nous fournit la borne inférieure $\aess_{\omega_{V_n}^{-1}}(Q)\leqslant n+3$. 
\end{proof}
\newpage
\begin{center}
	\LARGE Appendices
\end{center}
\appendix
\section{Distribution locale d'un point rationnel sur la droite projective}\label{se:app1}
	Dans un souci de complétude, nous redémontrons le résultat de S. Pagelot concernant la distribution locale d'un $\QQ$-point $Q$ sur la droite projective $\PP_\QQ^1$. Pour simplicité on suppose que $Q=[0:1]$; le résultat pour un point général diffère par une constante. 
	\subsection{Énoncé du théorème}
	\begin{theorem}[Pagelot \cite{Pagelot}]\label{th:pagelot}
On a $\alpha(Q,\PP^1)=\aess(Q)=1$.
	 On fixe la hauteur de Weil absolue associée au fibré $\mathcal{O}(1)$ définie par
	 $$H([u:v])=\max(|u|,|v|),\quad (u,v)\in\ZZ,\quad\pgcd(u,v)=1,$$
	 et l'on note $r\geqslant 1$ le facteur de zoom.
	 Soit $f$ une fonction intégrable à support compact sur $T_Q \PP^1$. Alors
	\begin{itemize}
		\item si $r=1$ alors
		$$\delta_{\PP^1,Q,B,1}(f)=B\int f(x)\frac{\sigma(x)\operatorname{d}x}{x^2}+O_{f}(1),$$
		où
		$$\sigma(x)=\sum_{n\leqslant |x|}\varphi(n),$$ 
		\item si $r>1$, alors 
		$$\delta_{\PP^1,Q,B,r}(f)=B^{2-\frac{1}{r}}\frac{3}{\pi^2}\int f(x)\operatorname{d}x+O_{f}(B(\log B)^\frac{2}{3}(\log\log B)^\frac{4}{3}).$$
	\end{itemize}
\end{theorem}
On voit que das les deux cas la mesure asymptotique existe. Pour le zoom critique, si $\operatorname{Supp}(f)\subset \mathopen]-1,1\mathclose[$, alors $\delta_{\PP^1,Q,B,1}(f)=0$. On observe donc qu'il y a un \og trou\fg{}\ autour du point $Q$. Pour un zoom sous-critique (c'est-à-dire $r>1$), la distribution est uniforme.
\subsection{Démonstration du Théorème \ref{th:pagelot}}
Rien dans cette section n'est nouveau. Le but est de rappeler  comment traiter ce type de problème dans le cas le plus simple. 
Une observation basique, qui est aussi celle que l'on va suivre dans la suite, est qu'au lieu de considérer des fonctions intégrables générales, il suffit de regarder les fonctions \og simple\fg, c'est-à-dire les fonctions caractéristiques d'intervalles, car ces fonctions suffisent pour caractériser la convergence faible de mesures de probabilité.

On note les coordonnées de $\PP^1_\RR$ par $[u:v]$. Sans perte de généralité on peut supposer que $u>0$. En supposant que le point à approcher est $Q=[0:1]$, on utilise le difféomorphisme local $\rho$ défini par
$$[u:v]\longmapsto \frac{u}{v}\in\RR.$$
Alors $Q$ est envoyé sur l'origine $O$.
Sur $T_Q \PP^1$ on utilise la distance canonique $$d([u:v],Q)=\left|\frac{u}{v}\right|.$$
Par symétrie, il suffit d'examiner les fonctions caractéristique des intervalles $\mathopen]\eta,\varepsilon\mathclose]\subset \mathopen]0,\infty[$. On désigne par $\chi(\eta,\varepsilon)$ une telle fonction 
Nous remarquons ici que l'usage de cette notation a un sens légèrement différent dans les parties précédentes du texte. L'étude de la convergence de la suite $(\delta_{\PP^1,Q,B,r}(\chi(\eta,\varepsilon)))_B$ revient à l'estimation du cardinal de l'ensemble 
\begin{equation}\label{eq:countingrationp1}
F(\varepsilon,B,r)=\left\{
\begin{aligned}
&P=[u:v]\in\PP^1(\QQ),
P\neq[0:1]\\
&(u,v)\in\NN_{\operatorname{prem}}^{*2}
\end{aligned}
\left|
\begin{aligned}
&0<B^{\frac{1}{r}}d(P,Q)=B^{\frac{1}{r}}\frac{u}{v}\leqslant \varepsilon\\
&H(P)=\sup(u,v)\leqslant B
\end{aligned}
\right\}\right. ,
\end{equation}
puisque $\delta_{\PP^1,Q,B,r}(\chi(\eta,\varepsilon))=\sharp F(\varepsilon,B,r)-\sharp F(\eta,B,r)$.
\subsubsection{Cas critique: $r=1$}
Les conditions \eqref{eq:countingrationp1} sur $(u,v)$ implique que 
$$u\leqslant \frac{\varepsilon v}{B}\leqslant \varepsilon.$$
Cela implique qu'il n'y a qu'un nombre fini de $u$ dans un voisinage fixé et nous permet de fixer $u$ tout en comptant $v$. On peut écrire
$$ F(\varepsilon,B,1)=\bigsqcup_{u\leqslant\varepsilon} F_u(\varepsilon,B),$$
où
\begin{equation*}
F_u(\varepsilon,B)=\{v\in\NN:\pgcd(u,v)=1, u\varepsilon^{-1}B\leqslant v\leqslant B\}.
\end{equation*}
On rappelle la définition de la fonction $\phi$ \eqref{eq:thefunctionphi} et l'on en déduit, en utilisant \cite[Exercise 5.2]{Browningbook},
\begin{equation*}
\sharp F_u(\varepsilon,B)=\phi(u)\left(1-\frac{u}{\varepsilon}\right)B+O(\tau(u)).
\end{equation*}
Alors en appliquant \cite[\S I.5.2]{tenenbaum} \begin{align*}
\delta_{\PP^1,Q,B,1}(\chi(\eta,\varepsilon))&=\sum_{u\leqslant \varepsilon}\sharp F_u(\varepsilon,B)-\sum_{u\leqslant \eta}\sharp F_u(\eta,B)\\
&=\sum_{u\leqslant \eta} u\phi
(u)\left(\frac{1}{\eta}-\frac{1}{\varepsilon}\right)B+\sum_{\eta<u\leqslant\varepsilon}\phi(u)\left(1-\frac{u}{\varepsilon}\right)B+O{\varepsilon}(1)\\
&=B \int \chi(\eta,\varepsilon)(x)\frac{\sigma(x)\operatorname{d}x}{x^2}+O{\varepsilon}(1).
\end{align*} 
Cela clôt la démonstration du cas critique.
\subsubsection{Cas $r>1$}
Comme précédemment on a la décomposition $$F(\varepsilon,B,r)=\bigsqcup_{u\leqslant \varepsilon} F_u(\varepsilon,B,r),$$
où
$$F_u(\varepsilon,B,r)=\{v\in\NN^*:u\varepsilon^{-1}B^\frac{1}{r}\leqslant v\leqslant B,\pgcd(u,v)=1\}.$$
Fixons $u$ dans cette réunion. 
Pour que $F_u(\varepsilon,B,r)$ soit non-vide, il faut que $$u\varepsilon^{-1}B^\frac{1}{r}\leqslant B \Leftrightarrow u\leqslant\varepsilon B^{1-\frac{1}{r}}.$$
Avant de calculer le cardinal, citons d'abord quelques formules asymptotique bien connues (voir \cite{tenenbaum} \S3.2, \S3.4):
\begin{equation}\label{fo:eulerfunction}
\sum_{n\leqslant x} \varphi(n)=\frac{3}{\pi^2}x^2+O(x(\log x)^\frac{2}{3}(\log\log x)^\frac{4}{3}),
\end{equation}
\begin{equation}\label{fo:phifunction}
\sum_{n\leqslant x} \phi(n)=\frac{6}{\pi^2}x+O((\log x)^\frac{2}{3}(\log\log x)^\frac{4}{3}),
\end{equation}
\begin{equation}\label{fo:divisorfunction}
\sum_{n\leqslant x}\tau(n)=O(x\log x).
\end{equation}
En appliquant ces formules, on obtient, encore d'après \cite[Exercise 5.2]{Browningbook}
\begin{align*}
\sharp F(\varepsilon,B,r)
&=\sum_{u\leqslant \varepsilon B^{1-\frac{1}{r}}} \left(\phi(u)\left(B-\frac{uB^{\frac{1}{r}}}{\varepsilon}\right)+O(\tau(u))\right)\\
&=\left(\frac{6}{\pi^2} \varepsilon B^{2-\frac{1}{r}}-\frac{3}{\pi^2}\varepsilon B^{2-\frac{1}{r}}\right)+O_\varepsilon( B(\log B)^{\frac{2}{3}}(\log\log B)^{\frac{4}{3}})+O_\varepsilon(B^{1-\frac{1}{r}}\log B)\\
&=\frac{3}{\pi^2}\varepsilon B^{2-\frac{1}{r}}+O_\varepsilon(B(\log B)^\frac{2}{3}(\log\log B)^\frac{4}{3}),
\end{align*}
d'où
$$\delta_{\PP^1,Q,B,r}(\chi(\eta,\varepsilon))=B^{2-\frac{1}{r}}\left(\frac{3}{\pi^2}\int \chi(\eta,\varepsilon)(x)\operatorname{d}x\right)+O_\varepsilon(B(\log B)^\frac{2}{3}(\log\log B)^\frac{4}{3}).$$
L'expression finale est achevée.
\begin{remark*}
	Quand $r>1$, les points à dénombrer se trouvent dans un triangle dont la longueur du bord est négligeable par rapport à l'aire. On pourrait adapter la méthode utilisée dans la démonstration du même problème pour les nombres irrationnels (cf. Théorème \ref{th:localdistributionofrealnumbers}). Le raisonnement décrit ici donne un meilleur terme d'erreur au niveau de la puissance de $\log B$.
\end{remark*}
\subsection{À propos des mesures limites}
L'ordre de grandeur des nombres de points dans les zooms pour $\PP^1$ étant en accord avec celui prévu, nous discutons maintenant à quoi correspondent les mesures limites.

Prenons comme d'habitude la fonction caractéristique $\chi(\varepsilon,\eta)$. 
Si l'on identifie $\PP^1(\QQ)$ avec l'ensemble des points primitifs dans $\ZZ^2$ le comptage de points de hauteur bornée 
$N(B)=\{P\in\PP^1(\QQ):H(P)\leqslant B\} $ équivaut au comptage des points primitifs dans le rectangle $R(B)=\{(x,y)\in\RR^2:\max(|x|,|y|)\leqslant B\}$. L'opération de zoom revient à prendre des points primitifs dans l'intérieur d'un triangle (cf. Figure \ref{fig:area}) noté $\triangle(\eta,\varepsilon,B,r)$ dont la taille dépend du facteur de zoom $r$ et de la borne $B$.

Pour les zooms sous-critiques $r>1$, l'ordre de grandeur de l'aire domine celui du bord, donc elle donne le terme principal. On a que 
$$\frac{\delta_{\PP^1,Q,B,r}(\chi(\eta,\varepsilon))}{\operatorname{Vol}(\triangle(\eta,\varepsilon,B,r))}=\frac{\sharp N(B) }{\operatorname{Vol}(R(B))}.$$
On voit que dans ce cas là les points primitifs sont équidistribués et l'on obtient une mesure proportionnelle à la mesure de Lebesgue. Cela coïncide avec celui du zoom sous-critique pour tous les nombres algébriques.

Pour le zoom critique $r=1$, l'aire de $\triangle(\eta,\varepsilon,B,1)$ et son bord ont le même ordre de grandeur $B$. Comme l'on a vu dans la démonstration, les points se trouvent en fait dans les droites horizontales dans l'intérieur du triangle dont la longueur est $u\left(\frac{1}{\eta}-\frac{1}{\varepsilon}\right)B$ (cf Figure \ref{fig:ratr1}).
\begin{figure}[h]
	\centering
	\includegraphics[scale=0.8]{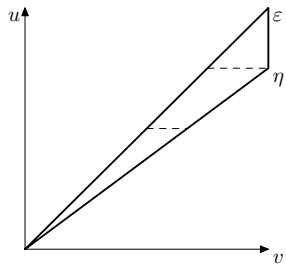}
	\caption{Le triangle $\triangle(\eta,\varepsilon,B,1)$}
	\label{fig:ratr1}
\end{figure}
Puisque $$\frac{\delta_{\PP^1,Q,B,1}(\chi(\eta,\varepsilon))}{B}=\sum_u \sharp L(u),$$ où 
$$L(u)=\left\{v\in \NN^*:\frac{u}{v} \text{ est une fraction réduite }, v\in \mathopen]\frac{u}{\varepsilon},\frac{u}{\eta}\mathclose]\right\}.$$
On en conclut que sur chaque droite on compte des nombres rationnels avec le numérateur fixé. 
Cela explique d'où viennent la fonction densité $\frac{1}{x^2}$ et la fonction arithmétique $\sigma$. On remarque que le phénomène pour des nombres algébriques est radicalement différent (pour eux dans le zoom de facteur $1$ on trouve aussi un autre type d'équidistribution, voir la Section \ref{se:weakzoomofalgebraicnumbers}).
\section{Le problème des diviseurs pour des formes binaires cubiques déployées}\label{se:app2}
\subsection{Énoncé du résultat}
Dans \cite{Browning}, T. D. Browning a étudié l'ordre moyen du nombre de diviseurs pour des formes binaires cubiques déployées, dont la technique remonte à \cite{B-B1}, où les auteurs ont déduit une formule asymptotique de l'ordre moyen primitif de fonctions arithmétiques \og ressemblant\fg{}\ à la fonction donnant le nombres de diviseurs pour certaines formes binaires quartique. Ils l'utilisent pour démontrer que le nombre de points de hauteur bornée sur une surface de del Pezzo de degré $4$ s'accorde avec la prédiction de Batyrev-Manin-Peyre. En les imitant, nous allons déduire une telle formule pour une forme binaire cubique particulière (la même démonstration marche pour toutes les formes binaires cubiques déployées, ce dont nous n'aurons pas besoin ici).
On considère les fonctions arithmétiques
	\begin{equation}\label{eq:arithmeticfun}
\Psi_1(n)=\prod_{p|n}\left(1+\frac{1}{p}\right)^{-1},\quad \Psi(n)=\sum_{d|n}\Psi_1(d)\sum_{e|d}\frac{\mu(e)}{e}=\sum_{d|n}\Psi_1(d)\phi(d),
\end{equation}
et la constante $C_1$:
\begin{equation}\label{eq:C1}
C_1=\prod_{p}\left(1-\frac{1}{p}\right)^3\left(1+\frac{3}{p}-\frac{1}{p^2}-\frac{18}{p(p+2)}\right).
\end{equation}
	\begin{proposition}\label{po:problemofdivisor}
	Soient $\tau_1>\tau_2>1,X\gg 1$, on a
	\begin{equation}\label{po:asymptoticforPsi}
	\sum_{\substack{\max(x_1,x_2)\leqslant X\\ 1<\tau_2\leqslant x_2x_1^{-1}\leqslant \tau_1\\ \pgcd(x_1,x_2)=1}} \Psi(x_1)\Psi(x_2)\Psi(x_2-x_1)=\frac{C_1}{2}\left(\frac{1}{\tau_2}-\frac{1}{\tau_1}\right) X^2(\log X)^3+O(X^2 (\log X)^2),
	\end{equation}
 \begin{equation}\label{po:asymptoticforPsiquotient}
	\sum_{\substack{\max(x_1,x_2)\leqslant X\\ 1<\tau_2\leqslant x_2x_1^{-1}\leqslant \tau_1\\ \pgcd(x_1,x_2)=1}} \frac{\Psi(x_1)\Psi(x_2)\Psi(x_2-x_1)}{x_2x_1^\frac{1}{2}}=4C_1\left(\frac{1}{\sqrt{\tau_2}}-\frac{1}{\sqrt{\tau_1}}\right) X^{\frac{1}{2}}(\log X)^3+O(X^{\frac{1}{2}}(\log X)^2).
 \end{equation}
\end{proposition}
\subsection{Préliminaires et résultats connus}
On fixe trois formes linéaires à coefficients entiers primitives $L_1,L_2,L_3$ en deux variables deux à deux linéairement indépendantes et $\mathcal{R}$ un sous-ensemble convexe fermé borné de $\RR^2$ dont les coordonnées sont notées $\mathbf{x}=(x_1,x_2)$. 
Pour $X>0$, on note
$$X\mathcal{R}=\{X\mathbf{x}:\mathbf{x}\in\mathcal{R}\}.$$ On suppose que $L_i(\mathbf{x})>0$ pour $\mathbf{x}\in \mathcal{R}$ et $i\in\{1,2,3\}$.
On note $$L_\infty=L_\infty(L_1,L_2,L_3)=\max\{\|L_1\|,\|L_2\|,\|L_3\|\},$$
où $\|L_i\|$ désigne le maximum des valeurs absolues des coefficients de $L_i$. 
On note aussi
$$r_\infty=r_\infty(\mathcal{R}) =\sup_{\mathbf{x}\in\mathcal{R}}\max(|x_1|,|x_2|),$$
$$r^\prime=r^\prime(L_1,L_2,L_3,\mathcal{R})=\max_{1\leqslant i\leqslant 3}(\sup_{\mathbf{x}\in\mathcal{R}} L_i(\mathbf{x})).$$
Pour $\mathbf{D}=(D_1,D_2,D_3)\in\NN_{\geqslant 1}^3$, on note $D=D_1D_2D_3$ et
$$\Lambda(\mathbf{D})=\{\mathbf{x}\in\ZZ^2:D_i|L_i(\mathbf{x}),i\in\{1,2,3\}\},\quad \varrho(\mathbf{D})=\sharp(\Lambda(\mathbf{D};L_1,L_2,L_3)\cap \mathopen[0,D\mathclose[ ^2).$$
La fonction $\varrho$ se calcule comme
$$\varrho(\mathbf{D};L_1,L_2,L_3)=\frac{D^2}{\det(\Lambda(\mathbf{D}))},$$ 
puisque $\Lambda(\mathbf{D})$ est un sous-réseau de $D\ZZ^2$.
Elle est donc multiplicative en dimension $3$:
$$\varrho(g_1h_1,g_2h_2,g_3h_3)=\varrho(g_1,g_2,g_3)\varrho(h_1,h_2,h_3),$$
pourvu que $\pgcd(g_1g_2g_3,h_1h_2h_3)=1$.
Pour $p$ un nombre premier, on définit \begin{equation}\label{eq:sigmap}
\sigma_p(L_1,L_2,L_3)=\left(1-\frac{1}{p}\right)^3 \sum_{\mathbf{\nu}\in\ZZ^3_{\geqslant 0}}\frac{\varrho(p^{\nu_1},p^{\nu_2},p^{\nu_3};L_1,L_2,L_3)}{p^{2(\nu_1+\nu_2+\nu_3)}}.
\end{equation}
On note  $\delta(\mathbf{D})$ le plus grand entier $\delta$ tel que $\Lambda(\mathbf{D})\subset \delta\ZZ^2$.
Pour $\mathbf{d},\mathbf{D}\in\NN_{\geqslant 1}^3$ tels que $d_i|D_i,\forall i\in\{1,2,3\}$, on a besoin de l'ordre moyen de la quantité suivante
\begin{equation}\label{eq:sumofdivisorsinitial}
S(X,\mathbf{d},\mathbf{D};L_1,L_2,L_3)=\sum_{\mathbf{x}\in\Lambda(\mathbf{D})\cap X\mathcal{R}}\tau\left(\frac{L_1(\mathbf{x})}{d_1}\right)\tau\left(\frac{L_2(\mathbf{x})}{d_2}\right)\tau\left(\frac{L_3(\mathbf{x})}{d_3}\right).
\end{equation}
\begin{theorem}[\cite{Browning}, Theorem 3]\label{th:browning}
	Soient $\varepsilon>0$, $\frac{1}{4}<\theta <1$. Supposons que $r^\prime X^{1-\theta}\geqslant 1$. Alors il existe un polynôme $P\in\RR[T]$ de degré $3$ tel que
	$$S(X,\mathbf{d},\mathbf{D};L_1,L_2,L_3)=\operatorname{vol}(\mathcal{R})X^2P(\log X)+O_\varepsilon\left(\frac{D^\varepsilon L_\infty ^{2+\varepsilon} r_\infty^\varepsilon}{\delta(\mathbf{D})}(r_\infty r^{\prime\frac{3}{4}}+r_\infty^2)X^{\frac{7}{4}+\varepsilon}\right),$$
	où $$\|P\|=O_\varepsilon(D^\varepsilon L_\infty^\varepsilon r_\infty^\varepsilon(1+r^{\prime -1})^\varepsilon (\operatorname{det}\Lambda(\mathbf{D}))^{-1}),$$
	et le coefficient du terme principal de $P$ est  $C(\mathbf{d},\mathbf{D})=\prod_{p}\sigma_p(\mathbf{d},\mathbf{D};L_1,L_2,L_3)$ avec
\begin{equation}\label{eq:sigmape}
	\sigma_p(\mathbf{d},\mathbf{D};L_1,L_2,L_3)=\left(1-\frac{1}{p}\right)^3\sum_{\mathbf{\nu}\in\ZZ^3_{\geqslant 0}}\frac{\varrho(p^{N_1},p^{N_2},p^{N_3},L_i)}{p^{2(N_1+N_2+N_3)}},
\end{equation}
	et pour $\mathbf{\nu}=(\nu_1,\nu_2,\nu_3)\in\ZZ^3_{\geqslant 0}$ et $i\in\{1,2,3\}$,
	\begin{equation}\label{eq:N_i}
	N_i=\max(v_p(D_i),\nu_i+v_p(d_i)).
	\end{equation}
\end{theorem}
Pour une utilisation ultérieure, nous voudrions en savoir plus sur la majoration de la constante $\sigma_p(\mathbf{d},\mathbf{D};L_1,L_2,L_3)$ ainsi que sur la constante $C(\mathbf{d},\mathbf{D})$. 
\begin{lemma}\label{le:browning}
	Pour tout $\varepsilon>0$, on a
	$$C(\mathbf{d},\mathbf{D})\ll_\varepsilon \frac{D^\varepsilon L_\infty^\varepsilon}{\det \Lambda(\mathbf{D})}.$$
\end{lemma}
\begin{proof}
	Nous esquissons des arguments se trouvant dans \cite[\S 2-\S 3]{Browning}.
	En introduisant les formes $M_i,1\leqslant i\leqslant 3$ \cite[p. 590 ligne -4]{Browning}, on a d'après \cite[p. 591 ligne 10]{Browning},
	$$L_\infty(M_1,M_2,M_3)=\max_{1\leqslant i\leqslant 3}(\|M_i\|)\leqslant D_1D_2D_3 L_\infty=DL_\infty.$$
	Grâce à \cite[p.592 lignes 6 et 12]{Browning},
	$$C(\mathbf{d},\mathbf{D})=\frac{\prod_p\sigma_p(M_1,M_2,M_3)}{\det(\Lambda(\mathbf{D}))}$$
	Il résulte du \cite[Lemma 2.4]{Browning} avec $\delta_1=\delta_2=\delta_3=0$ sous les notations dans \cite[p. 590 ligne 17]{Browning} que
	$$\prod_p\sigma_p(M_1,M_2,M_3)\ll_{\varepsilon} L_\infty(M_1,M_2,M_3)^\varepsilon.$$
	En rassemblant tous, on obtient
	$$C(\mathbf{d},\mathbf{D})\ll_{\varepsilon} \frac{L_\infty(M_1,M_2,M_3)^\varepsilon}{\det(\Lambda(\mathbf{D}))}\leqslant \frac{D^\varepsilon L_\infty^\varepsilon}{\det \Lambda(\mathbf{D})}.$$
	D'où la majoration souhaitée.
\end{proof}
\subsection{Démonstration des formules asymptotiques}
On définit une fonction arithmétique multiplicative $h=\Psi*\mu*\mu$. 
Rappelons la formule \eqref{eq:arithmeticfun}. Un calcul donne
\begin{equation*}
	(\Psi*\mu)(p^k)=\begin{cases}
	1-\frac{2}{p+1} \text{ si } k\geqslant 1;\\
	1 \text{ si } k=0.
	\end{cases}
\end{equation*}
Donc
\begin{equation*}
	h(p^k)=\begin{cases}
	0 \text{ si } k\geqslant 2;\\
	-\frac{2}{p+1} \text{ si } k=1;\\
	1 \text{ si } k=0.
	\end{cases}
\end{equation*}
On obtient que 
\begin{equation*}
h(n)=\begin{cases}
1 \text{ si } n=1;\\
\prod_{p|n}-\frac{2}{p+1} \text{ si } n\neq 1 \text{ et } \mu^2(n)=1;\\
0 \text{ si } \mu^2(n)=0.
\end{cases}
\end{equation*}
En particulier
\begin{equation}\label{eq:hsmall}
|h(n)|\leqslant\frac{2}{n}, \quad \forall n\in\NN_{\geqslant 1}.
\end{equation}
Donc $h$ est \textit{petite} dans le sens de (2.19) dans \cite{B-B1}: il existe $\delta_0$ petit tel que la série
$$\sum_{d\in\NN_{\geqslant 1}}\frac{|h(d)|}{d^{\frac{1}{2}-\delta_0}}$$
converge. Cette propriété jouera un rôle important  dans le traitement des termes d'erreur.
Dans la suite on fixe nos formes $L_i$ comme
\begin{equation}\label{eq:threeforms}
L_1=x_1,\quad L_2=x_2,\quad L_3=x_2-x_1,
\end{equation}
et la région 
\begin{equation}\label{eq:regionR}
\mathcal{R}=\mathcal{R}(\tau_1,\tau_2)=\{\mathbf{x}\in\RR^2:0<x_1,x_2\leqslant 1,\tau_2\leqslant x_2x_1^{-1}\leqslant \tau_1\}.
\end{equation}
Le but est d'obtenir une formule asymptotique en sommant sur les points dans $\mathcal{R}$ dont les coordonnées $x_1,x_2$ sont entières et premières entre elles. Cela consiste en une application directe du Théorème \ref{th:browning} avec une inversion de Möbius comme fait dans \cite[Corollaire 1]{B-B1} pour certaines formes binaires de degré $4$. Malheureusement cela n'est pas fait dans \cite{Browning}. Nous suivons les techniques venant de \cite{B-B1} pour démontrer les formules asymptotiques \eqref{po:asymptoticforPsiquotient}.

On définit pour $\mathbf{d}=(d_1,d_2,d_3)\in\NN_{\geqslant 1}^3$,
$$\varrho^*(\mathbf{d})=\varrho^*(\mathbf{d};L_1,L_2,L_3)=\sharp\{\mathbf{x}\in\Lambda(\mathbf{d})\cap\mathopen[0,d_1d_2d_3\mathclose[^2 :\pgcd(x_1,x_2,d_1d_2d_3)=1\}.$$ 
Soient $\mathbf{D}=(D_1,D_2,D_3),\mathbf{d}=(d_1,d_2,d_3)\in\NN_{\geqslant 1}^3$ comme précédemment avec $\pgcd(d_i,d_j)=\pgcd(D_i,D_j)=1,\forall i,j\in\{1,2,3\},i\neq j$. La sommation dont la formule asymptotique cherchée est 
\begin{equation}\label{eq:sumofdivisorsprimitive}
S^*(X,\mathbf{d},\mathbf{D})=\sum_{\substack{\mathbf{x}\in X\mathcal{R}\cap \Lambda(\mathbf{D})\\ \pgcd(x_1,x_2)=1}} \tau\left(\frac{L_1(\mathbf{x})}{d_1}\right)\tau\left(\frac{L_2(\mathbf{x})}{d_2}\right)\tau\left(\frac{L_3(\mathbf{x})}{d_3}\right).
\end{equation}
\begin{lemma}
	Pour tout $\varepsilon>0$, on a
	\begin{equation}\label{eq:Sstar}
	\begin{split}
	S^*(X,\mathbf{d},\mathbf{D})
	=C^*(\mathbf{d},\mathbf{D})\operatorname{vol}(\mathcal{R})X^2(\log X)^3 +O_{\varepsilon}\left(\left(\frac{D^\varepsilon}{\delta(\mathbf{D})}+1\right)X^{\frac{23}{12}+\varepsilon}+\frac{D^\varepsilon }{\operatorname{det}(\Lambda(\mathbf{D}))}X^2(\log X)^2\right),
	\end{split}
	\end{equation}
	où $C^*(\mathbf{d},\mathbf{D})=\prod_{p}\sigma^*(\mathbf{d},\mathbf{D})$ avec (rappelons les notations $N_i$ \eqref{eq:N_i})
	\begin{equation}\label{eq:sigmapstar}
	\sigma_p^*(\mathbf{d},\mathbf{D})=\left\{
	\begin{aligned}
	&\left(1-\frac{1}{p}\right)^3\sum_{\mathbf{\nu}\in\ZZ^3_{\geqslant 0}}\frac{\varrho^*(p^{N_1},p^{N_2},p^{N_3})}{p^{2(N_1+N_1+N_3)}}   \text{ si } v_p(D)\geqslant 1\\
	&\left(1-\frac{1}{p}\right)^3\left(1-\frac{1}{p^2}+\sum_{\substack{\mathbf{\nu}\in\ZZ^3_{\geqslant 0}\\\nu_1+\nu_2+\nu_3\geqslant 1}}\frac{\varrho^*(p^{\nu_1},p^{\nu_2},p^{\nu_3})}{p^{2(\nu_1+\nu_1+\nu_3)}} \right)  \text{ si } v_p(D)=0
	\end{aligned}\right.
	\end{equation}
\end{lemma}
\begin{proof}
Par une inversion de Möbius, on a
\begin{equation}\label{eq:mobiustransform}
S^*(X,\mathbf{d},\mathbf{D})=\sum_{e\in\NN_{\geqslant 1}} \mu(e)\sum_{\substack{\mathbf{x}\in X\mathcal{R}\cap \Lambda(\mathbf{D})\\ e|\pgcd(x_1,x_2)}} \tau\left(\frac{L_1(\mathbf{x})}{d_1}\right)\tau\left(\frac{L_2(\mathbf{x})}{d_2}\right)\tau\left(\frac{L_3(\mathbf{x})}{d_3}\right).
\end{equation}
On définit $y_i=x_i/e,i\in\{1,2,3\}$, et on note 
$$\Lambda_{e}(\mathbf{D})=\Lambda(\mathbf{D};eL_1,eL_2,eL_3)=\{\mathbf{x}\in\ZZ^2:D_i|eL_i(\mathbf{x}),i\in\{1,2,3\}\}.$$
Comme
$$\mathbf{x}\in X\mathcal{R}\cap\Lambda(\mathbf{D})\cap e\ZZ^2\quad \Leftrightarrow \quad \mathbf{y}\in (X/e)\mathcal{R}\cap \Lambda_{e}(\mathbf{D}),$$
en notant pour $e\in\NN_{\geqslant 1}$,
\begin{equation}\label{eq:sum3}
S(X/e,\mathbf{d},\mathbf{D};eL_1,eL_2,eL_3)= \sum_{\mathbf{x}\in\Lambda_{e}(\mathbf{D})\cap (X/e)\mathcal{R}}\tau\left(\frac{eL_1(\mathbf{x})}{d_1}\right)\tau\left(\frac{eL_2(\mathbf{x})}{d_2}\right)\tau\left(\frac{eL_3(\mathbf{x})}{d_3}\right),
\end{equation}
la somme \ref{eq:sumofdivisorsprimitive} s'écrit comme
\begin{equation}\label{eq:sum1}
S^*(X,\mathbf{d},\mathbf{D})=\sum_{e\in\NN_{\geqslant 1}} \mu(e) S(X/e,\mathbf{d},\mathbf{D};eL_1,eL_2,eL_3).
\end{equation}
On note 
\begin{equation}\label{eq:Cde}
C(\mathbf{d},\mathbf{D},e)=\prod_{p}\sigma_p(\mathbf{d},\mathbf{D};eL_1,eL_2,eL_3).
\end{equation}
On applique le Théorème \ref{th:browning} à $S(X/e,\mathbf{d},\mathbf{D};eL_1,eL_2,eL_3)$ \eqref{eq:sum3} en notant que 
$$L_\infty(eL_1,eL_2,eL_3)=e,\quad r^\prime(eL_1,eL_2,eL_3)=e,\quad r_\infty=1,$$ 
et l'on obtient, 
\begin{equation}\label{eq:sum2}
\begin{split}
S(X/e,\mathbf{d},\mathbf{D};eL_1,eL_2,eL_3)
&=C(\mathbf{d},\mathbf{D},e)\operatorname{vol}(\mathcal{R})\left(\frac{X}{e}\right)^2\left(\log\left(\frac{X}{e}\right)\right)^3\\ &\quad+O_\varepsilon\left(\frac{D^\varepsilon e^{2+\varepsilon} }{\delta(\mathbf{D})}(e^{\frac{3}{4}}+1)\left(\frac{X}{e}\right)^{\frac{7}{4}+\varepsilon}\right)+O_\varepsilon\left(X^2(\log X)^2\frac{\log e}{e^2}\frac{D^\varepsilon e^\varepsilon}{\operatorname{det}(\Lambda(\mathbf{D}))}\right)\\
&=C(\mathbf{d},\mathbf{D},e)\operatorname{vol}(\mathcal{R})\left(\frac{X}{e}\right)^2\left(\log\left(\frac{X}{e}\right)\right)^3\\ &\quad+O_\varepsilon\left(\frac{D^\varepsilon}{\delta(\mathbf{D})}eX^{\frac{7}{4}+\varepsilon}\right)+O_\varepsilon\left(X^2(\log X)^2\frac{\log e}{e^{2-\varepsilon}}\frac{D^\varepsilon }{\operatorname{det}(\Lambda(\mathbf{D}))}\right).
\end{split}
\end{equation}
En utilisant le Lemme \ref{le:browning} pour $C(\mathbf{d},\mathbf{D},e)$, on obtient
\begin{equation}\label{eq:upperboundCdD}
C(\mathbf{d},\mathbf{D},e)\ll_{\varepsilon}\frac{(De)^\varepsilon}{\det(\Lambda_{e}(\mathbf{D}))}\leqslant \frac{(De)^\varepsilon}{\det(\Lambda(\mathbf{D}))},
\end{equation}
d'où la convergence de la somme $\sum_{e=1}^{\infty} \frac{\mu(e)}{e^2}C(\mathbf{d},\mathbf{D},e)$.
On définit pour $k\in\NN_{\geqslant 1},\nu=(\nu_1,\nu_2,\nu_3)\in\ZZ^3$,
$$\bar{\varrho}_k(p^{\nu_1},p^{\nu_2},p^{\nu_3})=\frac{\sharp\{\mathbf{x}\in\mathopen[0,p^{\nu_1+\nu_2+\nu_3}\mathclose[^2:p^{\nu_i}|p^{v_p(k)}L_i(\mathbf{x}),1\leqslant i\leqslant 3\}}{p^{2(\nu_1+\nu_2+\nu_3)}}.$$
On calcule maintenant, en rappelant les notations \eqref{eq:sigmape} et \eqref{eq:Cde}, 
\begin{equation}\label{eq:cstar}
\begin{split}
\sum_{e=1}^{\infty} \frac{\mu(e)}{e^2}C(\mathbf{d},\mathbf{D},e)
&=\sum_{e\in\NN_{\geqslant 1}} \mu(e)\left(\frac{1}{e^2}\prod_{p} \left(1-\frac{1}{p}\right)^3\sum_{\mathbf{\nu}\in\ZZ^3_{\geqslant 0}}\bar{\varrho}_e(p^{N_1},p^{N_2},p^{N_3})\right)\\
&=\prod_p \left(1-\frac{1}{p}\right)^3\sum_{\mathbf{\nu}\in\ZZ^3_{\geqslant 0}}\sum_{k\in\{0,1\}}(-1)^k \frac{\bar{\varrho}_{p^k}(p^{N_1},p^{N_2},p^{N_3})}{p^{2k}}.
\end{split}
\end{equation}
Vérifions que les facteurs locaux dans \eqref{eq:cstar} sont donnés par \eqref{eq:sigmapstar}.
Fixons un nombre premier $p$. Si $v_p(D)\geqslant 1$, on définit le réseau 
$$\Lambda^\prime=\{\mathbf{x}\in\ZZ^2:p^{N_i}|pL_i(\mathbf{x}),1\leqslant i\leqslant 3\}.$$
Alors on a
$$\sharp(\Lambda^\prime\cap \mathopen[0,p^{N_1+N_2+N_3}\mathclose[^2)=\frac{p^{2(N_1+N_2+N_3)}}{\det(\Lambda^\prime)}=p^2 \sharp(\Lambda^\prime\cap \mathopen[0,p^{N_1+N_2+N_3-1}\mathclose[^2).$$
\begin{equation}\label{eq:sigmapD1}
\begin{split}
&\sum_{\mathbf{\nu}\in\ZZ^3_{\geqslant 0}}\sum_{k\in\{0,1\}}(-1)^k \frac{\bar{\varrho}_{p^k}(p^{N_1},p^{N_2},p^{N_3})}{p^{2k}}\\
&= \sum_{\mathbf{\nu}\in\ZZ^3_{\geqslant 0}}\left(\frac{\varrho(p^{N_1},p^{N_2},p^{N_3};L_1,L_2,L_3)}{p^{2(N_1+N_2+N_3)}}-\frac{\varrho(p^{N_1},p^{N_2},p^{N_3};pL_1,pL_2,pL_3)}{p^{2(N_1+N_2+N_3+1)}}\right)\\
& =\sum_{\mathbf{\nu}\in\ZZ^3_{\geqslant 0}}\left(\frac{\varrho(p^{N_1},p^{N_2},p^{N_3};L_1,L_2,L_3)}{p^{2(N_1+N_2+N_3)}}-\frac{\sharp(\Lambda^\prime\cap \mathopen[0,p^{N_1+N_2+N_3}\mathclose[^2)}{p^{2(N_1+N_2+N_3+1)}}\right)\\
& =\sum_{\mathbf{\nu}\in\ZZ^3_{\geqslant 0}}\left(\frac{\varrho(p^{N_1},p^{N_2},p^{N_3};L_1,L_2,L_3)}{p^{2(N_1+N_2+N_3)}}-\frac{\sharp(\Lambda^\prime\cap \mathopen[0,p^{N_1+N_2+N_3-1}\mathclose[^2)}{p^{2(N_1+N_2+N_3)}}\right)\\
& =\sum_{\mathbf{\nu}\in\ZZ^3_{\geqslant 0}}\left(\frac{\sharp\{\mathbf{x}\in\mathopen[0,p^{N_1+N_2+N_3}\mathclose[^2:p^{N_i}|L_i(\mathbf{x}),\forall i\}-\sharp\{\mathbf{y}\in\mathopen[0,p^{N_1+N_2+N_3}\mathclose[^2:p|y_i,p^{N_i}|L_i(\mathbf{y}),\forall i\}}{p^{2(N_1+N_2+N_3)}}\right)\\
&=\sum_{\mathbf{\nu}\in\ZZ^3_{\geqslant 0}}\frac{\varrho^*(p^{N_1},p^{N_2},p^{N_3})}{p^{2(N_1+N_2+N_3)}}.
\end{split}
\end{equation}
Si $v_p(D)=0$, on a dans ce cas pour $\mathbf{\nu}=(\nu_1,\nu_2,\nu_3)\in\ZZ^3_{\geqslant 0},N_i=\nu_i$ et donc
$$\sum_{\mathbf{\nu}\in\ZZ^3_{\geqslant 0}}\sum_{k\in\{0,1\}}(-1)^k\frac{\bar{\varrho}_{p^k}(1,1,1)}{p^{2k}}=1-\frac{1}{p^2}+\sum_{\substack{\mathbf{\nu}\in\ZZ^3_{\geqslant 0}\\\nu_1+\nu_2+\nu_3\geqslant 1}}\frac{\varrho^*(p^{\nu_1},p^{\nu_2},p^{\nu_3})}{p^{2(\nu_1+\nu_1+\nu_3)}}.$$
On en conclut que dans ces deux cas les facteur locaux sont bien $\sigma^*_p(\mathbf{d},\mathbf{D})$.

On divise la somme \eqref{eq:sum1} en deux parties, une pour les $e$ petits (i.e. $X^{\frac{1}{12}}\leqslant e\leqslant X$) et l'autre pour les $e$ grands (i.e. $e>X^\frac{1}{12}$). 
On peut majorer le terme d'erreur apparaissant dans \eqref{eq:sum2} comme
$$\sum_{e\leqslant X^{\frac{1}{12}}}\left(\frac{D^\varepsilon}{\delta(\mathbf{D})}e X^{\frac{7}{4}+\varepsilon}+X^2(\log X)^2\frac{\log e}{e^{2-\varepsilon}}\frac{D^\varepsilon }{\operatorname{det}(\Lambda(\mathbf{D}))}\right)\ll \frac{D^\varepsilon}{\delta(\mathbf{D})}X^{\frac{23}{12}+\varepsilon}+\frac{D^\varepsilon }{\operatorname{det}(\Lambda(\mathbf{D}))}X^2(\log X)^2.$$
En outre, on a une majoration (\cite{Browning}, (2.5))
\begin{equation}\label{upperboundfors}
	S(X/e,\mathbf{d},\mathbf{D};eL_1,eL_2,eL_3)\ll e^\varepsilon \left(\frac{X}{e}\right)^{2+\varepsilon}=\frac{X^{2+\varepsilon}}{e^2}.
\end{equation}
En reportant dans \eqref{eq:sum1}, en utilisant encore \eqref{eq:upperboundCdD},
\begin{equation*}
\begin{split}
	&S^*(X,\mathbf{d},\mathbf{D})=\left(\sum_{e\leqslant X^{\frac{1}{12}}}+\sum_{X^{\frac{1}{12}}\leqslant e\leqslant X} \right)\mu(e) S(X/e,\mathbf{d},\mathbf{D};eL_1,eL_2,eL_3)\\
	&=\sum_{e\leqslant X^{\frac{1}{12}}} \frac{\mu(e)}{e^2}C(\mathbf{d},\mathbf{D},e)\operatorname{vol}(\mathcal{R})X^2(\log X)^3 +O_\varepsilon\left(\frac{D^\varepsilon}{\delta(\mathbf{D})}X^{\frac{23}{12}+\varepsilon}+\frac{D^\varepsilon }{\operatorname{det}(\Lambda(\mathbf{D}))}X^2(\log X)^2+\sum_{X^\frac{1}{12}\leqslant e\leqslant X}\frac{X^{2+\varepsilon}}{e^2}\right)\\
	&=C^*(\mathbf{d},\mathbf{D})\operatorname{vol}(\mathcal{R})X^2(\log X)^3 + O_{\varepsilon}\left(\left(\frac{D^\varepsilon}{\delta(\mathbf{D})}+1\right)X^{\frac{23}{12}+\varepsilon}+\frac{D^\varepsilon }{\operatorname{det}(\Lambda(\mathbf{D}))}\left(X^2(\log X)^2+\sum_{e\geqslant X^{\frac{1}{12}}}\frac{X^2(\log X)^3}{e^{2-\varepsilon}}\right)\right)\\
	&=C^*(\mathbf{d},\mathbf{D})\operatorname{vol}(\mathcal{R})X^2(\log X)^3 +O_{\varepsilon}\left(\left(\frac{D^\varepsilon}{\delta(\mathbf{D})}+1\right)X^{\frac{23}{12}+\varepsilon}+\frac{D^\varepsilon }{\operatorname{det}(\Lambda(\mathbf{D}))}X^2(\log X)^2\right).
	\end{split}
\end{equation*}
\end{proof}
Maintenant on est prêt à déduire la formule asymptotique pour la fonction $\Psi$. 
\begin{proof}[Démonstration de la Proposition \ref{po:problemofdivisor}]
On voit que, d'après $h=\Psi*\mu* \mu$ et $\tau=\mathfrak{1}* \mathfrak{1}$, 
$$\Psi=\Psi* \mu*\mu*\mathfrak{1}*\mathfrak{1}=h*\mathfrak{1}*\mathfrak{1}=h*\tau.$$
On rappelle les trois formes linéaires \eqref{eq:threeforms}et la région $\mathcal{R}$ \eqref{eq:regionR}.
Alors
\begin{equation}\label{eq:sumofdisprimitive}
\begin{split}
&\sum_{\substack{\mathbf{x}\in X\mathcal{R}\\ \pgcd(x_1,x_2)=1}}\Psi(x_1)\Psi(x_2)\Psi(x_2-x_1)\\
&=\sum_{\substack{\mathbf{x}\in X\mathcal{R}\\ \pgcd(x_1,x_2)=1}}\ \sum_{\substack{\mathbf{d}\in \NN^3\\ d_1|x_1,d_2|x_2,d_3|x_2-x_1}} \prod_{i=1}^{3}h(d_i)\tau\left(\frac{x_1}{d_1}\right)\tau\left(\frac{x_2}{d_2}\right)\tau\left(\frac{x_2-x_1}{d_3}\right)\\
&=\sum_{\substack{\mathbf{d}=(d_1,d_2,d_3)\in\NN_{\geqslant 1}^3\\\forall i,j\in\{1,2,3\},i\neq j,\pgcd(d_i,d_j)=1}}\left(\prod_{i=1}^{3}h(d_i)\right)\sum_{\substack{\mathbf{x}\in X\mathcal{R}\cap\Lambda(\mathbf{d})\\\pgcd(x_1,x_2)=1}}\tau\left(\frac{x_1}{d_1}\right)\tau\left(\frac{x_2}{d_2}\right)\tau\left(\frac{x_2-x_1}{d_3}\right)\\
&=\sum_{\substack{\mathbf{d}=(d_1,d_2,d_3)\in\NN_{\geqslant 1}^3\\\forall i,j\in\{1,2,3\},i\neq j,\pgcd(d_i,d_j)=1}}\left(\prod_{i=1}^{3}h(d_i)\right)S^*(X,\mathbf{d},\mathbf{d}),
\end{split}
\end{equation} 
la quantité $S^*(X,\mathbf{d},\mathbf{d})$ étant nulle s'il existe $d_i$ tel que $d_i>X$.
D'après \eqref{eq:Sstar}, on a, en notant $C^*(\mathbf{d})=C^*(\mathbf{d},\mathbf{d})$, comme $\det(\Lambda(\mathbf{d}))=d_1d_2d_3$ pour $\pgcd(d_i,d_j)=1,\forall i,j\in\{1,2,3\},i\neq j$,
\begin{align*}
S^*(X,\mathbf{d},\mathbf{d})&=C^*(\mathbf{d})\operatorname{vol}(\mathcal{R})X^2(\log X)^3+ O_\varepsilon((d_1d_2d_3)^{\varepsilon-1}X^2 (\log X)^2)+O_\varepsilon((d_1d_2d_3)^\varepsilon X^{\frac{23}{12}+\varepsilon}).
\end{align*}
Le terme principal de la somme \eqref{eq:sumofdisprimitive} est 
$$C_1\operatorname{vol}(\mathcal{R})X^2(\log X)^3,$$ où d'après \eqref{eq:cstar},
\begin{equation}\label{eq:C1p}
C_1=\sum_{\substack{\mathbf{d}\in\NN_{\geqslant 1}^3\\\pgcd(d_i,d_j)=1}}\left(\prod_{i=1}^{3}h(d_i)\right)C^*(\mathbf{d}),
\end{equation}
dont la convergence résulte de \eqref{eq:hsmall} et \eqref{eq:upperboundCdD}.
Grâce à notre choix des formes linéaires \eqref{eq:threeforms}, fixons un nombre premier $p$,
on a que pour $\sharp \{i\in\{1,2,3\}:\nu_i\geqslant 1\}\geqslant 2$,
\begin{align*}
	\varrho^*(p^{\nu_1},p^{\nu_2},p^{\nu_3})
	=\sharp (\{\mathbf{x}\in\ZZ^2: p^{\nu_1}|x_1,p^{\nu_2}|x_2,p^{\nu_3}|x_2-x_1,p\nmid\pgcd(x_1,x_2)\}\cap \mathopen[0,p^{\sum_{i=1}^{3}\nu_i}\mathclose[^2)=0.
\end{align*}
On calcule pour $\nu\geqslant 1$,
\begin{align*}
	\varrho^*(p^{\nu},1,1)&=\sharp(\{\mathbf{x}\in\ZZ^2:p^{\nu}|x_1,p\nmid\pgcd(x_1,x_2)\}\cap\mathopen[0,p^{\nu}\mathclose[^2)\\
	&=\sharp(\{\mathbf{x}=(0,x_2):p\nmid x_2\}\cap \mathopen[0,p^{\nu}\mathclose[)=\varphi(p^{\nu})=p^{\nu}-p^{\nu-1}.
\end{align*}
De façon analogue on a
$$\varrho^*(1,p^{\nu},1)=	\varrho^*(1,1,p^{\nu})=p^\nu-p^{\nu-1}.$$
On en déduit les valeurs de $\sigma_p^*$ \eqref{eq:sigmapstar} comme suit.
\begin{align*}
	\sigma_p^*((1,1,1),(1,1,1))&=\left(1-\frac{1}{p}\right)^3 \left(1-\frac{1}{p^2}+\sum_{\nu\in\NN_{\geqslant 1}}\left(\frac{\varrho^*(p^{\nu},1,1)+\varrho^*(1,p^{\nu},1)+	\varrho^*(1,1,p^{\nu})}{p^{2\nu}}\right)\right)\\
	&=\left(1-\frac{1}{p}\right)^3 \left(1-\frac{1}{p^2}+3\left(1-\frac{1}{p}\right)\sum_{\nu\in\NN_{\geqslant 1}}\frac{1}{p^\nu}\right)=\left(1-\frac{1}{p}\right)^3 \left(1-\frac{1}{p^2}+\frac{3}{p}\right).
\end{align*}
\begin{align*}
	\sigma_p^*(p,1,1)&=\sigma_p^*(1,p,1)=\sigma_p^*(1,1,p)\\&= \left(1-\frac{1}{p}\right)^3 \sum_{\nu\in\NN_{\geqslant 1}}\frac{\varrho^*(p^{\nu},1,1)+\varrho^*(1,p^{\nu},1)+	\varrho^*(1,1,p^{\nu})}{p^{2\nu}}= \left(1-\frac{1}{p}\right)^3 \times \frac{3}{p}.
\end{align*}
Donc le facteur local \eqref{eq:C1p} se calcule, en rappelant que $h(p)=-\frac{2}{p+2}$, 
\begin{equation*}
\begin{split}
	&\sum_{\substack{\mathbf{\nu}\in\NN^3\\ \sharp\{i\in\{1,2,3\}:\nu_i\geqslant 1\}\leqslant 1}} \left(\prod_{i=1}^{3}h(p^{\nu_i})\right) \sigma_p^*(p^{\nu_1},p^{\nu_2},p^{\nu_3})\\
	&=h(1)\sigma_p^*(1,1,1)+h(p)\left(\sigma_p^*(p,1,1)+\sigma_p^*(1,p,1)+\sigma_p^*(1,1,p)\right)\\
&=\left(1-\frac{1}{p}\right)^3  \left(1+\frac{3}{p}-\frac{1}{p^2}-\frac{18}{p(p+2)}\right).
\end{split}
\end{equation*}
Donc
\begin{equation}
C_1=\prod_{p}\left(1-\frac{1}{p}\right)^3\left(1+\frac{3}{p}-\frac{1}{p^2}-\frac{18}{p(p+2)}\right).
\end{equation}
Or on a aussi une majoration pour $S^*(X,\mathbf{d},\mathbf{d})$ d'après \cite[(3.1)]{Browning}:
$$S^*(X,\mathbf{d},\mathbf{d})\leqslant S(X,\mathbf{d},\mathbf{d})\ll \frac{X^{2+\varepsilon}}{d_1d_2d_3}+X^{1+\varepsilon},$$
qui découle des majorations standards pour le nombre de diviseurs et pour le dénombrement des points sur un réseau.
Soit $\delta >0$ tel que $\delta>\varepsilon $ et $\varepsilon(1+3\delta)<\frac{1}{12}$. Comme l'on trait la somme \eqref{eq:sum1} précédemment, on décompose \eqref{eq:sumofdisprimitive} en deux parties 
$$\max_{1\leqslant i\leqslant 3}(d_i)\leqslant X^\delta\text{ et }\exists i , d_i>X^\delta $$ selon la taille de $\mathbf{d}$. Cela nous permet de contrôler les termes d'erreur, compte-tenu de \eqref{eq:hsmall},
\begin{align*}
	\sum_{\mathbf{d}\in\NN_{\geqslant 1}^3} \prod_{i=1}^{3}|h(d_i)|(d_1d_2d_3)^{\varepsilon-1}X^2(\log X)^2
	\ll \sum_{\mathbf{d}\in\NN_{\geqslant 1}^3} (d_1d_2d_3)^{\varepsilon-2}X^2(\log X)^2 \ll X^2(\log X)^2.
\end{align*}
\begin{align*}
	\sum_{\substack{\mathbf{d}\in\NN_{\geqslant 1}^3\\ \max (d_i)\leqslant X^\delta}}\prod_{i=1}^{3}|h(d_i)|(d_1d_2d_3)^\varepsilon X^{\frac{23}{12}+\varepsilon}
	\ll \sum_{\substack{\mathbf{d}\in\NN_{\geqslant 1}^3\\ \max (d_i)\leqslant X^\delta}}(d_1d_2d_3)^{\varepsilon-1}X^{\frac{23}{12}+\varepsilon}
	\ll X^{\frac{23}{12}+\varepsilon(1+3\delta)}.
\end{align*}
\begin{align*}
	\sum_{\substack{\mathbf{d}\in\NN_{\geqslant 1}^3\\\max(d_i)\leqslant X\\ \exists i:d_i> X^\delta}}\prod_{i=1}^{3}|h(d_i)|\left(\frac{X^{\varepsilon+2}}{d_1d_2d_3}+X^{1+\varepsilon}\right)
	\ll \sum_{\substack{\mathbf{d}\in\NN_{\geqslant 1}^3\\ \exists d_i> X^\delta}}\frac{X^{\varepsilon+2}}{(d_1d_2d_3)^2}+\sum_{\substack{\mathbf{d}\in\NN_{\geqslant 1}^3\\\max(d_i)\leqslant X}}\frac{X^{1+\varepsilon}}{d_1d_2d_3}
	\ll X^{2+\varepsilon-\delta}+X^{1+\varepsilon}(\log X)^3.
\end{align*}
D'après le Lemme \ref{le:browning} appliqué à $C(\mathbf{d},\mathbf{d})$, on a
$$C^*(\mathbf{d},\mathbf{d})\leqslant C(\mathbf{d},\mathbf{d})\ll_\varepsilon\frac{(d_1d_2d_3 )^{\varepsilon}}{\det(\Lambda(\mathbf{d}))}=(d_1d_2d_3)^{\varepsilon-1},$$ 
puisque $\det(\Lambda(\mathbf{d}))=d_1d_2d_3$ pour $\pgcd(d_i,d_j)=1,\forall i,j\in\{1,2,3\},i\neq j$ et $L_\infty=1$, d'où
\begin{align*}
\sum_{\substack{\mathbf{d}\in\NN_{\geqslant 1}^3\\\pgcd(d_i,d_j)=1\\\max(d_i)\leqslant X^\delta}}\prod_{i=1}^{3}h(d_i) C^*(\mathbf{d})&=C_1+O\left(\sum_{\substack{\mathbf{d}\in\NN_{\geqslant 1}^3\\\pgcd(d_i,d_j)=1\\\exists i:d_i>X^\delta}}\prod_{i=1}^{3}|h(d_i)| C(\mathbf{d},\mathbf{d})\right)\\
	&=C_1+O\left(\sum_{\substack{\mathbf{d}\in\NN_{\geqslant 1}^3\\\exists i:d_i>X^\delta}}\frac{1}{(d_1d_2d_3)^{2-\varepsilon}}\right)=C_1+O(X^{\delta(\varepsilon-1)}).
\end{align*}
On en conclut que \eqref{eq:sumofdisprimitive} peut se calculer comme
\begin{equation}\label{eq:sumofpsifunction}
\begin{split}
	&\sum_{\substack{\mathbf{x}\in X\mathcal{R}\\ \pgcd(x_1,x_2)=1}}\Psi(x_1)\Psi(x_2)\Psi(x_2-x_1)\\
	&=\sum_{\substack{\mathbf{d}\in\NN_{\geqslant 1}^3\\\pgcd(d_i,d_j)=1\\\max(d_i)\leqslant X^\delta}}+\sum_{\substack{\mathbf{d}\in\NN_{\geqslant 1}^3\\\pgcd(d_i,d_j)=1\\ \exists i,d_i>X^\delta}}\left(\prod_{i=1}^{3}h(d_i)\right) S^*(X,\mathbf{d},\mathbf{d})\\
	&=C_1\operatorname{vol}(\mathcal{R})X^2(\log X)^3+O(X^2(\log X)^2 + X^{\frac{23}{12}+\varepsilon(1+3\delta)})+O(X^{2-\delta(1-\varepsilon)}(\log X)^3+X^{2-(\delta-\varepsilon)})\\
	&=C_1\operatorname{vol}(\mathcal{R})X^2(\log X)^3+O(X^2(\log X)^2).
	\end{split}
\end{equation}
Cela démontre la formule \ref{po:asymptoticforPsi} de la Proposition \ref{po:problemofdivisor} en remarquant que 
$$\operatorname{vol}(\mathcal{R})=\frac{1}{2}\left(\frac{1}{\tau_2}-\frac{1}{\tau_1}\right).$$ 

Nous démonstration la formule \ref{po:asymptoticforPsiquotient}.
La méthode ressemble à une intégration par partie et s'inspire de \cite[\S 8.3]{B-B1}. Tout d'abord on traite la somme
\begin{equation}
T(\tau_1,\tau_2)=\sum_{\substack{\mathbf{x}\in X\mathcal{R}\\\pgcd(x_1,x_2)=1}}\frac{\Psi(x_1)\Psi(x_2)\Psi(x_2-x_1)}{x_2^{\frac{3}{2}}}.
\end{equation}
Comme
$$\frac{1}{x_2^{\frac{3}{2}}}=\frac{3}{2}\int_{x_2}^{X}\frac{\operatorname{d}t}{t^{\frac{5}{2}}}+\frac{1}{X^{\frac{3}{2}}},$$
on obtient, d'après le théorème de Fubini,
\begin{align*}
	T(\tau_1,\tau_2)&=\sum_{\substack{\mathbf{x}\in X\mathcal{R}\\\pgcd(x_1,x_2)=1}}\Psi(x_1)\Psi(x_2)\Psi(x_2-x_1)\left(\frac{3}{2}\int_{x_2}^{X}\frac{\operatorname{d}t}{t^{\frac{5}{2}}}+\frac{1}{X^{\frac{3}{2}}}\right)\\&=\int_{1}^{X}\frac{3}{2}\sum_{\substack{\mathbf{x}\in t\mathcal{R}\\\pgcd(x_1,x_2)=1}}\Psi(x_1)\Psi(x_2)\Psi(x_2-x_1)\frac{\operatorname{d}t}{t^{\frac{5}{2}}}+X^{-\frac{3}{2}}\sum_{\substack{\mathbf{x}\in X\mathcal{R}\\\pgcd(x_1,x_2)=1}}\Psi(x_1)\Psi(x_2)\Psi(x_2-x_1)\\
	&=\frac{3C_1}{2}\operatorname{vol}(\mathcal{R})\int_{1}^{X}\frac{(\log t)^3}{t^{\frac{1}{2}}}\operatorname{d}t+O\left(\int_{1}^{X}\frac{(\log t)^2}{t^\frac{1}{2}}\operatorname{d}t\right)+C_1\operatorname{vol}(\mathcal{R}) X^\frac{1}{2}(\log X)^3+O(X^{\frac{1}{2}}(\log X)^2).
\end{align*}
Comme
$$\int_{1}^{X}\frac{(\log t)^3}{t^{\frac{1}{2}}}\operatorname{d}t=2X^{\frac{1}{2}}(\log X)^3+O(X^{\frac{1}{2}}(\log X)^2),\quad \int_{1}^{X}\frac{(\log t)^2}{t^\frac{1}{2}}\operatorname{d}t =O(X^\frac{1}{2}(\log X)^2),$$
l'égalité ci-dessus s'écrit
\begin{equation}\label{eq:Ttau1tau2}
T(\tau_1,\tau_2)=4C_1\operatorname{vol}(\mathcal{R})X^{\frac{1}{2}}(\log X)^3+O(X^{\frac{1}{2}}(\log X)^2).
\end{equation}
Finalement on arrive à sommer 
$$\frac{\Psi(x_1)\Psi(x_2)\Psi(x_2-x_1)}{x_2x_1^{\frac{1}{2}}}.$$
On définit $f(t)=\sqrt{t}$ et on rappelle que
$$\mathcal{R}(t,\tau_1)=\{\mathbf{x}\in\RR^2:0<x_1,x_2\leqslant 1, t\leqslant \frac{x_2}{x_1}\leqslant \tau_1\}.$$
Alors
\begin{align*}
	&\sum_{\substack{\mathbf{x}\in X\mathcal{R}\\\pgcd(x_1,x_2)=1}} \frac{\Psi(x_1)\Psi(x_2)\Psi(x_2-x_1)}{x_2x_1^\frac{1}{2}}\\
	&=\sum_{\substack{\mathbf{x}\in X\mathcal{R}\\\pgcd(x_1,x_2)=1}} \frac{\Psi(x_1)\Psi(x_2)\Psi(x_2-x_1)}{x_2^{\frac{3}{2}}} f\left(\frac{x_2}{x_1}\right)\\
	&=\sum_{\substack{\mathbf{x}\in X\mathcal{R}\\\pgcd(x_1,x_2)=1}} \frac{\Psi(x_1)\Psi(x_2)\Psi(x_2-x_1)}{x_2^{\frac{3}{2}}}\left(f(\tau_2)+\int_{\tau_2}^{\frac{x_2}{x_1}}f^\prime(t)\operatorname{d}t\right)\\
	&=f(\tau_2)T(\tau_1,\tau_2)+\sum_{\substack{\mathbf{x}\in X\mathcal{R}\\\pgcd(x_1,x_2)=1}} \frac{\Psi(x_1)\Psi(x_2)\Psi(x_2-x_1)}{x_2^{\frac{3}{2}}}\int_{\tau_2}^{\frac{x_2}{x_1}}f^\prime(t)\operatorname{d}t.
\end{align*}
Il résulte de \eqref{eq:Ttau1tau2} que
$$f(\tau_2)T(\tau_1,\tau_2)=2C_1\left(\frac{1}{\sqrt{\tau_2}}-\frac{\sqrt{\tau_2}}{\tau_1}\right)X^{\frac{1}{2}}(\log X)^3+O(X^{\frac{1}{2}} (\log X)^2).$$
Il faut traiter le deuxième terme. D'après le théorème de Fubini et \eqref{eq:Ttau1tau2} en remplaçant $\tau_2$ par $t$,
\begin{align*}
&\sum_{\substack{\mathbf{x}\in X\mathcal{R}\\\pgcd(x_1,x_2)=1}} \frac{\Psi(x_1)\Psi(x_2)\Psi(x_2-x_1)}{x_2^{\frac{3}{2}}}\int_{\tau_2}^{\frac{x_2}{x_1}}f^\prime(t)\operatorname{d}t\\
&=\int_{\tau_2}^{\tau_1}\sum_{\substack{\mathbf{x}\in X\mathcal{R}(t,\tau_1)
		\\\pgcd(x_1,x_2)=1}}\frac{\Psi(x_1)\Psi(x_2)\Psi(x_2-x_1)}{x_2^{\frac{3}{2}}} f^\prime(t)\operatorname{d}t\\
&=4C_1X^{\frac{1}{2}}(\log X)^3\int_{\tau_2}^{\tau_1}\frac{1}{2}\left(\frac{1}{t}-\frac{1}{\tau_1}\right)\frac{1}{2\sqrt{t}}\operatorname{d}t+O(X^{\frac{1}{2}}(\log X)^2)\\
&=2C_1X^{\frac{1}{2}}(\log X)^3 \left(\frac{\sqrt{\tau_2}}{\tau_1}+\frac{1}{\sqrt{\tau_2}}-\frac{2}{\sqrt{\tau_1}}\right)+O(X^{\frac{1}{2}}(\log X)^2).
\end{align*}
On obtient finalement que
\begin{equation*}
\sum_{\substack{\mathbf{x}\in X\mathcal{R}\\\pgcd(x_1,x_2)=1}} \frac{\Psi(x_1)\Psi(x_2)\Psi(x_2-x_1)}{x_2\sqrt{x_1}}=4C_1\left(\frac{1}{\sqrt{\tau_2}}-\frac{1}{\sqrt{\tau_1}}\right) X^{\frac{1}{2}}(\log X)^3+O(X^{\frac{1}{2}}(\log X)^2).\qedhere
\end{equation*}
\end{proof}

\section*{Remerciement}
Je tiens à remercier Emmanuel Peyre pour m'avoir proposé ce problème et pour des conseils tout au long de mon travail. Régis de la Bretèche, Étienne Fouvry, Marc Hindry et Florent Jouve m'ont donné des indications très précieuses. De nombreux commentaires pertinents de l'arbitre anonyme ont beaucoup éclairci la lisibilité. Le présent texte est un effort en commun de tous. Je les en remercie chaleureusement. Des parties de ce travail ont été faites à l'université Paris-Diderot et à Beijing International Center for Mathematical Research dont l'atmosphère stimulante est sincèrement appréciée. L'auteur était supporté par le projet ANR Gardio, Riemann Fellowship et le budget DE1646/4-2 Deutsche Forschungsgemeinschaft.

	 \bigskip
	
	{\bfseries Current address}: \textsc{Institut für Algebra, Zahlentheorie und Diskrete Mathematik, Leibniz Universität Hannover,
		30167 Hannover, Deutschland}\par\nopagebreak

\end{document}